\documentclass[12pt,reqno,a4paper]{amsart}            % AMS article style
\newcommand{\finalVersion}{}

\usepackage{tikz}
\usetikzlibrary{matrix, arrows, backgrounds}
\usetikzlibrary{cd}

\usepackage{hyperref}

\usepackage{scrlayer-scrpage}    % change headings, see Franzis book, p. 500)
\usepackage{scrtime}     % give time by \thistime

\usepackage{amsmath, amssymb}%, theorem}
\usepackage{amsthm}

\usepackage{accents}     % for \accentset, \ring...
\usepackage{mathtools}  % for \coloneqq etc.

\usepackage{bbm}         % blackboard (see: /usr/share/texmf/mytex/bbm.dvi)

\usepackage{mathrsfs}    % use ``real'' caligraphic letters for \mathcal
\usepackage{graphicx}                % for graphics
\usepackage{iftex}

\usepackage[totalwidth=475pt, totalheight=735pt]{geometry}

\usepackage{scrtime}     % give time by \thistime

\usepackage{ifthen}       % used for ``draft'' mode

\ifthenelse{\isundefined \finalVersion}
{ %% Draft mode: show remarks \look etc.
  \rehead{\texttt{\jobname.tex}, \today, \thistime}
  \lohead{\texttt{\jobname.tex}, \today, \thistime}

  \newcommand{\marker}[1]{\fbox{\rule{0pt}{0.1ex}\textbf{#1}}}
  \newcommand{\markerO}{\marker{Olaf}}
  \newcommand{\markerS}{\marker{Sebastian}}

  \newcounter{lookcounter}
  \setcounter{lookcounter}{0}
  \newcommand{\look}[2][$\diamond$]
             {
               \stepcounter{lookcounter}
               \marker{#1}
               \footnote[\arabic{lookcounter}]{\marker{$\diamond$} #2}
             }
  \newcommand{\lookO}[1]
             {
               \stepcounter{lookcounter}
               \markerO
               \footnote[\arabic{lookcounter}]{\markerO #1}
             }
  \newcommand{\lookS}[1]
             {
               \stepcounter{lookcounter}
               \markerS
               \footnote[\arabic{lookcounter}]{\markerS #1}
             }

  \ifthenelse{\isundefined \OlafsVersion}
  {\usepackage[notref,notcite]{showkeys}}              % on other computers
  {\usepackage{/home/post/Aktuell/LaTeX/my-showkeys}}  % on Olaf's computer
}
{ %% final mode: don't show remarks --- use standard headers
  \lohead{\textsf{\shorttitle}}
  \rehead{\textsf{\authors}}
  \automark[subsection]{section}
  \automark[subsection]{section}

  \newcommand{\look}[1]{}%
  \newcommand{\lookO}[1]{}%
  \newcommand{\lookS}[1]{}%
}

\newcommand{\comment}[1]{}   % put old stuff here (for debugging etc.)

\usepackage{graphicx}                % for graphics
\usepackage{amsmath, amssymb}%, theorem}
\usepackage{amsthm}

\usepackage{amscd}

\usepackage{accents}     % for \accentset, \ring...
\usepackage{mathtools}  % for \coloneqq etc.

\usepackage{bbm}         % blackboard (see: /usr/share/texmf/mytex/bbm.dvi)

\usepackage{mathrsfs}    % use ``real'' caligraphic letters for \mathcal

\usepackage{xspace}  % for adding space after macros automatically

\numberwithin{equation}{section}

\newcounter{myenumi}

\newcommand{\itemref}[1]{\noindent\eqref{#1}}

\newcommand{\myfont}{\sffamily}
\newcommand{\myparagraph}[1]{\noindent\textbf{\myfont{#1}}}

\newtheoremstyle{mythmstyle}% name
  {\topsep}% Space above1
  {\topsep}% Space below
  {\itshape}% Body font
  {}% Indent amount
  {\bfseries \myfont}% Theorem head font
  {.}%Punctuation after theorem head
  {.5em}%Space after theorem head
  {}% theorem head spec

\newtheoremstyle{mydefstyle}% name
  {\topsep}% Space above
  {\topsep}% Space below
  {\normalfont}% Body font
  {}% Indent amount
  {\bfseries \myfont}% Theorem head font
  {.}%Punctuation after theorem head
  {.5em}%Space after theorem head
  {}% theorem head spec

\theoremstyle{mythmstyle}       % my style (new fonts) -- body italics

\newcounter{intro}

\newtheorem{maintheorem}[intro]{Theorem}
\newtheorem{maincorollary}[intro]{Corollary}

\swapnumbers                    % 1.1 Theorem instead of Theorem 1.1

\newtheorem{theorem}{Theorem}[section]
\newtheorem{proposition}[theorem]{Proposition}
\newtheorem{lemma}[theorem]{Lemma}
\newtheorem{corollary}[theorem]{Corollary}

\theoremstyle{mydefstyle}        % my style (new fonts) -- body roman
\newtheorem{definition}[theorem]{Definition}

\newtheorem{example}[theorem]{Example}

\newtheorem{remark}[theorem]{Remark}
\newtheorem{remarks}[theorem]{Remarks}
\newtheorem*{remark*}{Remark}

\expandafter\let\expandafter\oldproof\csname\string\proof\endcsname
\let\oldendproof\endproof
\renewenvironment{proof}[1][\bfseries\myfont\proofname]{%
  \oldproof[\bfseries \myfont #1]%
}{\oldendproof}

\makeatletter
\renewcommand\section{\@startsection{section}{1}%
  \z@{.7\linespacing\@plus\linespacing}{.5\linespacing}%
  {\Large\myfont\bfseries}}
\renewcommand\subsection{\@startsection{subsection}{2}%
  \z@{-.5\linespacing\@plus-.7\linespacing}{.5\linespacing}%
  {\large\myfont\bfseries}}
\renewcommand\subsubsection{\@startsection{subsubsection}{3}%
  \z@{.5\linespacing\@plus.7\linespacing}{-.5em}%
  {\myfont\bfseries}}
\renewenvironment{abstract}{%
  \ifx\maketitle\relax
    \ClassWarning{\@classname}{Abstract should precede
      \protect\maketitle\space in AMS document classes; reported}%
  \fi
  \global\setbox\abstractbox=\vtop \bgroup
    \normalfont\Small
    \list{}{\labelwidth\z@
      \leftmargin3pc \rightmargin\leftmargin
      \listparindent\normalparindent \itemindent\z@
      \parsep\z@ \@plus\p@
      
    }%
    \item[\hskip\labelsep%\scshape
      \myfont\bfseries
    \abstractname.]%
}{%
  \endlist\egroup
  \ifx\@setabstract\relax \@setabstracta \fi
}
\renewcommand\contentsnamefont{\myfont\bfseries}%{\scshape}
\renewcommand\@starttoc[2]{\begingroup
  \setTrue{#1}%
  \par\removelastskip\vskip\z@skip
  \@startsection{}\@M\z@{\linespacing\@plus\linespacing}%
    {.5\linespacing}{%\centering
      \contentsnamefont}{#2}%
  \ifx\contentsname#2%
  \else \addcontentsline{toc}{section}{#2}\fi
  \makeatletter
  \@input{\jobname.#1}%
  \if@filesw
    \@xp\newwrite\csname tf@#1\endcsname
    \immediate\@xp\openout\csname tf@#1\endcsname \jobname.#1\relax
  \fi
  \global\@nobreakfalse \endgroup
  \addvspace{32\p@\@plus14\p@}%
  \let\tableofcontents\relax
}
\renewcommand\@settitle{\begin{center}%
  \baselineskip14\p@\relax
    \LARGE
    \bfseries
    \myfont
  \@title
  \end{center}%
}
\renewcommand\@setauthors{%
  \begingroup
  \def\thanks{\protect\thanks@warning}%
  \trivlist
  \centering\footnotesize \@topsep30\p@\relax
  \advance\@topsep by -\baselineskip
  \item\relax
  \author@andify\authors
  \def\\{\protect\linebreak}%
  \large
  \myfont\bfseries\authors
  \ifx\@empty\contribs
  \else
    ,\penalty-3 \space \@setcontribs
    \@closetoccontribs
  \fi
  \endtrivlist
  \normalfont\myfont\@setaddresses
  \endgroup
}
\renewcommand\@setaddresses{\par
  \nobreak \begingroup
\footnotesize
  \def\author##1{\nobreak\addvspace\bigskipamount}%
  \def\\{\unskip, \ignorespaces}%
  \interlinepenalty\@M
  \def\address##1##2{\begingroup
    \par\addvspace\bigskipamount\indent
    \@ifnotempty{##1}{(\ignorespaces##1\unskip) }%
    {%\scshape
      \ignorespaces##2}\par\endgroup}%
  \def\curraddr##1##2{\begingroup
    \@ifnotempty{##2}{\nobreak\indent\curraddrname
      \@ifnotempty{##1}{, \ignorespaces##1\unskip}\/:\space
      ##2\par}\endgroup}%
  \def\email##1##2{\begingroup
    \@ifnotempty{##2}{\nobreak\indent\emailaddrname
      \@ifnotempty{##1}{, \ignorespaces##1\unskip}\/:\space
      \ttfamily##2\par}\endgroup}%
  \def\urladdr##1##2{\begingroup
    \def~{\char`\~}%
    \@ifnotempty{##2}{\nobreak\indent\urladdrname
      \@ifnotempty{##1}{, \ignorespaces##1\unskip}\/:\space
      \ttfamily##2\par}\endgroup}%
  \addresses
  \endgroup
}
\renewcommand\enddoc@text{\ifx\@empty\@translators \else\@settranslators\fi
}
\renewcommand\@secnumfont{\myfont\bfseries} %{\mdseries}

\renewcommand\maketitle{\par
  \@topnum\z@ % this prevents figures from falling at the top of page 1
  \@setcopyright
  \thispagestyle{firstpage}% this sets first page specifications
  \ifx\@empty\shortauthors \let\shortauthors\shorttitle
  \else \andify\shortauthors
  \fi
  \@maketitle@hook
  \begingroup
  \@maketitle
  \toks@\@xp{\shortauthors}\@temptokena\@xp{\shorttitle}%
  \toks4{\def\\{ \ignorespaces}}% defend against questionable usage
  \edef\@tempa{%
    \@nx\markboth{\the\toks4
      \@nx\MakeUppercase{\the\toks@}}{\the\@temptokena}}%
  \@tempa
  \endgroup
  \c@footnote\z@
  \@cleartopmattertags
}
\makeatother
\newcommand{\Sec}[1]{Section~\ref{sec:#1}}

\newcommand{\Subsec}[1]{Subsection~\ref{ssec:#1}}

\newcommand{\Thm}[1]{Theorem~\ref{thm:#1}}

\newcommand{\ThmS}[2]{Theorems~\ref{thm:#1}--\ref{thm:#2}}

\newcommand{\Ex}[1]{Example~\ref{ex:#1}}

\newcommand{\ExS}[2]{Examples~\ref{ex:#1}--\ref{ex:#2}}

\newcommand{\Lem}[1]{Lemma~\ref{lem:#1}}

\newcommand{\Lemenum}[2]{Lemma~\ref{lem:#1}~(\ref{#2})}

\newcommand{\Cor}[1]{Corollary~\ref{cor:#1}}

\newcommand{\Prp}[1]{Proposition~\ref{prp:#1}}
\newcommand{\Prps}[2]{Propositions~\ref{prp:#1} and~\ref{prp:#2}}

\newcommand{\Prpenum}[2]{Proposition~\ref{prp:#1}~(\ref{#2})}

\newcommand{\Rem}[1]{Remark~\ref{rem:#1}}

\newcommand{\Def}[1]{Definition~\ref{def:#1}}
\newcommand{\Defs}[2]{Definitions~\ref{def:#1} and~\ref{def:#2}}

\newcommand{\Defenum}[2]{Definition~\ref{def:#1}~(\ref{#2})}

\newcommand{\abs}[2][{}]{\lvert{#2}\rvert_{{#1}}}    % abs value
\newcommand{\abssqr}[2][{}]{\lvert{#2}\rvert^2_{#1}} % abs squared
\newcommand{\bigabs}[2][{}]{\bigl\lvert{#2}\bigr\rvert_{#1}}     % abs
\newcommand{\normsymb}{\|}
\newcommand{\bignormsymb}[1]{#1\|}

\newcommand{\norm}[2][{}]{\normsymb{#2}\normsymb_{{#1}}}    % norm
\newcommand{\normsqr}[2][{}]{\normsymb{#2}\normsymb^2_{#1}} % norm squared
\newcommand{\bignorm}[2][{}]{\bignormsymb{\bigl}{#2}\bignormsymb{\bigr}_{#1}}
\newcommand{\Bignorm}[2][{}]{\bignormsymb{\Bigl}{#2}\bignormsymb{\Bigr}_{#1}}
\newcommand{\iprod}[3][{}]{\langle{#2},{#3}\rangle_{#1}}  % inner product

\newcommand{\set}[2]{\{ \, #1 \, | \, #2 \, \} }      % set { #1 | #2 }
\newcommand{\bigset}[2]{\bigl\{ \, #1 \, \bigl|\bigr. \, #2 \, \bigr\} }
\newcommand{\Bigset}[2]{\Bigl\{ \, #1 \, \Bigl|\Bigr. \, #2 \, \Bigr\} }

\DeclareMathOperator*{\bigdcup}{\mathaccent\cdot{\bigcup}}
\DeclareMathOperator*{\dcup}   {\mathaccent\cdot\cup}

\newcommand{\map}[3]{ #1 \colon #2 \longrightarrow #3}    % maps
\newcommand{\restr}[1]{{\restriction}_{#1}} % symbol for map restriction

\newcommand{\card}[1]{\lvert#1\rvert}   % from AMS proceedings file
\newcommand{\dd}    {\, \mathrm d}    % not optimal: no \, if at beginning

\DeclareMathOperator{\dom}    {dom}
\DeclareMathOperator{\ran}    {ran}
\DeclareMathOperator{\id}     {id}   % identity map
\DeclareMathOperator{\supp}   {supp}

\newcommand{\specsymb} {\sigma} % symbol for spectrum

\newcommand{\spec}[2][{}]   {\specsymb_{\mathrm{#1}}(#2)}

\newcommand{\essspec}[1]{\spec[ess] {#1}}
\newcommand{\disspec}[1]{\spec[disc]{#1}}

\newcommand{\eps}{\varepsilon} % shortcut
\renewcommand{\phi}{\varphi}   % shortcut
\renewcommand{\rho}{\varrho}   % shortcut
\newcommand{\Q}{\mathbb{Q}} % symbol for rational numbers
\newcommand{\R}{\mathbb{R}} % symbol for real numbers
\newcommand{\C}{\mathbb{C}} % symbol for complex numbers
\newcommand{\N}{\mathbb{N}} % symbol for natural numbers
\newcommand{\1}{\mathbbm 1}                    % blackboard 1

\newcommand{\wt}{\widetilde}           % shortcut
\newcommand {\qf}[1]{\mathfrak{#1}}    % font for quadratic forms

\newcommand{\HS}{\mathscr H}           % symbol for Hilbert space
\newcommand{\HSaux}{\mathscr G}        % symbol for auxiliary Hilbert space
\newcommand{\Lsymb}    {\mathsf L}     % symbol for int L-spaces
\newcommand{\Lpspace}[1][p]    {\Lsymb_{#1}}     % symbol for int L-spaces
\newcommand{\Lsqrspace}    {\Lpspace[2]}     % symbol for int L-spaces
\newcommand{\Unitary}[1]{\mathcal U({#1})}
\newcommand{\Lsqr}[2][{}]{\Lsqrspace^{#1}({#2})} % L_2^{#1}(#2)-spaces
\newcommand{\Err}{\mathrm O}

\newcommand{\quadtext}[1]{\quad\text{#1}\quad}
\newcommand{\qquadtext}[1]{\qquad\text{#1}\qquad}

\newcommand{\compl}[1]{{#1}^{\mathrm c}}

\providecommand{\myfont}{}
\providecommand{\look}[1]{}
\providecommand{\lookO}[1]{}
\providecommand{\lookS}[1]{}

\newcommand{\LevyProkhorov}{L\'evy-Prokhorov\xspace}
\newcommand{\NbarO}{\bar \N_0}
\newcommand{\Proj}[1]{\ProjOpsymb(#1)}
\newcommand{\Borel}{\mathcal B}   % symbol for set of Borel subsets
\newcommand{\Top}{\mathcal O}     % symbol for set of open subsets
\newcommand{\OpInt}{\ring{\mathcal I}}     % symbol for set of open intervals
\newcommand{\Comp}{\mathcal K}     % symbol for set of compact sets
\newcommand{\Fin}{\mathcal F}     % symbol for set of finite subsets

\newcommand{\openBall}{\ring B}
\newcommand{\closedBall}{B}
\renewcommand{\implies}{\Rightarrow}
\renewcommand{\iff}{\Leftrightarrow}

\newcommand{\D}{A}        % Name for operator
\newcommand{\HSgen}{\HS}  % notation for common Hilbert space (formerly \HS_0)
\newcommand{\Nbar}{\overline \N} % notation for {1,2,3,...,\infty}

\newcommand{\gnrc}{\overset{\mathrm{gnrc}}%
    {\underset{\mathrm{QUE}}{\longrightarrow}}}

\newcommand{\gnrcW}{\overset{\mathrm{gnrc}}%
  {\underset{\mathrm{Weid}}{\longrightarrow}}}

\newcommand{\gnrcUni}{\overset{\mathrm{gnrc}}%
  {\underset{\mathrm{uni}}{\longrightarrow}}}

\providecommand{\BdOpsymb}{\mathsf{Lin}}
\newcommand{\UnitOpsymb}{\mathsf{Uni}}
\newcommand{\ProjOpsymb}{\mathsf{Proj}}
\newcommand{\IsoOpsymb}{\mathsf{Iso}}

\newcommand{\Lin}[2][{}]{\BdOpsymb_{#1}(#2)}
\renewcommand{\Unitary}[1]{\UnitOpsymb(#1)}
\DeclareMathOperator{\IsoX}{\IsoOpsymb}
\newcommand{\Iso}[2][{}]{\IsoX_{{#1}}(#2)}

\newcommand{\UnitOrbit}[1]{\UnitOpsymb_*(#1)}

\newcommand{\setOp}{\BdOpsymb} % symbol for set of bounded operators

\newcommand{\dIso}{\operatorname d_{\text{\normalfont iso}}} % symbol for metric
\newcommand{\bardIso}{\bar {\operatorname d}_{\text{\normalfont iso}}} % symbol for metric

\newcommand{\dUni}{\operatorname d_{\text{\normalfont uni}}}   % symbol for metric
\newcommand{\bardUni}{\bar {\operatorname d}_{\text{\normalfont uni}}}   % symbol for metric
\newcommand{\dCmf}{\operatorname d_{\text{\normalfont cmf}}}   % symbol for metric
\newcommand{\dCmffin}{\delta_{\text{\normalfont fin}}}  % symbol for finite set distance
\newcommand{\dCmffiN}{\acute \delta_{\text{\normalfont fin}}}  % symbol for finite set distance
\newcommand{\ess}{\text{\normalfont ess}}
\newcommand{\disc}{\text{\normalfont disc}}
\newcommand{\dDisc}{\operatorname d_\disc}   % symbol for metric

\newcommand{\dHaus}{\operatorname d_{\text{\normalfont Hausd}}}   % symbol for metric
\newcommand{\dHauS}{\acute {\operatorname d}_{\text{\normalfont Hausd}}}   % symbol for metric

\newcommand{\dSpec}{\operatorname d_{\text{\normalfont spec}}}   % symbol for metric

\newcommand{\dGH}{\operatorname d_{\text{\normalfont GH}}}   % symbol for GH metric
\newcommand{\dQUE}{\operatorname d_{\text{\normalfont que}}} % symbol for metric
\newcommand{\bardQUE}{\bar {\operatorname d}_{\text{\normalfont que}}} % symbol for metric
\newcommand{\dQUEvar}{\wt {\operatorname d}_{\text{\normalfont que}}} % symbol for metric
\DeclareMathOperator{\Specsymb}{\sigma}
\DeclareMathOperator{\MYsupp}{supp}  %(want to use \renewcommand{\supp}... later )
\newcommand{\NEWsupp}[2][{}]   {\MYsupp_{#1}#2}
\renewcommand{\supp}{\NEWsupp}
\newcommand{\esssupp}{\supp[\ess]}
\newcommand{\dissupp}{\supp[\disc]}

\DeclareMathOperator{\graph}{\mathsf{gr}}

\DeclareMathOperator{\dist}{dist}
\renewcommand{\specsymb}{\Specsymb}

\renewcommand{\card}[1]{\operatorname \# #1} %

\DeclareMathOperator{\rank}{rank}
\DeclareSymbolFont{bbold}{U}{bbold}{m}{n}
\DeclareSymbolFontAlphabet{\mathbbold}{bbold}
\DeclareMathSymbol{\bblambda}{\mathord}{bbold}{"15}
\ifTUTeX
  \usepackage{fontspec}
\else
  \usepackage[T1]{fontenc}
  \usepackage[utf8]{inputenc} % The default since 2018
  \DeclareUnicodeCharacter{200B}{{\hskip 0pt}}
\fi

\begin{document}

\title[]%
{Distances between operators acting on different Hilbert spaces}

\author{Olaf Post}
\address{Fachbereich 4 -- Mathematik,
  Universit\"at Trier,
  54286 Trier, Germany}
\email{olaf.post@uni-trier.de}

\author{Sebastian Zimmer}%
\address{Fachbereich 4 -- Mathematik,
  Universit\"at Trier,
  54286 Trier, Germany}
\email{zimmerse@uni-trier.de}

\ifthenelse{\isundefined \finalVersion} %
{\date{\today, \thistime,  \emph{File:} \texttt{\jobname.tex}}}
% draft version
{\date{\today}}  % final version

%\vspace{3cm}
%-------------------------------------------------------------
% Abstract.
%-------------------------------------------------------------
\begin{abstract}
  The aim of this article is to define and compare several distances
  (or metrics) between operators acting on different (separable)
  Hilbert spaces.  We consider here three main cases of how to measure
  the distance between two bounded operators: first by taking the
  distance between their unitary orbits, second by isometric
  embeddings (this generalises a concept of Weidmann) and third by
  quasi-unitary equivalence (using a concept of the first author of
  the present article).

  Our main result is that the unitary and isometric distances are
  equal provided the operators are both self-adjoint and have $0$ in
  their essential spectra.  Moreover, the quasi-unitary distance is
  equivalent (up to a universal constant) with the isometric distance
  for any pair of bounded operators.  The unitary distance gives an
  upper bound on the Hausdorff distance of their spectrum.  If both
  operators have purely essential spectrum, then the unitary distance
  equals the Hausdorff distance of their spectra.  Using a finer
  spectral distance respecting multiplicity of discrete eigenvalues,
  this spectral distance equals the unitary distance also for
  operators with essential and discrete spectrum.  In particular, all
  operator distances mentioned above are equal to this spectral
  distance resp.\ controlled by it in the quasi-unitary case for
  self-adjoint operators with $0$ in the essential spectrum.  We also
  show that our results are sharp by presenting various
  (counter-)examples.  Finally, we discuss related convergence
  concepts complementing results from our first
  article~\cite{post-zimmer:22}.
\end{abstract}

%\subjclass[2020]{Primary 47A55; Secondary 47A58, 47A10, 47B02}

\maketitle

%\tableofcontents

%\keywords{ }

%%%%%%%%%%%%%%%%%%%%%%%%%%%%%%%%%%%%

%-----------------------------------------------------------------------
%
% 1111
\section{Introduction}
\label{sec:intro}
%
%-----------------------------------------------------------------------

%-----------------------------------------------------------------------
\subsection{Distances for operators acting in different spaces}
\label{ssec:dist.op}
%-----------------------------------------------------------------------
Two bounded operators acting in a common complex Hilbert space can be
compared by taking the operator norm of their difference.  It is then
known that this norm is an upper bound on the Hausdorff distance of
their spectra (see for example~\eqref{eq:herbst-nakamura} taken
from~\cite[App.~A]{herbst-nakamura:99}).  In particular, norm
convergence of operators implies convergence of their spectra with the
same convergence speed or possibly faster.

In applications, it often happens that the underlying spaces also vary
with the convergence parameter, e.g.\ when considering
\begin{itemize}
\item domain perturbations (see
  e.g.~\cite{daners:08,arrieta-lamberti:17,anne-post:21} and
  references therein);
\item limits of spaces where the dimension shrinks, such as operators
  acting on thin neighbourhoods of embedded metric graphs, so-called
  \emph{graph-like manifolds} or \emph{thick graphs}, see
  e.g.~\cite{post:12,post-simmer:20b} and references therein);

\item discrete Laplacians as approximations of Laplacians on metric
  spaces such as the Laplacian on the Sierpi\'nski gasket or other
  post-critically finite fractals (see e.g.~\cite{post-simmer:18}
  or~\cite[Sec.~2]{post-simmer:19} and references therein.  In this
  case, typically, one space is finite dimensional while the other is
  not.
\end{itemize}
The aim of the present paper is to define different distance and
convergence notions for such problems.  We use the notion ``distance''
here, as $\dQUE$ defined below is not a metric in the proper sense (we
show that the triangle inequality for $\dQUE$ is fulfilled only up to
a constant, see \Lem{dque.metric}).

For $n=1,2$, let $\HS_n$ be a separable Hilbert space.  The Banach
space of bounded linear operators from $\HS_1$ into $\HS_2$ is denoted
by $\Lin{\HS_1,\HS_2}$.  Moreover, we set
$\Lin{\HS_n}:=\Lin{\HS_n,\HS_n}$.  In the sequel, we consider two
operators $R_n \in \Lin {\HS_n}$.

%-------------------------------------------------------------------------------
\subsubsection*{Unitary distance}
%-------------------------------------------------------------------------------
A natural approach to compare operators acting in varying Hilbert
spaces is to use \emph{unitary equivalence}.  For
$R_n \in \Lin{\HS_n}$ ($n=1,2$) we define their \emph{unitary
  distance} as
\begin{equation}
  \label{eq:uni-dist}
  \dUni(R_1,R_2)
  := \inf \bigset{\norm[\Lin {\HS_2}]{U R_1U^*-R_2}}
  {U \in \Unitary{\HS_1,\HS_2}},
\end{equation}
where $\Unitary{\HS_1,\HS_2}$ is the set of all unitary operators
$\map{U}{\HS_1}{\HS_2}$ and
$\norm[\Lin
{\HS_2}]{R_2}=\sup\set{\norm[\HS_2]{R_2u_2}}{\norm[\HS_2]{u_2}\le 1}$
is the operator norm.  In this approach one necessarily needs that
$\HS_1$ and $\HS_2$ have the same dimension (i.e.\ the cardinality of
an orthonormal Hilbert space basis), otherwise we set
$\dUni(R_1,R_2)=\infty$.  In particular, this distance is not
applicable to the case of finite-dimensional approximations of
self-adjoint operators on infinite-dimensional Hilbert spaces.

%-------------------------------------------------------------------------------
\subsubsection*{Crude multiplicity functions}
%-------------------------------------------------------------------------------
A useful tool in understanding the unitary distance is the \emph{crude
  multiplicity function} $\alpha_{R_n}$ defined for a self-adjoint
operator $R_n$ by
\begin{equation}
  \label{eq:def.cmf}
  \alpha_{R_n}(\lambda)
  := \lim_{r \to 0} \rank \1_{(\lambda-r,\lambda+r)}(R_n)
  \in \NbarO := \{0,1,2,\dots\} \cup \{\infty\}.
\end{equation}
Here, $\1_B(R_n)$ denotes the spectral projection of the self-adjoint
operator $R_n$ onto the Borel set $B \subset \R$.  Crude multiplicity
functions were first used implicitly by Gellar and
Page~\cite{gellar-page:74}.  Later, Azoff and Davis
in~\cite{azoff-davis:84} named them ``crude multiplicity functions'',
defined a \LevyProkhorov distance
$\delta(\alpha_{R_1},\alpha_{R_2})$ and showed that this distance
equals $\dUni(R_1,R_2)$.  In particular, the crude multiplicity
function $\alpha_{R_n}$ is a complete invariant of the (operator norm)
\emph{closure} of the unitary orbit
$\UnitOrbit {R_n}=\set{U_n R_n U_n^*}{U \in \Unitary {\HS_n}}$ .  We
explain the concept of crude multiplicity functions and related
distances in \Subsec{cmf}. %{unitar.dist}.

%-------------------------------------------------------------------------------
\subsubsection*{Some known results on the unitary distance and its
  relation to spectra}
%-------------------------------------------------------------------------------
We compare now the unitary distance with some distances between the
spectra: For the essential spectrum
$\Sigma_{n,\mathrm{ess}}=\essspec{R_n}$ (see \Sec{spec.thm} for a
reminder), we use the Hausdorff distance
\begin{align}
  \label{eq:d.ess}
  \dHaus(\Sigma_{1,\ess},\Sigma_{2,\ess})
  :=& \inf \set{\eps>0}{\Sigma_{1,\ess} \subset \closedBall_\eps(\Sigma_{2,\ess})
     \text{ and }
     \Sigma_{2,\ess} \subset \closedBall_\eps(\Sigma_{1,\ess})}
\end{align}
where $\closedBall_\eps(\Sigma)$ is the (closed) $\eps$-neighbourhood
of $\Sigma$, see~\eqref{eq:nbd.ball} for more details on the Hausdorff
distance see \Subsec{gh}.

The Hausdorff distance is too simple for the discrete spectrum, as it
does not see multiplicities: operators with spectrum $1,1,2$ and
$1,2,2$ (with multiplicity) will have Hausdorff distance $0$, as the
underlying set in both cases is $\{1,2\}$.  Therefore, for
the discrete spectrum $\Sigma_{n,\disc}=\disspec{R_n}$ we define
\begin{align}
  \nonumber
  \dDisc(R_1,R_2)
  :=& \inf\bigset{\eps>0}{\forall F \in \Fin(\disspec{R_1})\colon
    \rank \1_F(R_1) \le \rank \1_{\closedBall_\eps(F)}(R_2)
    \text{ and }
  \\
  \label{eq:dmult}
    & \hspace{0.1\textwidth}
      \forall F \in \Fin(\disspec{R_2}) \colon
    \rank \1_F(R_2) \le \rank \1_{\closedBall_\eps(F)}(R_1)},
\end{align}
taking multiplicities into account.  Here, $\Fin(\Sigma)$ denotes the
set of all finite subsets of $\Sigma \subset \R$.  For short, we set
\begin{equation}
  \label{eq:dspec}
  \dSpec(R_1,R_2)
  :=     \max\{\dHaus(\essspec{R_1},\essspec{R_2}), \dDisc(R_1,R_2)\}.
\end{equation}

We start with \Thm{main-1}, a result already proven
in~\cite{davidson:86}.
%----------------------------------------------------------------------------
\begin{maintheorem}[{see also~\cite[Prp.~1.3]{davidson:86}}]
  \label{thm:main-1}
  \begin{subequations}
  Assume that $R_n$ are self-adjoint and bounded operators in a
  separable Hilbert space $\HS_n$ for $n=1,2$,
  % and containing $0$ (i.e.,~\eqref{eq:cond.ess.spec} holds)
  then we have
  \begin{equation}
    \label{eq:dhaus.duni}
    \dHaus(\spec{R_1},\spec{R_2})
    \le \dUni(R_1,R_2)
%    = \dIso(R_1,R_2).
  \end{equation}
  and
  \begin{equation}
    \label{eq:dspec.duni}
    \dSpec(R_1,R_2)
    = \dUni(R_1,R_2)
  \end{equation}
  In particular, if $R_1$ and $R_2$ both have purely essential
  spectrum then
  \begin{equation}
    \label{eq:dhaus.duni'}
    \dHaus(\spec{R_1},\spec{R_2})
    = \dUni(R_1,R_2)
  \end{equation}
\end{subequations}
\end{maintheorem}
%----------------------------------------------------------------------------
The inequality~\eqref{eq:dhaus.duni} follows actually from the
classical result
\begin{equation}
  \label{eq:herbst-nakamura}
  \dHaus(\spec{R_1},\spec{R_2})
  \le \norm[\Lin \HS]{R_1-R_2}
\end{equation}
for operators $R_1,R_2$ acting in the \emph{same} Hilbert space $\HS$,
proven e.g.\ in~\cite[Lem.~A.1]{herbst-nakamura:99}.  We give an
alternative proof here in terms of abstract crude multiplicity
functions not directly referring to operators; also
because~\cite{davidson:86} uses a slightly different version of the
\LevyProkhorov distance $\delta(\alpha_{R_1},\alpha_{R_2})$.  We
discuss various distances between crude multiplicity functions in
\Subsec{cmf}.  We address the question of discrete spectrum and other
``finer'' versions of a Hausdorff distance respecting multiplicities
also in our forthcoming article~\cite{post-zimmer:pre24b} (see
also~\cite{zimmer:24}).

%-------------------------------------------------------------------------------
\subsubsection*{Isometric distance}
%-------------------------------------------------------------------------------
A less restrictive approach to define a distance for operators acting
in different Hilbert spaces is inspired by Weidmann's concept of a
generalisation of norm resolvent convergence
in~\cite[Sec.~9.3]{weidmann:00}.  We embed both spaces $\HS_1$ and
$\HS_2$ isometrically into a common Hilbert space $\HS$.  The
\emph{isometric distance} of $R_1$ and $R_2$ is then defined by
\begin{subequations}
\begin{equation}
  \label{eq:iso-dist}
  \dIso(R_1,R_2)
  =\inf\bigset{\norm[\Lin \HS]{D_{\HS, \iota_1, \iota_2}(R_1,R_2)}}
  {\iota_n\in \Iso{\HS_n,\HS}, n=1,2},
\end{equation}
where the infimum is taken over all Hilbert spaces $\HS$ and
isometries $\map{\iota_n}{\HS_n}\HS$.  Here
\begin{equation}
  \label{eq:def.diff.op}
  D_{\HS, \iota_1, \iota_2}(R_1,R_2)
  :=\iota_1 R_1 \iota_1^*-\iota_2 R_2 \iota_2^*.
\end{equation}
\end{subequations}
Note that in the isometric setting, different dimensions of $\HS_1$
and $\HS_2$ are allowed; in contrast to the unitary distance.

The definition of $\dIso$ is inspired by the Gromov-Hausdorff distance
between two arbitrary metric spaces, where the infimum of all possible
isometric embeddings into a third metric space is taken, see
\Rem{gh-dist}.

Our first main result is now equality of both distances under a mild
assumption:
%----------------------------------------------------------------------------
\begin{maintheorem}
  \label{thm:main-2}
  Assume that $R_n$ are self-adjoint and bounded operators in a
  separable Hilbert space $\HS_n$ for $n=1,2$.  If
  \begin{subequations}
  \begin{equation}
    \label{eq:cond.ess.spec}
    0 \in \essspec{R_1} \cap \essspec{R_2}
  \end{equation}
  then we have
  \begin{equation}
    \label{eq:duni.eq.diso}
    \dUni(R_1,R_2) = \dIso(R_1,R_2).
  \end{equation}
\end{subequations}
\end{maintheorem}
%----------------------------------------------------------------------------
Note that the inequality $\ge$ in~\eqref{eq:duni.eq.diso} is trivial
as $\iota_1=U_{12}$ and $\iota_2=\id_{\HS_2}$ are isometries into
$\HS=\HS_2$.

The condition~\eqref{eq:cond.ess.spec} implies that both Hilbert
spaces $\HS_1$ and $\HS_2$ are infinite-dimensional.  Moreover, we
cannot drop this condition as counterexamples show (see
\ExS{duni.eq.diso.counterex}{diso.0.duni.not}).  As in most examples
one has unbounded operators such as various Laplacians
$\Delta_n \ge 0$, we consider convergence of their resolvents
$R_n=(\Delta_n+1)^{-1}$.  Note that in this
case,~\eqref{eq:cond.ess.spec} is always fulfilled for the resolvents.

For the proof of the non-trivial inequality $\le$
in~\eqref{eq:duni.eq.diso} we use crude multiplicity functions, namely
that $\alpha_{\iota_n R_n \iota_n^*}=\alpha_{R_n}$
if~\eqref{eq:cond.ess.spec} holds (see \Lem{diso_vs_delta}).

%-------------------------------------------------------------------------------
\subsubsection*{Quasi-unitary distance}
%-------------------------------------------------------------------------------
We introduce now a third distance for $R_1$ and $R_2$, based on
\emph{quasi-unitary equivalence}, a quantitative generalisation of
unitary equivalence introduced by the first author in~\cite{post:06}.
We define the \emph{quasi-unitary distance} of $R_1$ and $R_2$ as
\begin{subequations}
\begin{equation}
  \label{eq:que-dist}
  \dQUE(R_1,R_2)
  =\inf\bigset{\delta_J(R_1,R_2)} %
  {J \in \Lin{\HS_1,\HS_2}, \norm J \le 1},
\end{equation}
where
\begin{multline}
  \label{eq:delta_J}
  \delta_J(R_1,R_2)
  :=\max\bigl\{
  \norm{R_1^*(\id-J^*J)R_1}^{1/2},
  \norm{R_2^*(\id- JJ^*)R_2}^{1/2},\\
  \norm{JR_1-R_2J},
  \norm{JR_1^*-R_2^*J}
  \bigr\}.
\end{multline}
\end{subequations}
The operator $J$ is also called \emph{identification operator}.  We
need both, the third \emph{and} fourth term, in~\eqref{eq:delta_J} for
non-self-adjoint operators $R_n$ for the symmetry of $\dQUE$, see
\Lem{dque.metric}.

If $\delta \ge \delta_J(R_1,R_2)$ we say that $R_1$ and $R_2$ are
\emph{$\delta$-quasi unitarily equivalent}.  Quasi-unitary equivalence
is a quantitative generalisation of unitary equivalence for operators
$R_n$ with dense range, as $\delta_J(R_1,R_2)=0$ then implies that $J$
is unitary and $R_1$ and $R_2=JR_1J^*$ are unitarily equivalent, see
\Lem{uni.eq.dense}.  Note that we use a slightly modified variant of
quasi-unitary equivalence compared to~\cite{post-zimmer:22} and
earlier publications, see \Rem{new-que}.  We compare these different
notions of quasi-unitary equivalences and their associated distances
in a forthcoming article~\cite{post-zimmer:pre24c}.

%----------------------------------------------------------------------------
\begin{maintheorem}
  \label{thm:main-3}
  Assume that $R_n$ are bounded operators in a separable Hilbert space
  $\HS_n$ for $n=1,2$, then we have
  \begin{equation}
    \label{eq:dque.eq.diso}
    \dQUE(R_1,R_2)
    \le \dIso(R_1,R_2)
    \le \sqrt 3\dQUE(R_1,R_2).
  \end{equation}
\end{maintheorem}
%----------------------------------------------------------------------------
Note that in \Thm{main-3}, the Hilbert spaces $\HS_1$ and $\HS_2$ are
allowed to have different dimensions.  We doubt that the constant
$\sqrt 3$ is sharp, but it has to be at least $\sqrt 2$; see
\Cor{at.least.sqrt2}.

It is worth mentioning that for all distances $d_\bullet$ ($\dUni$,
$\dIso$ and $\dQUE$) we may have
$d_\bullet(R_1,R_2)<\norm[\Lin \HS]{R_1-R_2}$ even if
$\HS_1=\HS_2=\HS$, i.e., if the operators act in the \emph{same}
Hilbert space, see \Ex{dist.norm.diff}.

Summarising the results from \ThmS{main-1}{main-3}, we conclude:
%----------------------------------------------------------------------------
\begin{maincorollary}
  \label{cor:main-1'}
  If $R_n$ are self-adjoint and $0 \in \essspec{R_n}$ for $n=1,2$, then we have
  \begin{align*}
    \dQUE(R_1,R_2)
    \le \dIso(R_1,R_2)
    = \dUni(R_1,R_2)
    = \dSpec(R_1,R_2)
    \le \sqrt 3 \dQUE(R_1,R_2).
  \end{align*}
\end{maincorollary}
%----------------------------------------------------------------------------

If one of the distances is $0$ then we have:
%----------------------------------------------------------------------------
\begin{maincorollary}
  \label{cor:main-4}
  The following assertions are equivalent for general
  $R_n \in \Lin{\HS_n}$ ($n=1,2$):
  \begin{enumerate}
  \item
    \label{main-4.a}
    $\dIso(R_1,R_2)=0$;\footnote{$\dIso(R_1,R_2)=0$ means that there
      are isometries $\iota_{k,1}$ resp.\ $\iota_{k,2}$ from $\HS_1$
      resp.\ $\HS_2$ into some Hilbert space $\HS^{(k)}$ such that
      $\norm[\Lin{\HS^{(k)}}]{\iota_{k,1}R_1\iota_{k,1}^*-\iota_{k,2}R_2\iota_{k,2}^*}
      \to 0$; without loss of generality, one may assume that there is
      \emph{one} parent space for all $k$, namely
      $\hat \HS:=\bigoplus_k \HS^{(k)}$ and
      $\map{\hat \iota_{k,n}}{\HS_n}{\hat \HS}$ maps $f_n$ into
      $\hat \iota_{k,n} f_n=(\delta_{k,k'}) \iota_{k',n} f_n$;}
  \item
    \label{main-4.b}
    $\dQUE(R_1,R_2)=0$.\footnote{$\dQUE(R_1,R_2)=0$ means that there
      are $\map {J_k}{\HS_1}{\HS_2}$ with $\norm{J_k}\le 1$,
      $\norm{R_1^*(\id-J_k^*J_k)R_1} \to 0$,
      $\norm{R_2^*(\id-J_kJ_k^*)R_2} \to 0$,
      $\norm{J_k R_1-R_2J_k} \to 0$ and
      $\norm{J_k R_1^*-R_2^*J_k} \to 0$.}
  \end{enumerate}
  The following assertions are equivalent for self-adjoint operators
  $R_n \in \Lin{\HS_n}$ ($n=1,2$):
  \begin{enumerate}
    \addtocounter{enumi}{2}
  \item
    \label{main-4.c}
    $\dUni(R_1,R_2)=0$, i.e., $R_1$ and $R_2$ are approximately
    unitarily equivalent, i.e., there are unitary operators
    $\map {U_n}{\HS_1}{\HS_2}$ such that
    $\norm{U_nR_1U_n^*-R_2} \to 0$.;
  \item
    \label{main-4.e}
    the crude multiplicity functions agree, i.e.,
    $\alpha_{R_1}=\alpha_{R_2}$ on $\R$;
  \item
    \label{main-4.f}
    $\dSpec(R_1,R_2)=0$, or equivalently,
    \begin{align*}
      \essspec{R_1}=\essspec{R_2}
      \quadtext{and}
      \disspec{R_1}=\disspec{R_2}
    \end{align*}
    (including equality of the multiplicities of the discrete
    eigenvalues);
  \end{enumerate}
  Moreover, one of~\itemref{main-4.c}--\itemref{main-4.f} implies one
  of~\itemref{main-4.a}--\itemref{main-4.b}.  If additionally,
  $0 \in \essspec{R_1}\cap \essspec{R_2}$ for self-adjoint $R_1$ and
  $R_2$, then the opposite implications are true, i.e.\
  ~\itemref{main-4.a}--\itemref{main-4.f} are all equivalent.
\end{maincorollary}
%----------------------------------------------------------------------------
The equivalence of~\itemref{main-4.c}--\itemref{main-4.f} is basically
shown in~\cite{gellar-page:74} (even for normal operators) and also
in~\cite{azoff-davis:84}.  Note
that~\itemref{main-4.a}$\implies$\itemref{main-4.c} or
~\itemref{main-4.b}$\implies$\itemref{main-4.c} may fail if $0$ is not
in the essential spectrum, see \Ex{diso.0.duni.not}.

It is worth noting the classical fact that $\dUni(R_1,R_1)=0$ (or any
other of the equivalent conditions above) does \emph{not} imply that
$R_1$ and $R_2$ are unitarily equivalent.  Actually, $R_1$ and $R_2$
must have the same essential spectrum, but the essential spectrum may
either be dense pure point, absolutely or singularly continuous or of
mixed type.  For different spectral types, such operators are not
mutually unitarily equivalent, see e.g.~\cite[Examples~1.2, 1.4
and~1.6]{abrahamse:78} and \Subsec{duni.zero}.

%-----------------------------------------------------------------------
\subsection{Convergence of operators acting in different spaces}
\label{ssec:conv.ops}
%-----------------------------------------------------------------------
We now address the question of convergence of a sequence of operators
$R_n$ acting in Hilbert spaces $\HS_n$ towards an operator $R_\infty$
acting in $\HS_\infty$.  In our previous article~\cite{post-zimmer:22}
we introduced the notion of operator-norm convergence of resolvents in
the sense of Weidmann (for short \emph{Weidmann's convergence}) and in
the sense of quasi-unitary equivalence (for short
\emph{QUE-convergence}), generalising both the classical concept of
norm resolvent convergence presented e.g.\ in~\cite{kato:66}
and~\cite{reed-simon-1}.  Originally, both concepts
in~\cite{post-zimmer:22} were formulated for \emph{unbounded}
operators $\D_n$, but these operators enter only via their resolvents
\begin{equation}
  \label{eq:resolvents}
  R_n=(\D_n-z_0)^{-1}
  \quadtext{in a common point}
  z_0 \in \bigcap_{n \in \Nbar} \rho(\D_n),
\end{equation}
where $\rho(\D_n):=\C \setminus \spec {\D_n}$ denotes the
\emph{resolvent set} (we will not stress the dependency on $z_0$ in
the notation).  For
$\bullet \in \{\mathrm{uni},\mathrm{iso},\mathrm{que}\}$ we define the
$\bullet$-resolvent distance of $\D_1$ and $\D_2$ by
\begin{align*}
  \bar d_\bullet(\D_1,\D_2)
  := d_\bullet(R_1,R_2),
\end{align*}
and we suppress the dependence on $z_0$ in the notation.
% In particular, we
% will speak also of Weidmann's or QUE-convergence of $R_n$ towards
% $R_\infty$ instead of $\D_n$ towards $\D_\infty$ as
% in~\cite{post-zimmer:22}.
Note that~\eqref{eq:cond.ess.spec} is equivalent with the fact that
$A_1$ and $A_2$ are both unbounded operators.

The main result in~\cite{post-zimmer:22} is the
equivalence of Weidmann's and QUE-convergence, although there might be
a loss of convergence speed when using the original definition of
quasi-unitary equivalence, see \Rem{new-que} for details.

Our last main result is now the following (for the definition of
$\D_n \gnrcW \D_\infty$ and $\D_n \gnrc \D_\infty$ we refer to
\Defs{gnrc-weid}{gnrc-que}):
%----------------------------------------------------------------------------
\begin{maintheorem}
  \label{thm:main-5}
  Let $(\D_n)_n$ be a sequence of closed operators with common
  resolvent point $z_0$ each acting in $\HS_n$
  ($n \in \Nbar=\N \cup \{\infty\}$).  The following statements are
  equivalent:
  \begin{enumerate}
  \item
    \label{main-5.a}
    $\D_n \gnrcW \D_\infty$.
  \item
    \label{main-5.b}
    $\bardIso(\D_n,\D_\infty) \to 0$, i.e.,
    $\dIso(R_n,R_\infty) \to 0$ for $n \to \infty$.
  \item
    \label{main-5.c}
    $\D_n \gnrc \D_\infty$.
  \item
    \label{main-5.d}
    $\bardQUE(\D_n, \D_\infty) \to 0$, i.e., $\dQUE(R_n, R_\infty) \to 0$
    for $n \to \infty$.
  \end{enumerate}
  If the operators $\D_n$ are self-adjoint ($n \in \Nbar$), then the
  following statements are equivalent:
  \begin{enumerate}
    \addtocounter{enumi}{4}
  \item
    \label{main-5.e}
    $\D_n \gnrcUni \D_\infty$, i.e., there are unitary operators
    \begin{align*}
      \map{U_n}{\HS_n}{\HS_\infty}
      \quadtext{such that}
      \norm[\Lin{\HS_\infty}]{U_nR_nU_n^*-R_\infty} \to 0.
    \end{align*}
  \item
    \label{main-5.f}
    $\bardUni(\D_n,\D_\infty) \to 0$, i.e.,
    $\dUni(R_n,R_\infty) \to 0$ for $n \to \infty$.
  \item
    \label{main-5.f'}
    $\dSpec(R_n,R_\infty) \to 0$ for $n \to \infty$.
  \end{enumerate}
  Moreover, if $\D_n$ are self-adjoint ($n \in \Nbar$), then
  \itemref{main-5.e}--\itemref{main-5.f'} each implies one
  of~\itemref{main-5.a}--\itemref{main-5.d}.

  If all operators $\D_n$ are additionally unbounded ($n \in \Nbar$),
  then the converse is true,
  i.e.~\itemref{main-5.a}--\itemref{main-5.f'} are equivalent.

  If in addition, all operators $\D_n$ have purely essential spectrum
  ($n \in \Nbar$), then the following statement is also equivalent
  to~\itemref{main-5.a}--\itemref{main-5.f}):
  \begin{enumerate}
    \addtocounter{enumi}{7}
  \item
    \label{main-5.g}
    $\dHaus(\spec{R_n},\spec{R_\infty}) \to 0$.
  \end{enumerate}
  The convergence speed is the same in all cases.
\end{maintheorem}
%----------------------------------------------------------------------------
There is a subtle point in the definition of Weidmann's convergence
(see \Sec{conv} for details): we assume in~\itemref{main-5.a} the
existence of a \emph{common} Hilbert space in which \emph{all} Hilbert
spaces $\HS_n$ are embedded \emph{simultaneously}, while for the
definition of $\dIso(R_n,R_\infty)$ we only need a common space for
$R_n$ and $R_\infty$, which still may depend on $n \in \N$.  So
\Thm{main-5} does not immediately follow from our main results
in~\cite{post-zimmer:22}.

%-----------------------------------------------------------------------
\subsection{Related works}
\label{ssec:related.works}
%-----------------------------------------------------------------------

Probably the starting point of relations between the spectra and norms
of operators is an inequality formulated in~\cite[Satz~I]{weyl:12} for
operators $R_n$ in $\Lsqr{[a,b]}$ with kernel function
$r_n \in \Lsqr{[a,b]^2}$, which (in modern language) leads to
\begin{align}
  \label{eq:weyl}
  \bigabs{\lambda_k(R_1)-\lambda_k(R_2)}
  \le \max_k\bigabs{\lambda_k(R_1-R_2)}
  = \norm[\Lin{\Lsqr{[a,b]}}]{R_1-R_2},
\end{align}
where $\lambda_k(R_n)$ is the $k$-th eigenvalue of the compact
operator $R_n$ ordered decreasingly according to the absolute value of
$\lambda_k(R_n)$ and repeated according to their multiplicity.  Taking
now unitarily equivalent operators do not change the left hand side
of~\eqref{eq:weyl}.  The infimum over all such unitary equivalent
operators gives $\dUni(R_1,R_2)$, and the infimum is actually a
minimum, namely, when both operators are diagonalised and eigenvalues
are ordered properly.  In this case, we have even equality of the left
hand side and $\dUni(R_1,R_2)$.  A good overview of related results
for matrices or compact operators can be found in
Bhatia~\cite{bhatia:07} and references therein.  One of the main goals
of~\cite{bhatia:07} is to express distances between eigenvalues of two
operators (matrices) $R_1$ and $R_2$ acting in a common Hilbert space
$\HS$ in terms of $\norm{R_1-R_2}$.  Some recent developments on
generalisations of Schatten classes can be found
in~\cite{gueven-bandtlow:21}.  A focus on non-self-adjoint matrices
and operators is given in~\cite[Sec.~4]{gil:03}.  We will relate
eigenvalue matchings of compact (and Schatten class) operators acting
in different spaces and its relation with operator distances such as
$\dUni$ and $\dIso$ in~\cite{post-zimmer:pre24b} (see
also~\cite{zimmer:24}).  .

We already mentioned the work of Gellar and Page~\cite{gellar-page:74}
who characterised the closure of the unitary orbit of a normal
operator in terms of, what is called here \emph{crude multiplicity
  function}.  It was introduced under this name by Azoff and
Davis~\cite{azoff-davis:84}, where they showed that the unitary
distance and the \LevyProkhorov distance of the crude multiplicity
functions agree for self-adjoint bounded operators ($R^*=R$) acting in
the \emph{same} Hilbert space, see \Prp{azoff-davis}.  For
\emph{normal} operators ($R^*R=RR^*$) one naturally asks if
$\dUni(R_1,R_2)=\delta(\alpha_{R_1},\alpha_{R_2})$ still holds.  Here,
``$\le$'' is still true, but ``$\ge$'' may fail, even for matrices.
One can only show
$\dUni(R_1,R_2) \ge c^{-1} \delta(\alpha_{R_1},\alpha_{R_2})$ for some
$c \in (2.903,2.904)$, see~\cite[\S~33]{bhatia:07} and~\cite[end of
Sec.~2]{holbrook:92}.  Davidson~\cite{davidson:86} generalised this
inequality and also Azoff's and Davis' result to normal operators on
general (not necessarily separable) Hilbert spaces with the same
constant $c$ if discrete spectrum is present (if both operators are
purely essential, one can choose $c=1$, i.e.\ one has again equality
as in the self-adjoint case).  Generalisations to von Neumann and
$C^*$-algebras can be found
in~\cite{hiai-nakamura:89,sherman:07,hu-lin:15,jsv:21}, see also references
therein.  The latter focuses on viewing the \LevyProkhorov distance
as a Wasserstein distance used in optimal transport.

Works on comparing operators acting in different Hilbert spaces deal
mostly with some generalisation of strong resolvent convergence, see
the literature cited in~\cite{post-zimmer:22} and also
in~\cite{boegli:17}.  The concept of isometric distance is based on a
related convergence notion of Weidmann~\cite[Sec.~9.3]{weidmann:00}.
The concept of quasi-unitary distance was first introduced
in~\cite[Sec.~2.4]{post-simmer:20}.  The present article can be seen
as an affirmative answer to the open question
in~\cite[Sec.~1.3]{post-simmer:20} of how to express operator norm
convergence of operators acting in different spaces in metric terms
and show equivalence with spectral convergence (see \Thm{main-5}).

Our concept needs quite low regularity.  If one has more information
on the dependence of the family of Hilbert spaces on a parameter, one
might use asymptotic expansions as for example
in~\cite{panasenko-pileckas:15}.

%-----------------------------------------------------------------------
\subsection{Outlook}
%-----------------------------------------------------------------------
In a future work~\cite{post-zimmer:pre24b} (see
also~\cite{zimmer:24}), we will address the question of eigenvalue
enumeration for compact and Schatten class operators and relate the
optimal matching distance of the eigenvalues with some unitary and
isometric operator norm distances.  Moreover,
in~\cite{post-zimmer:pre24c} we will give more details on the quasi
unitary distance and variants of it.  It is still unclear to us
whether $\dQUE$ (or some reasonable variant of it) fulfils the
triangle inequality; our results only imply that it is fulfilled with
a constant between $\sqrt 2$ and $\sqrt 3$ by \Lem{dque.metric} and
\Cor{at.least.sqrt2}.  Also, the question whether the constant
$\sqrt 3$ in \Thm{main-3} is sharp remains open.

Some of our results extend to normal operators, namely \Thm{main-2}
should still hold.  An extension to general operators would need a
completely new proof avoiding crude multiplicity functions, but a
direct proof of $\dUni \le \dIso$ using~\eqref{eq:cond.ess.spec}.

Another continuation of our work would be to work with Banach spaces
and define convergence and distance concepts similarly as the unitary,
isometric and quasi-unitary notion.

%-----------------------------------------------------------------------
\subsection{Structure of the article}
%-----------------------------------------------------------------------
In \Sec{spec.cmf} we fix some notation and explain in detail the
concept of crude multiplicity functions and measures as well as some
distances of \LevyProkhorov type between them independently from their
appearance from operators.  Some of the results appeared already
in~\cite{azoff-davis:84,davidson:86}, but the authors use slightly
different definitions of the \LevyProkhorov distance, so we give a
self-contained presentation here (taking the crude multiplicity
function as a basic object itself).  In \Sec{duni.cmf} we briefly
review the results of~\cite{azoff-davis:84} on the equality of the
unitary distance with the \LevyProkhorov distance.  Moreover,
\Thm{main-1} is proven here.  In \Sec{diso.duni} we prove some facts
about the isometric distance and relate it with the unitary distance;
we prove \Thm{main-2} here.  \Sec{dque} provides some facts about the
quasi-unitary distance and its relation with the isometric distance,
namely the proof of \Thm{main-3}.  \Sec{examples} contains various
examples and counterexamples in order to show that the assumptions in
the main theorems cannot be weakened.  Finally, \Sec{conv} is about
related convergence notions of the distances; here, we relate our
results with the ones already shown in our previous
article~\cite{post-zimmer:22}.

%-----------------------------------------------------------------------
\subsection{Acknowledgements}
%-----------------------------------------------------------------------
We would like to thank our colleagues Leonhard Frerick, Fernando
Lled\'o and Jochen Wengenroth for useful discussions concerning this
article.

%-----------------------------------------------------------------------
%
% 2222
\section{Spectra and crude multiplicity functions}
\label{sec:spec.cmf}
%
%-----------------------------------------------------------------------

In this section, we fix some notation and collect facts used later for
our arguments; namely the concept of Hausdorff distance, crude
multiplicity functions, unitary distance and its relation with the
spectrum.  We start with some general notation.

Denote by $2^\R$ the power set of $\R$, i.e.\ the set of subsets of
$\R$, let $\Top(\R) \subset 2^\R$ be the set of open subsets of $\R$,
let $\OpInt \subset 2^\R$ be the set of open intervals and
$\Borel(\R)$ be the set of all measurable subsets of $\R$, i.e., the
Borel-$\sigma$-algebra generated by $\Top(\R)$.

%-----------------------------------------------------------------------
\subsection{Operators in a Hilbert space}
\label{sec:spec.thm}
%-----------------------------------------------------------------------

All Hilbert spaces considered here such as $\HS_1$ and $\HS_2$ are
assumed to be \emph{separable} (including sometimes also the
finite-dimensional case).  We consider two operators
$\map{R_1}{\HS_1}{\HS_1}$ and $\map{R_2}{\HS_2}{\HS_2}$.  We denote by
$\Lin{\HS_1,\HS_2}$ the space of all bounded operators equipped with
the operator norm $\norm[\Lin{\HS_1,\HS_2}]\cdot$.  If
$\HS=\HS_1=\HS_2$ we write briefly $\Lin \HS$ instead of
$\Lin{\HS,\HS}$.  By $\Unitary{\HS_1,\HS_2}$ we denote the set of
unitary operators $\map U{\HS_1}{\HS_2}$.  If $\HS=\HS_1=\HS_2$ we
write briefly $\Unitary \HS$ for the group of unitary operators
instead of $\Unitary{\HS,\HS}$.

Elements $\lambda \in \C$ are in the \emph{spectrum of $R$}
($\lambda \in \spec R$) if and only if $\map{R-\lambda}\HS \HS$ is not
bijective.  The \emph{discrete} spectrum is defined by
\begin{align*}
  \spec[disc] R
  := \set{\lambda \in \C}%
  {\text{$1 \le \dim \ker(R-\lambda)<\infty$, $\lambda$ is isolated in $\spec R$}},
\end{align*}
i.e., it is the set of all eigenvalues of finite multiplicity, which
are isolated in the spectrum.  The \emph{essential} spectrum is the
remaining part of $\spec R$, i.e.,
$\spec[ess] R = \spec R \setminus \spec[disc] R$, and it is a closed
set, as well as $\spec R$.  The spectrum is a unitary invariant, i.e.,
\begin{equation}
  \label{eq:spec.uni.eq}
  \spec{R_1}=\spec{UR_1U^*}
\end{equation}
for any unitary map $\map U{\HS_1}{\HS_2}$ (note that
$\map{R_1-z}{\HS_1}{\HS_1}$ is bijective if and only if
$\map{U(R_1-z)U^*=U R_1U^*-z}{\HS_2}{\HS_2}$ is bijective).  The same
is true for the discrete and essential spectrum.  For simplicity, we
mostly consider only self-adjoint operators here, some results extend
to normal or even non-self-adjoint operators.

Denote by $\Proj \HS$ the set of orthogonal projections in $\HS$.  The
\emph{spectral measure} associated with a self-adjoint operator $R$ on
$\HS$ is the map $\map{\1_{(\cdot)}(R)}{\Borel(\R)}{\Proj \HS}$, i.e.\
$\1_B(R)$ denotes the spectral projection of $R$ onto
$B \in \Borel(\R)$.  Moreover, $\lambda \in \spec R$ if and only if
$\lambda$ is in the support of this projection-valued measure, i.e.,
if
\begin{equation}
  \label{eq:spec.char}
  \forall \eps>0 \colon \1_{(\lambda-\eps,\lambda+\eps)}(R) \ne 0.
\end{equation}

%-----------------------------------------------------------------------
\subsection{The Hausdorff distance of arbitrary sets in metric spaces}
\label{ssec:gh}
%------------------------------------------------------------------------
The Hausdorff distance of two subsets $A_1,A_2 \subset M$ in a given
metric space $(M,d)$ is defined as
\begin{equation}
  \label{eq:hausdorff}
  \dHaus(A_1,A_2)
  := \max \bigl\{ \dHauS(A_1,A_2),\dHauS(A_2,A_1)\bigr\},
\end{equation}
where
\begin{equation}
  \label{eq:hausdorff0}
  \dHauS(A_1,A_2)
  := \sup_{x_1 \in A_1}d(x_1,A_2)
  \quadtext{and}
  \dHaus(x,A) := \inf_{a \in A} d(x,a).
\end{equation}
Equivalently, we have the representation
\begin{equation}
  \label{eq:eq.hausdorff}
  \dHaus(A_1,A_2) =
  \inf \set{\eps>0}{A_1 \subset \closedBall_\eps(A_2) \text{ and }
    A_2 \subset \closedBall_\eps(A_1)}
\end{equation}
for the Hausdorff distance, where
\begin{equation}
  \label{eq:nbd.ball}
  \closedBall_\eps(A):=\bigcup_{a \in A}\closedBall_\eps(a)
  \qquadtext{and}
  \closedBall_\eps(a):=\set{x \in M}{d(x,a)\le \eps}.
\end{equation}

%-----------------------------------------------------------------------
\subsection{Crude multiplicity functions}
\label{ssec:cmf}
%-----------------------------------------------------------------------
We introduce a concept first used implicitly by Gellar and
Page~\cite{gellar-page:74} and then introduced under the name
\emph{crude multiplicity function} by Azoff and
Davis~\cite{azoff-davis:84}.

The main and name-giving example comes from the spectral
characterisation of the discrete and essential spectrum of a
self-adjoint operator in terms of its crude multiplicity function, see
\Prp{char.cmf}.  The label ``crude'' refers to the fact that if the
distance between two crude multiplicity functions defined in
\Def{dist.cmf} is $0$, then the associated operators are not
necessarily unitarily equivalent, only the essential and discrete
spectra (including multiplicities) agree.  We comment on this fact in
\Subsec{duni.zero}.

Here, we first define crude multiplicity functions abstractly without
referring to operators; using an equivalent characterisation
in~\cite[Prop.~2.8]{azoff-davis:84}.

Let
\begin{align*}
  \Nbar := \N \cup \{\infty\}
  \qquadtext{and}
  \NbarO:=\{0\} \cup \Nbar = \N_0 \cup \{\infty\}.
\end{align*}

%---------------------------------------------------------------------------------
\begin{definition}[crude multiplicity function]
  \label{def:cmf}
  \indent %
  A map $\map \alpha \R \NbarO$ is called a \emph{crude multiplicity
    function} if
  \begin{enumerate}
  \item
    \label{def.cmf.a}
    $\alpha$ is upper semi-continuous, i.e., for all
    $\lambda_0 \in \R$ we have
    $\limsup_{\lambda \to \lambda_0}
    \alpha(\lambda)=\alpha(\lambda_0)$;
  \item
    \label{def.cmf.b}
    each $\lambda \in \alpha^{-1}(\N)$ is isolated in $\supp \alpha$,
    i.e., for all $\lambda \in \alpha^{-1}(\N)$ there is
    $\eps_\lambda>0$ such that
    $\openBall_{\eps_\lambda}(\lambda) \cap \supp \alpha=\{\lambda\}$.
  \end{enumerate}
\end{definition}
%--------------------------------------------------------------------------------

%--------------------------------------------------------------------------------
\begin{remark}
  We can replace~\itemref{def.cmf.a} by
  \begin{enumerate}
  \item[(\ref{def.cmf.a}')]
    \label{def.cmf.a'}
    $\alpha^{-1}(\{\infty\})$ is closed in $\R$.
  \end{enumerate}
  It follows that the \emph{support}
  $\supp \alpha:=\set{\lambda \in \R}{\alpha(\lambda)>0}$ of $\alpha$
  is closed in $\R$.
\end{remark}
%-------------------------------------------------------------------------------

We give an important characterisation for crude multiplicity functions
in terms of a bounded and self-adjoint operator in \Prp{char.cmf}.

%--------------------------------------------------------------------------------
\begin{definition}[discrete and essential support]
  \label{def:ess.cmf}
  Let $\map \alpha \R \NbarO$ be a crude multiplicity function.
  \begin{enumerate}
  \item We call
    \begin{align*}
      \dissupp \alpha := \alpha^{-1}(\N)
      \quadtext{and}
      \esssupp \alpha := \alpha^{-1}(\{\infty\})
    \end{align*}
    the \emph{discrete} and \emph{essential} support\footnote{The
      notion \emph{essential support} here should not be mixed up with
      the same notion for functions on a measure space.} of $\alpha$,
    respectively.
  \item We say that $\alpha$ is \emph{purely discrete} or \emph{purely
      essential} if $\supp \alpha = \dissupp \alpha$ or
    $\supp \alpha=\esssupp \alpha$, respectively.
  \end{enumerate}
\end{definition}
%--------------------------------------------------------------------------------
Hence, a function $\map \alpha \R \NbarO$ is a purely discrete resp.\
essential crude multiplicity function if it takes values only in $\N_0$
resp.\ $\{0,\infty\}$, and if $\supp \alpha =\alpha^{-1}(\{\infty\})$
is discrete resp.\ closed in $\R$.

We now show how we can construct a unique measure from the crude
multiplicity function.  In particular, we can either use the crude
multiplicity function or the corresponding measure.  The proof of the
next lemma is inspired by~\cite[Sec.~2]{azoff-davis:84}):
%-------------------------------------------------------------------------------
\begin{proposition}[measure associated with crude multiplicity
  function]
  \label{prp:cmf}
  For any crude multiplicity function $\alpha$ there is a unique
  countably additive measure $\map{\alpha^*}{\Borel(R)}\NbarO$ such
  that
  \begin{equation}
    \label{eq:cmm.char}
    \map{\alpha^*}{\Borel(\R)}\NbarO,\qquad
    B \mapsto \alpha^*(B)
    := \inf \set{ \alpha^*(U)}{U \in \Top(\R), B \subseteq U},
  \end{equation}
  called the \emph{crude multiplicity measure} associated with
  $\alpha$.  In particular, we have
  $\alpha(\lambda)=\alpha^*(\{\lambda\})$.
\end{proposition}
%-------------------------------------------------------------------------------
\begin{proof}
  We define $\alpha^*$ first on open subsets $U$ of $\R$.  We set
  \begin{align*}
    \alpha^*(U):= \sum_{\lambda \in U}\alpha(\lambda),
  \end{align*}
  where the sum is finite if and only if $U \cap \dissupp \alpha$ is
  finite and $U \cap \esssupp \alpha=\emptyset$.  It is straightforward
  to see that $\alpha^*$ is countably additive.  As the set of open
  subsets $\Top(\R)$ generates the Borel $\sigma$-algebra
  $\Borel(\R)$, we can extend $\alpha^*$ uniquely to $\Borel(\R)$.
  As $\alpha^*(\openBall_\eps(\lambda))$ is decreasing for the open ball
  $\openBall_\eps(\lambda)=(\lambda-\eps,\lambda+\eps)$ if $\eps \searrow 0$
  the measure of the set $\{\lambda\}$ can be calculated as
  \begin{align*}
    \alpha^*(\{\lambda\})
    = \lim_{\eps \to 0} \alpha^*(\openBall_\eps(\lambda)).
  \end{align*}
  We have to show that this limit equals the originally given value
  $\alpha(\lambda)$ for all $\lambda \in \R$: Let $n=\alpha(\lambda)$.
  \begin{itemize}
  \item If $n=0$, then $\alpha^*(\openBall_\eps(\lambda))=0$ for some $\eps>0$
    small enough; this follows from the very definition of $\alpha^*$
    and the fact that $\R \setminus \alpha^{-1}(\Nbar)$ is open.  In
    particular, $\alpha^*(\{\lambda\})=\alpha(\lambda)=0$.
  \item If $n \in \N$ then
    $\alpha^*(\openBall_\eps(\lambda))\ge \alpha(\lambda)$, and by the
    discreteness of $\alpha^{-1}(\{n\})$, we have equality for $\eps$
    small enough.
  \item If $n=\infty$ then
    $\infty=\alpha(\lambda)\le \alpha^*(\openBall_\eps(\lambda)) \le
    \infty$ for all $\eps>0$, hence also the desired
    equality. \qedhere
  \end{itemize}
\end{proof}
%-------------------------------------------------------------------------------

% Note that if $\alpha$ is purely essential, then also its associated
% measure $\alpha^*$ only takes values in $\{0,\infty\}$.

We need the measure associated from the crude multiplicity function
later when defining a distance between two crude multiplicity
functions, see \Def{dist.cmf}.

We now define the \emph{outer crude multiplicity measure}
$\wt \alpha^*$ associated with $\alpha^*$ for any subset $S \in 2^\R$
by
\begin{equation}
  \label{eq:ocmm.char}
  \map{\wt \alpha^*}{2^\R}\NbarO,\qquad
  S \mapsto \wt \alpha^* (S)
  := \inf \set{ \alpha^*(U)}{U \in \Top(\R), S \subseteq U},
\end{equation}
% Needed???
We collect some facts for the crude multiplicity measure, following
again~\cite{azoff-davis:84}:
%--------------------------------------------------------------------------------
\begin{proposition}
  \label{prp:cmf2}
  Let $\alpha$ be a crude multiplicity function and $\alpha^*$ its
  associated measure.  Then we have:
  \begin{enumerate}
  \item
    \label{cmf.a} For any $S \in 2^\R$ there is an open set
    $U_S \in \Top$ with $S \subseteq U_S$ and
    $\alpha^*(S)=\alpha^*(U_S)$.
  \item
    \label{cmf.b}
    For any disjoint compact subsets $K$ and $L$ of $\R$ there are
    disjoint open subsets $U_K$ and $U_L$ with $K \subset U_K$ and
    $L \subset U_L$ such that
    \begin{align*}
      \alpha^*(K \dcup L)
      =\alpha^*(U_K \dcup U_L)
      =\alpha^*(U_K)+\alpha^*(U_L)
      =\alpha^*(K)+\alpha^*(L)
    \end{align*}
  \item
    \label{cmf.c}
    For any finite subset $F \in 2^\R$ we have
    \begin{align*}
      \alpha^*(F)=\sum_{\lambda \in F}\alpha(\lambda)
    \end{align*}
  \item
    \label{cmf.d}
    The outer crude multiplicity measure $\wt \alpha^*$ is inner regular
    with respect to the discrete topology of $\R$, i.e.,
    \begin{align*}
      \wt \alpha^*(S)
      = \sup \set{\alpha^*(F)}{\text{$F$ finite subset of $S$}}
    \end{align*}
    for subsets $S \in 2^\R$, where the outer crude multiplicity
    measure $\wt \alpha^*$ associated with $\alpha^*$ is defined as
    in~\eqref{eq:cmm.char} (with $S \in 2^\R$ instead of
    $B \in \Borel(\R)$).

  \item
    \label{cmf.e}
    The outer crude multiplicity measure $\wt \alpha^*$ is already a
    measure on $2^\R$ and in particular countably additive.
  \end{enumerate}
\end{proposition}
%--------------------------------------------------------------------------------
Because of \Prpenum{cmf2}{cmf.e} we write again $\alpha^*$ instead of
$\wt \alpha^*$ for the (outer) measure defined on all subsets of $\R$.
%--------------------------------------------------------------------------------
\begin{proof}
  \itemref{cmf.a}~The infimum $\alpha^*(S)$ is attained as any subset of
  the cardinals $\NbarO$ has a minimum.  \itemref{cmf.b}~follows
  from~\itemref{cmf.a} and that any two disjoint compact subsets have
  disjoint open neighbourhoods.  \itemref{cmf.c}~follows
  from~\itemref{cmf.b}.

  \itemref{cmf.d}~A set is compact with respect to the discrete
  topology if and only if it is finite.  If $\alpha ^*(S)<\infty$,
  then $F:=\dissupp \alpha \cap S$ is finite and
  $\alpha^*(S)=\alpha^*(F)$.  If $\alpha^*(S)=\infty$, then either
  there is $\lambda \in S$ with $\alpha(\lambda)=\infty$; in this case
  choose $F=\{\lambda\}$; or there is a sequence
  $\lambda_n \to \lambda_\infty$ with $\lambda_n \in S$ and
  $\alpha(\lambda_n)\ge 1$.  Now, for each $n \in \N$ set
  $F_n:=\{\lambda_1,\dots,\lambda_n\}$ then
  $\alpha^*(F) =\sum_{k=1}^n\alpha(\lambda_n)\ge n$.  In particular,
  $\alpha^*(F_n) \nearrow \infty=\alpha^*(S)$ and we are done.

  \itemref{cmf.e}~A set $B \in 2^\R$ is $\wt \alpha^*$-measurable if
  and only if
  \begin{align*}
    \wt \alpha^*(S)
    =\wt \alpha^*(S\cap B) + \wt \alpha^*(S \cap \compl B)
  \end{align*}
  for all $S \in 2^\R$.
  % According to~\cite[Thm.~12.A]{halmos:50}, all
  % Borel sets $B \in \Borel(\R)$ are $\alpha^*$-measurable (does this
  % hold also for any $B \in 2^\R$?). %% YES!
  %
  % The crude multiplicity measure $\alpha^*$ is countably
  % additive, i.e., $\alpha^*(\bigcup_n U_n)=\sum_n \alpha^*(U_n)$ for
  % any disjoint sequence of open subsets $(U_n)_n \in \Top^\N$.
  By~\itemref{cmf.d}, we have
  \begin{align*}
    \wt \alpha^*(S\cap B) + \wt \alpha^*(S \cap \compl B)
    &=\sup\set{\alpha^*(F')}{F' \subset S \cap B}
      +\sup\set{\alpha^*(F'')}{F'' \subset S \cap \compl B}\\
    &=\sup\set{\alpha^*(F')+\alpha(F'')}{F' \subset S \cap B,
      F'' \subset S \cap \compl B}\\
    &=\sup\set{\alpha^*(F)}{F \subset S}
     =\wt \alpha ^*(S)
  \end{align*}
  for all $S \in 2^\R$, where $F'$, $F''$ and $F$ are finite sets,
  hence all sets $B \in 2^\R$ are $\wt \alpha^*$-measurable.  The
  statement that $\wt \alpha^*$ is a measure on $2^\R$ follows from
  the fact that an outer measure $\wt \alpha^*$ associated with a
  measure $\alpha^*$ on a $\sigma$-algebra extends uniquely to all
  $\wt \alpha^*$-measurable sets of the outer measure $\wt \alpha^*$.
\end{proof}
%-------------------------------------------------------------------------------

%-------------------------------------------------------------------------------
\begin{proposition}
  \label{prp:char.cmf}
  We have the following main example of a crude multiplicity
  functions:
  \begin{enumerate}
  \item
    \label{char.cmf.a}
    Let $R$ be a bounded and self-adjoint operator in a separable
    Hilbert space then
    \begin{align*}
      \alpha_R(\lambda)
      := \lim_{\eps \to 0} \rank \1_{\openBall_\eps(\lambda)}(R)
    \end{align*}
    is a crude multiplicity function with compact support.  We call
    $\alpha_R$ the \emph{crude multiplicity function associated with
      $R$}.

  \item
    \label{char.cmf.b}
    For the spectrum of $R$ and its associated crude multiplicity
    function $\alpha=\alpha_R$, we have:
    \begin{enumerate}
    \item
      \label{cmf.spec.a}
      $\spec R = \supp \alpha \mathrel{(=} \alpha^{-1}(\Nbar))$,
    \item
      \label{cmf.spec.b}
      $\disspec R = \dissupp \alpha \mathrel{(=} \alpha^{-1}(\N))$,
    \item
      \label{cmf.spec.c}
      $\essspec R = \esssupp \alpha \mathrel{(=}
      \alpha^{-1}(\{\infty\}))$.
    \end{enumerate}
  \item
    \label{char.cmf.c}
    If $\alpha_R$ is the crude multiplicity function of a bounded
    operator $R$, then its associated measure $\alpha^*_R$ (cf.\
    \Prp{cmf}) fulfils
    \begin{align*}
      \alpha^*_R(I)=\rank \dim \1_I(R)
    \end{align*}
    for any open subset $I \subset \R$.
  \end{enumerate}
\end{proposition}
%-------------------------------------------------------------------------------
\begin{proof}
  We prove~\itemref{char.cmf.a} and~\itemref{char.cmf.b} in one step:
  First, we use the characterisation that $\lambda \in \spec R$ if and
  only if for all $\eps>0$ we have
  $\1_{\openBall_\eps(\lambda)}(R)\ne 0$, but the latter is equivalent
  to $\alpha_R(\openBall_\eps(\lambda)) \ge 1$, or, taking the limit
  $\eps \to 0$, the latter is further equivalent to
  $\alpha_R(\lambda) \ge 1$, hence we have $\spec R=\supp {\alpha_R}$.

  As $\rank \1_{\openBall _\eps(\lambda)}(R)$ has values in $\Nbar_0$,
  the limit is eventually constant.  In particular, if
  $\alpha_R(\lambda) \in \N$, then there is $\eps_0>0$ such that
  $\rank \1_{\openBall_{\eps_0}(\lambda)\setminus \{\lambda\}}(R)=0$,
  i.e., $\lambda$ is isolated in $\alpha_R^{-1}(\N)$.  This shows
  \Defenum{cmf}{def.cmf.b} and also $\disspec R=\dissupp \alpha_R$.
  For \Defenum{cmf}{def.cmf.a}, note that
  \begin{align*}
    \alpha_R^{-1}(\{\infty\})
    =\esssupp \alpha_R
    =\supp \alpha_R \setminus \dissupp \alpha_R
    =\spec R \setminus \disspec R
    =\essspec R,
  \end{align*}
  and the latter set is closed.

  \itemref{char.cmf.c}~Let $I$ be an open interval, then, by
  definition,
  \begin{align*}
    \alpha ^*_R(I)
    =\sum_{\lambda \in I} \alpha_R(\lambda)
    =\sum_{\lambda \in I}
    \lim_{\eps \to 0} \rank \1_{\openBall_\eps(\lambda)}(R).
  \end{align*}
  If $\alpha^*_R(I)=\infty$, then either there is $\lambda \in I$ such
  that $\alpha_R(\lambda)=\infty$, i.e.,
  \begin{align*}
    \infty
    = \alpha_R(\lambda)
    \le \rank \1_{\openBall_\eps(\lambda)}(R)
    \le \rank \1_I(R)
    \le \infty,
  \end{align*}
  as $\alpha_R(\lambda)$ is the limit of a function monotonous in
  $\eps$.  In particular, we have shown
  $\alpha^*_R(I)=\rank \1_I(R)=\infty$.  If $\alpha^*_R(I)=\infty$,
  then there might also be a monotonous sequence $(\lambda_k)_k$ with
  $\lambda_k \in I$, $1\le \alpha_R(\lambda_n) <\infty$ and
  $\lambda_k \to \lambda_\infty$.  In this case,
  $\lambda_k \in \disspec {\alpha_R}$ and
  $\lambda_\infty \in \essspec R=\esssupp \alpha_R$.  Then we have
  \begin{align*}
    \infty \ge
    \rank \1_I(R)
    \ge \rank \1_{\bigdcup_k \{\lambda_k\}}
    = \sum_k \rank \1_{\{\lambda_k\}}
    \ge \sum 1
    = \infty,
  \end{align*}
  i.e., again $\alpha^*_R(I)=\rank \1_I(R)=\infty$.

  If $\alpha^*_R(I)<\infty$, then $\supp \alpha_R \cap I$ consists of
  finitely many isolated points only.  Choose $\eps_0>0$ such that
  $\bigcup_{\lambda \in \supp \alpha_R \cap
    I}\openBall_{\eps_0}(\lambda)$ is a disjoint union.  Then
  $\alpha_R(\lambda)=\rank \1_{\openBall_{\eps_0}(\lambda)}$ for each
  $\lambda \in \supp \alpha_R \cap I$ and
  \begin{align*}
    \rank \1_I(R)
    =\rank \1_{I \cap \supp \alpha_R}(R)
    = \sum_{\lambda \in I \cap \supp \alpha_R}
    \rank \1_{\{\openBall_{\eps_0}(\lambda)\}}
    = \sum_{\lambda \in I \cap \supp \alpha_R} \alpha_R(\lambda),
  \end{align*}
  i.e., the desired equality also in this case.
\end{proof}
%-------------------------------------------------------------------------------

We have the following converse statement (see
also~\cite[Prop.~2.8]{azoff-davis:84}):
%-------------------------------------------------------------------------------
\begin{proposition}
  \label{prp:char.cmf2}
  If $\alpha$ is a crude multiplicity function with compact support
  then there is a bounded operator $R$ in a separable Hilbert space
  such that its crude multiplicity function $\alpha_R$ equals $\alpha$
\end{proposition}
%-------------------------------------------------------------------------------
\begin{proof}
  For $n \in \N$ let $S_n=\alpha^{-1}(\{n\})$; for $n=\infty$ let
  $S_n$ be a countable (or finite) and dense subset of
  $\alpha^{-1}(\{\infty\})$.  Let $\HS$ be a (separable) Hilbert space
  with orthonormal basis $(\phi_k)_k$ with the same cardinality as
  $\alpha(\R)$.  Let now $(\lambda_k)_k$ be a sequence such that each
  value $\lambda \in S_n$ ($n \in \Nbar$) appears in $(\lambda_k)_k$.
  Moreover, if $\alpha(\lambda)=n$ then
  $\set{k \in \Nbar}{\lambda_k=\lambda}$ contains exactly $n$ elements
  (in other words, if $\alpha(\lambda)=n$, then we repeat $\lambda$
  exactly $n$ times in $(\lambda_k)_k$.  For each $n \in \Nbar$, let
  $K_n$ be a subset of $\N$ such that
  \begin{align*}
    \alpha^{-1}(\{n\})
    =\set{\lambda_k}{k \in K_n},
    \qquadtext{then we have}
    \N=\bigdcup_{n \in \Nbar} K_n,
  \end{align*}
  Now we set
  $R=\sum_{n \in \Nbar} \sum_{k \in K_n} \lambda_k \iprod \cdot
  {\phi_k} \phi_k$.  As $\supp \alpha$ is compact, $R$ is a bounded
  operator with $\spec R = \supp \alpha$.  Note that each
  $\lambda \in \supp \alpha$ appears for exactly one $k \in K_n$ if
  $n=\alpha(\lambda) \in \N$ and at most once if $n=\infty$.
  Moreover, note that
  \begin{gather*}
    \rank \1_I(R)
    =\card{{\set{k \in \N}{\lambda_k \in I}}}
    =\sum_{n \in \Nbar} \card{{\set{k \in K_n}{\lambda_k \in I}}}
    \quad\text{and}\\
    \lim_{\eps \to 0} \card{{\set{k \in K_n}{\lambda_k
    \in \openBall_\eps(\lambda)}}}
    = \alpha(\lambda).
  \end{gather*}
  It remains to show that the crude multiplicity function $\alpha_R$
  associated with $R$ is $\alpha$:
  \begin{align*}
    \alpha_R(\lambda)
    &=\lim_{\eps \to 0} \rank \1_{\openBall_\eps(\lambda)}(R)\\
    &=\lim_{\eps \to 0} \sum_{n \in \Nbar}
    \card{\set{k \in K_n}{\lambda_k \in \openBall_\eps(\lambda)}}\\
    &=\sum_{n \in \Nbar}
      \lim_{\eps \to 0}
      \card{\set{k \in K_n}{\lambda_k \in \openBall_\eps(\lambda)}}\\
    &=\alpha(\lambda). \qedhere
  \end{align*}
\end{proof}
%-------------------------------------------------------------------------------

%--------------------------------------------------------------------------------
\begin{remarks}
  \label{rem:cmf}
  \indent
  \begin{enumerate}
  \item As we only consider separable (or finite dimensional) Hilbert
    spaces, one infinite cardinal number $\infty$ is enough.

  \item We do not have $\alpha(\lambda)=\rank \1_{\{\lambda\}}(R)$ in
    general: we have $\alpha(\lambda)>0$ for all $\lambda \in \spec R$
    (see \Prpenum{char.cmf}{char.cmf.b}), but
    $\1_{\{\lambda\}}(R)\ne 0$ if and only if $\lambda$ is an
    eigenvalue of $\R$.
  \end{enumerate}
\end{remarks}
%----------------------------------------------------------------------------

%-----------------------------------------------------------------------
\subsection{Distances between crude multiplicity functions}
\label{ssec:dist.cmf}
%-----------------------------------------------------------------------

Next, we define the \emph{\LevyProkhorov distance} of the associated
measures $\alpha_1^*$ and $\alpha_2^*$
(see~\cite[Def.~2.2]{azoff-davis:84}).  Originally, this distance has
been defined only for probability measures, but the extension to
cardinality-valued measures is straightforward.
%---------------------------------------------------------------------------------
\begin{definition}[\LevyProkhorov distance]
  \label{def:dist.cmf}
  For two crude multiplicity functions $\alpha_1$ and $\alpha_2$ and a
  subset $S \subset \R$ we define
  \begin{subequations}
    \begin{align}
      \label{eq:dist.cmf.a}
      \acute \delta(\alpha_1,\alpha_2,S)
      &:= \inf \set{\eps>0}{\alpha_1^*(S) \le \alpha_2^*(\closedBall_\eps(S))},\\
      \acute \delta(\alpha_1,\alpha_2)
      &:= \sup \set{\acute \delta(\alpha_1,\alpha_2,I)}{I \in \OpInt}
      \quad\text{and}\\
      \label{eq:dist.cmf.c}
      \dCmffiN(\alpha_1,\alpha_2)
      &:= \sup \set{\acute \delta(\alpha_1,\alpha_2,F)}
        {F \in \Fin(\dissupp \alpha_1)},
    \end{align}
    where $\closedBall_\eps(S)$ denotes the (closed)
    $\eps$-neighbourhood of $S\subset \R$, see~\eqref{eq:nbd.ball}.
    Recall that $\OpInt$ is the set of all open intervals and
    $\Fin(\Sigma)$ the set of all finite subsets of $\Sigma$.
    \begin{enumerate}
    \item The \emph{(\LevyProkhorov) distance} of $\alpha_1$ and
      $\alpha_2$ is defined as
      \begin{equation}
        \label{eq:delta.cr.mult.c}
        \delta(\alpha_1, \alpha_2)
        := \max \{\acute \delta(\alpha_1,\alpha_2),
        \acute \delta(\alpha_2,\alpha_1)\}.
      \end{equation}
    \item The \emph{finite (\LevyProkhorov) distance} of $\alpha_1$
      and $\alpha_2$ is defined as
      \begin{equation}
        \label{eq:delta.cr.mult.d}
        \dCmffin(\alpha_1, \alpha_2)
        := \max \{    \dCmffiN(\alpha_1,\alpha_2),
        \dCmffiN(\alpha_2,\alpha_1)\}.
      \end{equation}
    \end{enumerate}
  \end{subequations}
\end{definition}
%--------------------------------------------------------------------------------

We have the following alternative description of
$\delta(\alpha_1,\alpha_2)$.  This characterisation resembles the
characterisation of the Hausdorff distance in~\eqref{eq:eq.hausdorff};
its relation with the Hausdorff distance is made precise in
\Prp{dist.cmf.haus}.
%----------------------------------------------------------------------------
\begin{subequations}
\begin{proposition}
  \label{prp:dist.cmf}
  We have
  \begin{align}
    \label{eq:delta.cr.mult.e}
    \delta(\alpha_1, \alpha_2)
    &= \inf\set{\eps > 0}
    {\forall I \in \OpInt \colon
    \alpha_1^*(I)\le \alpha_2^*(\closedBall_\eps(I))
    \text{ and }
    \alpha_2^*(I)\le \alpha_1^*(\closedBall_\eps(I))},\\
    \dCmffin(\alpha_1, \alpha_2)
    &= \inf\set{\eps > 0}
    {\forall F_1 \in \Fin(\dissupp \alpha_1) \colon
    \alpha_1^*(F_1)\le \alpha_2^*(\closedBall_\eps(F_1))
      \text{ and }\nonumber\\
    &\hspace{0.25\textwidth}
      \forall F_2 \in \Fin(\dissupp \alpha_2) \colon
      \alpha_2^*(I)\le \alpha_1^*(\closedBall_\eps(I))},
    \label{eq:delta.cr.mult.f}
  \end{align}
  (recall that $\OpInt$ is the set of all open intervals and that
  $\Fin(\dissupp \alpha_n)$ consists of all finite subsets of
  $\dissupp \alpha_n$).
\end{proposition}
%---------------------------------------------------------------------------
\begin{proof}
  Let
  $\acute L(S) := \acute
  L(\alpha_1,\alpha_2,S):=\set{\eps>0}{\alpha_1^*(S) \le
    \alpha_2^*(\closedBall_\eps(S))}$ and
  $\grave L(S) := \acute L(\alpha_2,\alpha_1,S)$.  Note that
  $\acute L(S)$ is an unbounded interval in $(0,\infty)$ as
  $\eps \in \acute L(S)$ and $\eps' \ge \eps$ implies
  $\eps' \in \acute L(S)$.  In particular, the intersection of any such
  intervals is always non-empty.

  Now,~\eqref{eq:delta.cr.mult.e} is equivalent to
  \begin{align*}
    \max \{ \sup_I \inf \acute L(I), \sup_I \inf \grave L(I) \}
    = \inf \bigcap_I (\acute L(I) \cap \grave L(I)),
  \end{align*}
  where $I$ runs through all open subsets of $\R$.  Actually, this is
  a consequence of
  \begin{align}
    \label{eq:inf.cap}
    \max \{\inf L_1,\inf L_2\}
    = \inf L_1 \cap L_2
    \quadtext{and}
    \sup_{k \in K} \inf L_k
    = \inf \bigcap_{k \in K} L_k.
  \end{align}
  Note that the first one is a special case of the second one (if
  $K=\{1,2\}$).  In both equalities, ``$\ge$'' is always true.  The
  other inequalities are only true if the intersection is non-empty
  and if all sets are \emph{intervals}.\footnote{A counterexample is
    given by $L_k=\{1/(k+1),1\}$, then
    $\max \{\inf L_1,\inf L_2\}=1/2$, while $\inf L_1 \cap L_2=1$;
    here the $L_k$'s are not intervals.  But even if they are, the
    statement might fail: choose e.g.\ $L_k=(1-k,2-k)$ for $k=1,2$,
    then the maximum of the infima is $0$, while
    $\inf L_1 \cap L_2=\inf \emptyset=\infty$.  Similar statements
    hold for $k \in \N$.}

  Here, the desired equality follows as all sets $\acute L(I)$ and
  $\grave L(I)$ under competition are intervals and their intersection
  is non-empty.
\end{proof}
\end{subequations}
%---------------------------------------------------------------------------

We now show the following equivalent characterisations of the
condition appearing in the definition of crude multiplicity functions:
%---------------------------------------------------------------------------
\begin{lemma}
  \label{lem:cmf.open}
  Let $\alpha_n$ be two crude multiplicity functions ($n=1,2$).
  \begin{enumerate}
  \item
    \label{cmf.open.a}
    For $\eps>0$ we have
    \begin{align*}
      \forall I \in \OpInt\colon
      \alpha_1(I) \le \alpha_2^*(\openBall_\eps(I))
      \quad\iff\quad
      \forall U \in \Top(\R) \colon
      \alpha_1(U) \le \alpha_2^*(\openBall_\eps(U)).
    \end{align*}
  \item
    \label{cmf.open.b}
    We have
    \begin{align*}
      \acute \delta(\alpha_1,\alpha_2)
      &= \inf\set{\eps>0}{\forall I \in \OpInt \colon
        \alpha_1(I) \le \alpha_2^*(\openBall_\eps(I))}\\
      &= \inf\set{\eps>0}{\forall U \in \Top(\R) \colon
        \alpha_1(U) \le \alpha_2^*(\openBall_\eps(U))}.
      % \\
      % &= \inf\set{\eps>0}{\forall K \in \Comp(\R) \colon
      %   \alpha_1(K) \le \alpha_2^*(\closedBall_\eps(K))}.
    \end{align*}
    Here, $\openBall_\eps(I)=\set{\lambda \in \R}{d(\lambda,I)<\eps}$
    is the \emph{open} $\eps$-neighbourhood of $I$,
    cf.~\eqref{eq:nbd.ball};
  \end{enumerate}
\end{lemma}
%---------------------------------------------------------------------------
\begin{proof}
  \itemref{cmf.open.a}: ``$\Leftarrow$'' is trivial.  For
  ``$\Rightarrow$'', let $\openBall_\eps(U)=\bigcup_k I_{k,\eps}$ (any
  open subset of $\R$ is the disjoint union of open intervals).  The
  length of each $I_{k,\eps}$ must be at least $2\eps$: Let
  $\lambda \in U$ then
  $(\lambda-\eps,\lambda+\eps)\subset \openBall_\eps(U)=\bigdcup_k
  I_{k,\eps}$, and as the latter union is disjoint, there must be
  $k=k_\lambda$ such that
  $(\lambda-\eps,\lambda+\eps)\subset I_{k,\eps}$.  Then let $I_k$ be
  the non-empty open interval such that
  $I_{k,\eps}=\openBall_\eps(I_k)$.  We have $\lambda \in I_k$, hence
  also $U \subset \bigdcup_k I_k$.  Now we conclude
  \begin{equation*}
    \alpha_1^*(U)
    \le\alpha_1^*\Bigl(\bigdcup_k I_k\Bigl)
    =\sum_k \alpha_1^*(I_k)
    \le \sum_k \alpha_2^*(\openBall_\eps(I_k))
    =\alpha_2\Bigl(\bigdcup_k \openBall_\eps(I_k)\Bigl)
    =\alpha_2(\openBall_\eps(U)).
  \end{equation*}
  \itemref{cmf.open.b}: Denote the left hand side $\delta_1$ and the
  right hand sides $\delta_2$ and $\delta_3$, % and $\delta_4$,
  respectively.  We first show $\delta_1 = \delta_2$: For
  $S \subset \R$, let
  $\acute {\ring L}(S):=\set{\eps>0}{\alpha_1(S) \le
    \alpha_2^*(\openBall_\eps(S)}$ and $\acute L(S)$ as in the proof
  of \Prp{dist.cmf}.  We have
  $\acute {\ring L}(S) \subset \acute L(S)$, hence
  \begin{align*}
    \bigcap_I \acute {\ring L}(I)
    \subset \bigcap_I \acute L(I),
    \quadtext{hence}
    \delta_1
    =\inf \bigcap_I \acute L(I)
    \le \inf \bigcap_I \acute {\ring L}(I)
    =\delta_2
  \end{align*}
  (we use again~\eqref{eq:inf.cap}).  For $\delta_2 \le \delta_1$, let
  $\eps>0$ such that $\eps> \delta_1$, then there is an open interval
  $I$ such that $\eps \notin \acute {\ring L}(I)$, i.e.,
  $\alpha_1(I)>\alpha_2^*(\closedBall_\eps(I))$.  As
  $\alpha_2^*(\closedBall_\eps(I))\ge \alpha_2^*(\openBall_\eps(I))$,
  we also have $\eps \notin \acute L(I)$, hence $\eps > \delta_2$ and
  we conclude $\delta_2 \le \delta_1$.  $\delta_2=\delta_3$ follows
  from~\itemref{cmf.open.a}.
\end{proof}
%---------------------------------------------------------------------------

Note that if $\alpha_1$ and $\alpha_2$ are purely essential, then
$\dCmffin(\alpha_1,\alpha_2)=0$, as
$\acute \delta(\alpha_1,\alpha_2,\emptyset)=0$.

The following propositions have been proven
in~\cite[Prp.~1.3]{davidson:86} using a slightly different
(equivalent) definition for the \LevyProkhorov distance.  We give
again a proof here, also because we defined crude multiplicity
functions \emph{independently} of operators (see \Def{cmf} and
\Prp{char.cmf}).
%------------------------------------------------------------------------------
\begin{proposition}
  \label{prp:alt.dist.cmf}
  We have
  \begin{align*}
    \dCmffiN(\alpha_1,\alpha_2)
    \le \acute \delta(\alpha_1,\alpha_2)
    \quadtext{and}
    \dCmffin(\alpha_1,\alpha_2)
    \le \delta(\alpha_1,\alpha_2)
  \end{align*}
\end{proposition}
%------------------------------------------------------------------------------
\begin{proof}
  Let $\eps \in \R$ such that
  $0<\eps <\dCmffiN(\alpha_1,\alpha_2)$, and let $\eps'>0$ such
  that $\eps < \eps' < \dCmffiN(\alpha_1,\alpha_2)$.  Note that
  \begin{align}
    \nonumber
    \dCmffiN(\alpha_1,\alpha_2)
    &=\sup \set{\acute \delta(\alpha_1,\alpha_2,F)}
      {F \in \Fin(\dissupp \alpha_1)}\\
    \label{eq:char.cmffin}
    &=\inf \bigset{\eps>0}%
    {\forall F \in \Fin(\dissupp \alpha_1)\colon
      \alpha_1^*(F) \le \alpha_2^*(\closedBall_\eps(F))}.
  \end{align}
  by~\eqref{eq:inf.cap}.  In particular, there is a finite set
  $F \subset \dissupp \alpha_1$ such that
  $\alpha_2^*(\closedBall_{\eps'}(F))< \alpha_1^*(F)< \infty$.  Let
  $\lambda \in F$ such that
  $\alpha_2^*(\closedBall_{\eps'}(\lambda)) <
  \alpha_1^*(\{\lambda\})$.  Set $U:=\openBall_{\eps'-\eps}(F)$.  As
  $\eps'>\eps$, we have $F \subset U$.  Moreover,
  $\closedBall_\eps(U)=\closedBall_{\eps'}(F)$ and
  \begin{align*}
    \alpha_2^*(\openBall_\eps(U))
    \le \alpha_2^*(\closedBall_\eps(U))
    <\alpha_1^*(F) \le
    \alpha_1^*(U),
  \end{align*}
  hence $\eps < \acute \delta(\alpha_1,\alpha_2)$, showing the desired
  inequality for open sets; for intervals use~\Lem{cmf.open}.  The
  last assertion follows from the first one.
\end{proof}
%------------------------------------------------------------------------------

%--------------------------------------------------------------------------------
\begin{proposition}
  \label{prp:dist.cmf.haus}
  Let $\alpha_n$ be two crude multiplicity functions ($n=1,2$).  Then we have
  \begin{subequations}
    \begin{enumerate}
    \item
      \label{dist.cmf.haus.a}
      \begin{equation}
      \label{eq:dist.cmf.haus.a}
      \dHaus(\supp[\bullet] \alpha_1,\supp[\bullet] \alpha_2)
      \le \delta(\alpha_1,\alpha_2)
    \end{equation}
    for the discrete, essential or entire support
    ($\bullet \in\{\mathrm{disc}, \mathrm{ess}, \emptyset\}$).
    \item
      \label{dist.cmf.haus.b}
      Moreover, we have
      \begin{equation}
        \label{eq:dist.cmf.haus.b}
        \max\{\dHaus(\esssupp \alpha_1,\esssupp \alpha_2),
        \dCmffin(\alpha_1,\alpha_2)\}
        = \delta(\alpha_1,\alpha_2).
      \end{equation}
    \item
      \label{dist.cmf.haus.c}
      In particular, if $\alpha_1$ and $\alpha_2$ are purely essential,
      then we have
      \begin{equation}
        \label{eq:dist.cmf.haus.c}
        \dHaus(\supp \alpha_1,\supp \alpha_2)
        = \delta(\alpha_1,\alpha_2).
      \end{equation}
    \end{enumerate}
  \end{subequations}
\end{proposition}
%--------------------------------------------------------------------------------
\begin{proof}
  Set $\Sigma_{n,\bullet} := \supp[\bullet] \alpha_n$.  For the first
  assertion~\eqref{eq:dist.cmf.haus.a}, let $\eps \in \R$ such that
  $0<\eps <\dHaus(\Sigma_{1,\bullet},\Sigma_{2,\bullet})$. Let
  $\eps'>0$ such that
  $\eps < \eps' < \dHaus(\Sigma_{1,\bullet},\Sigma_{2,\bullet})$.  By
  the characterisation~\eqref{eq:eq.hausdorff}, at least one of the
  conditions for $\eps'$ is not fulfilled, say the first one, i.e.,
  $\Sigma_{1,\bullet} \nsubseteq
  \closedBall_{\eps'}(\Sigma_{2,\bullet})$.  Let
  $\lambda \in \Sigma_{1,\bullet}$ such that
  $\lambda \notin \closedBall_{\eps'}(\Sigma_{2,\bullet})$, i.e,
  $\dist(\lambda,\Sigma_{2,\bullet})\ge \eps'$.  Let
  $I:=\openBall_{\eps'-\eps}(\lambda)$.  As $\eps'>\eps$, we have
  $I \neq \emptyset$.  Moreover,
  $\closedBall_\eps(I)=\closedBall_{\eps'}(\lambda)$ and we have
  $\closedBall_\eps(I) \cap \Sigma_{2,\bullet}=\emptyset$.  In
  particular,
  $\closedBall_\eps(I) \cap \Sigma_{2,\bullet} = \emptyset$ and
  $\alpha_2^*(\closedBall_\eps(I))=0$.  As
  $\lambda \in I \cap \Sigma_{1,\bullet}$, we have
  $\alpha_1^*(I) = \infty$ for $\bullet=\mathrm{ess}$ resp.\
  $\alpha_1^*(I) \ge 1$ for $\bullet \in \{\mathrm{disc},\emptyset\}$,
  hence the first condition in the
  characterisation~\eqref{eq:delta.cr.mult.e} of
  $\delta(\alpha_1,\alpha_2)$ is not fulfilled, i.e.,
  $\eps < \delta(\alpha_1,\alpha_2)$, showing the desired inequality.

  For the second assertion~\eqref{eq:dist.cmf.haus.b} note that
  ``$\le$'' holds because of~\eqref{eq:dist.cmf.haus.a} and
  \Prp{alt.dist.cmf}.  In order to show ``$\ge$'' let
  $r<\delta(\alpha_1,\alpha_2)$.  Without loss of generality, assume
  that $\delta(\alpha_1,\alpha_2)=\acute \delta(\alpha_1,\alpha_2)$,
  then there is an open interval $I$ such that
  \begin{equation*}
    \alpha_2^*(\closedBall_r(I))
    < \alpha_1^*(I).
  \end{equation*}
  \begin{itemize}
  \item If $\alpha_1^*(I)=\infty$, then
    $I \cap \Sigma_{1,\mathrm{ess}} \ne \emptyset$.  For the essential
    support of $\alpha_2$, this inequality means that
    \begin{equation*}
      0=\alpha_2^*(\closedBall_r(I)\cap \Sigma_{2,\mathrm{ess}})
      <\alpha_1^*(I)
      =\infty.
    \end{equation*}
    In particular, there is
    $\lambda_0 \in I \cap \Sigma_{1,\mathrm{ess}}$; and, as
    $\closedBall_r(I) \cap \Sigma_{2,\mathrm{ess}}=\emptyset$, we have
    $d(\lambda_0,\Sigma_{2,\mathrm{ess}})>r$ and hence
    $\lambda_0 \notin \closedBall_r(\Sigma_{2,\mathrm{ess}})$.  In particular,
    the first condition in the
    characterisation~\eqref{eq:eq.hausdorff} of the Hausdorff distance
    is violated for $\eps=r$, hence
    $r<\dHaus(\Sigma_{1,\mathrm{ess}},\Sigma_{2,\mathrm{ess}})$.  As
    $r<\delta(\alpha_2,\alpha_2)$ was
    arbitrary, we have shown
    $\dHaus(\Sigma_{1,\mathrm{ess}},\Sigma_{2,\mathrm{ess}}) \ge
    \delta(\alpha_1,\alpha_2)$.
  \item If $0<\alpha_1^*(I)<\infty$, then we may choose
    $\lambda \in I$ such that
    $\alpha_2^*(\closedBall_r(\lambda))<\alpha_1^*(\{\lambda\})$.  In
    particular, $r<\dCmffiN(\alpha_1,\alpha_2)$, as the
    inequality in~\eqref{eq:char.cmffin} is not fulfilled for the
    finite set $F=\{\lambda\}$.  Hence , we have shown the inequality
    $\dCmffiN(\alpha_1,\alpha_2) \ge \acute
    \delta(\alpha_1,\alpha_2)$.
  \end{itemize}
  The last assertion~\eqref{eq:dist.cmf.haus.c} follows
  from~\eqref{eq:dist.cmf.haus.b} and the fact that
  $\dCmffin(\alpha_1,\alpha_2)=0$ if $\alpha_1$ and $\alpha_2$ are
  both purely essential.
\end{proof}
%------------------------------------------------------------------------------

We end with a naive approach how to measure the distance between two
crude multiplicity function, which is true only in a special case:
%------------------------------------------------------------------------------
\begin{lemma}
  \label{lem:pointwise}
  Let $\alpha_n$ be two crude multiplicity functions ($n=1,2$).  For
  $r>0$, we set $\alpha_2^r(\lambda):= \alpha_2(\closedBall_r(I))$.
  Now, \itemref{pointwise.a}, \itemref{pointwise.b} or
  \itemref{pointwise.c} implies~\itemref{pointwise.d}, where
  \begin{enumerate}
  \item
    \label{pointwise.a}
    for all open intervals $I \subset \R$ we have
    $\alpha_1^*(I) \le \alpha_2^*(\closedBall_r(I))$ (closed
    $r$-neighbourhood);
  \item
    \label{pointwise.b}
    for all open intervals $I \subset \R$ we have
    $\alpha_1^*(I) \le \alpha_2^*(\openBall_r(I))$ (open
    $r$-neighbourhood);
  \item
    \label{pointwise.c}
    for all finite sets $F \subset \R$ we have
    $\alpha_1^*(F) \le \alpha_2^*(\closedBall_r(F))$;
  \item
    \label{pointwise.d}
    for all $\lambda \in \R$ we have
    $\alpha_1(\lambda) \le \alpha_2^r(\lambda)$.
  \end{enumerate}
  If $\alpha_1$ is purely essential ($\alpha_1(\R) =\{0,\infty\}$),
  then the converse statements are also true, i.e., all
  statements~\itemref{pointwise.a}--\itemref{pointwise.d} are
  equivalent.
\end{lemma}
%------------------------------------------------------------------------------
\begin{proof}
  \itemref{pointwise.a}$\Rightarrow$\itemref{pointwise.d}: follows
  from taking the limit $\eps \to 0$ for
  $I=\closedBall_\eps(\lambda)$, note that
  $\closedBall_\eps(I)=\closedBall_{\eps+\eps'}(\lambda) \nearrow
  \closedBall_\eps(\lambda)$ as well).  The same is true for
  \itemref{pointwise.b}; for
  \itemref{pointwise.c}$\Rightarrow$\itemref{pointwise.d} take
  $F=\{\lambda\}$.

  \itemref{pointwise.d}$\Rightarrow$\itemref{pointwise.a}: If
  $\alpha_1(\lambda)\in \{0,\infty\}$ for all $\lambda \in \R$ only,
  then the only non-trivial case is $\alpha_1^*(I)=\infty$.  Here, we
  have $\lambda \in I$ with $\alpha_1(\lambda)=\infty$ ($\alpha_1$
  does not take any finite value), then (by assumption),
  $\infty\le\alpha_1(\lambda)\le \alpha_2^*(\closedBall_r(\lambda))\ge
  \alpha_2^*(\closedBall_r(I))$, hence~\itemref{pointwise.a} holds.
  The other implications follow similarly.
\end{proof}
%------------------------------------------------------------------------------
%------------------------------------------------------------------------------
\begin{remark}[a minimal counterexample]
  \label{rem:pointwise}
  The converse statement of \Lem{pointwise} is in general false:
  assume that
  \begin{align*}
    \alpha_1(\lambda)=
    \begin{cases}
      2, & \lambda \in \{2,4\},\\
      0, & \text{otherwise},
    \end{cases}
    \quadtext{and}
    \alpha_2(\lambda)=
    \begin{cases}
      1, & \lambda \in \{1,3,5,\lambda_0\},\\
      0, & \text{otherwise},
    \end{cases}
  \end{align*}
  where $\lambda_0>1$.  Now, it is easily seen that
  \Lemenum{pointwise}{pointwise.d} holds for all $r \ge 1$, hence the
  infimum (minimum) is $r=1$: We only have to check
  $\lambda \in \supp \alpha_1$, and
  $\alpha_1(\lambda)=\alpha_2^r(\lambda)=2$ for $\lambda \in \{2,4\}$
  once $r \ge 1$.  For $\acute \delta(\alpha_1,\alpha_2)=1$, note that
  for any open interval $I$, say $I=(2-\eps,4+\eps)$, we have
  $\alpha_1^*(I)=4$, but $\alpha_2^*(\closedBall_r(I))=4$ if and only
  if $4+\eps+r \ge \lambda_0$; taking all such $\eps>0$ gives
  $4+r \ge \lambda_0$, and hence the infimum of all such $r>0$ is
  $\acute \delta(\alpha_1,\alpha_2)=\lambda_0-4$.  For $\lambda_0>5$,
  we have found a counterexample.  The same argument applies also for
  the finite set characterisation: here, we also have
  $\dCmffiN(\alpha_1,\alpha_2)=\lambda_0-4$.
\end{remark}
%------------------------------------------------------------------------------

%------------------------------------------------------------------------------
\begin{remark}[relation between Hausdorff and \LevyProkhorov
  distance]
  \label{rem:pointwise2}
  The previous lemma allows a more formal way to prove
  \Prpenum{dist.cmf.haus}{dist.cmf.haus.c}: we can reformulate the set
  of $\eps>0$ in~\eqref{eq:eq.hausdorff} by the pointwise inequality
  \begin{align*}
    \1_{\Sigma_1} \le \1_{\closedBall_\eps(\Sigma_2)}
    \quadtext{and}
    \1_{\Sigma_2} \le \1_{\closedBall_\eps(\Sigma_1)}.
  \end{align*}
  Define $\wt \alpha_n = \infty \1_{\Sigma_n}$ (with the convention
  that $\infty \times 0=0$), then $\wt \alpha_n$ are crude
  multiplicity functions for $n=1,2$.  Moreover,
  $\wt \alpha_n^\eps=\infty \1_{\closedBall_\eps (\Sigma_n)}$.  If
  $\alpha_1$ and $\alpha_2$ are purely essential crude multiplicity,
  then the pointwise characterisation \Lemenum{pointwise}{pointwise.d}
  is equivalent with \itemref{pointwise.a}; interchanging $\alpha_1$
  and $\alpha_2$ one can then see that Hausdorff distance of
  $\Sigma_1$ and $\Sigma_2$ is the same as the \LevyProkhorov
  distance $\delta(\alpha_1,\alpha_2)$.
\end{remark}
%------------------------------------------------------------------------------

%-----------------------------------------------------------------------
%
% 3333
\section{Unitary distance and crude multiplicity functions}
\label{sec:duni.cmf}
%
%-----------------------------------------------------------------------
In this section, we review and adopt briefly the results of Azoff and
Davis~\cite{azoff-davis:84} on the equality of the unitary distance of
two self-adjoint operators with the \LevyProkhorov distance of their
crude multiplicity functions needed for our analysis.  A proof for
normal operators (with some modifications if discrete spectrum appears
(see \Subsec{related.works}) is given in~\cite{davidson:86}.

%--------------------------------------------------------------------
\subsection{Unitary distance and crude multiplicity functions for
  operators in the same space}
\label{ssec:unitar.dist}
%--------------------------------------------------------------------

Let $R$, $R_1$, $R_2 \in \Lin \HS$ be bounded and self-adjoint
operators acting in a separable Hilbert space $\HS$.

The group of unitaries $\Unitary \HS$ acts isometrically on the space of
bounded operators $\Lin \HS$ via $U_*(R):=U R U^*$ The corresponding
\emph{unitary orbit} is denoted by
\begin{equation}
  \label{eq:orbit}
  \UnitOrbit R :=\set{U_*(R):=U R U^*}{U \in \Unitary \HS}.
\end{equation}

Let $\alpha$, $\alpha_1$ and $\alpha_2$ be the crude multiplicity
functions of $R$, $R_1$ and $R_2$, respectively (see \Subsec{cmf}).

In order to refer to the corresponding operators of the crude
multiplicity function, we write
\begin{equation}
  \label{eq:def.dcmf}
  \dCmf(R_1,R_2) := \delta(\alpha_1,\alpha_2)
\end{equation}
for the distance of the crude multiplicity functions $\alpha_n$
associated with $R_n$ ($n=1,2$).

We denote the \emph{distance} between the unitary orbits by
\begin{equation}
  \label{eq:dist-unit-orbit}
  \dist\bigl(\UnitOrbit {R_1},\UnitOrbit{R_2}\bigr)
  := \inf\bigset{\norm[\Lin \HS]{U_1R_1U_1^*-U_2R_2U_2^*}}
  {U_1,U_2 \in \Unitary\HS}.
\end{equation}
As the group $\Unitary \HS$ acts isometrically on $\Lin \HS$, we have
\begin{align*}
  \dist(\UnitOrbit{R_1},\UnitOrbit{R_2})=\dist(R_1,\UnitOrbit{R_2}),
\end{align*}
and also
\begin{align*}
  \dHauS(\UnitOrbit{R_1},\UnitOrbit{R_2})
  =\sup_{U_1 \in \Unitary \HS}\dist(U_1R_1U_1^*,\UnitOrbit{R_2})
  =\dist(R_1,\UnitOrbit{R_2}),
\end{align*}
where $\dHauS$ has been defined in~\eqref{eq:hausdorff0}.
  Interchanging
the roles of $R_1$ and $R_2$, one then sees
\begin{equation}
  \dist\bigl(\UnitOrbit {R_1},\UnitOrbit{R_2}\bigr)
  = \dHaus\bigl(\UnitOrbit {R_1},\UnitOrbit{R_2}\bigr),
\end{equation}
where the Hausdorff distance is defined in~\eqref{eq:hausdorff}.
The main result of Azoff and Davis~\cite[Thm.~1.3]{azoff-davis:84} is
now the following:
% --------------------------------------------------------------------------------
\begin{proposition}
  \label{prp:azoff-davis}
  Let $R_n \in \Lin \HS$ ($n=1,2$) be two self-adjoint and bounded
  operators with crude multiplicity functions $\alpha_n$ acting in the
  same Hilbert space $\HS$.  Then we have
  \begin{equation}
    \label{eq:azoff-davis}
    \dist(\UnitOrbit {R_1},\UnitOrbit{R_2})
    =\delta(\alpha_1,\alpha_2).
  \end{equation}
\end{proposition}
%--------------------------------------------------------------------------------
The proof of ``$\ge$'' (\cite[Prp.~2.3]{azoff-davis:84}) is rather
easy; while ``$\le$'' (\cite[Sec.~4]{azoff-davis:84}) needs more
advanced techniques including a generalisation of the marriage theorem
to infinite sets.

%--------------------------------------------------------------------
\subsection{Unitary distance zero}
\label{ssec:duni.zero}
%--------------------------------------------------------------------

It is an interesting fact of operator theory that the unitary orbit is
(in general) \emph{not} closed in the operator topology, i.e., if
$R_2$ is in the closure of $\UnitOrbit {R_1}$, or equivalently,
$\dUni(R,\wt R)=0$, then $\wt R$ is (in general) \emph{not} unitarily
equivalent with $R$ (i.e., there is not always a $U \in \Unitary \HS$
such that $\wt R=URU^*$, but only a sequence of unitary operators
$U_n \in \Unitary \HS$ with
$\norm[\Lin \HS]{U_nR U_n^* - \wt R} \to 0$.  Nevertheless, from
\Thm{azoff-davis-duni} and \Prpenum{char.cmf}{char.cmf.b} we see that
$\wt R$ and $R$ have the same discrete and essential spectrum.  In
particular, $R$ and $\wt R$ might have different (essential) spectral
type: let $\HS_n=\Lsqr{[0,1],\mu_n}$ for $n \in \{1,2,3\}$, where
$\mu_1$ is the Lebesgue measure $\mu_1$ and $\mu_2$ a measure singular
with respect to $\mu_1$ and without atoms (a singular continuous
measure).  Finally, let $(\lambda_k)_{k \in \N}$ be an enumeration of
$[0,1] \cap \Q$ and define
$\mu_3(B)=\sum_{k \in \N, \lambda_k \in B} 2^{-k}$.  Let
$(R_nf_n)(\lambda)=\lambda f_n(\lambda)$ for $f_n \in \HS_n$ be the
multiplication operator.  One can now see that
$\spec{R_n}=\essspec{R_n}=[0,1]$ for $n=1,2,3$.  In particular, by
\Thm{main-1}, $\dUni(R_m,R_n)=0$ for all $m,n \in \{1,2,3\}$, but the
operators $R_n$ are mutually non-unitarily equivalent.  A proof and
more interesting examples of unitarily (non-)equivalent multiplication
operators can be found in~\cite{abrahamse:78}.

The case when $\UnitOrbit R$ is closed (for separable Hilbert spaces
$\HS$) can be characterised by saying that $\spec R$ (as a set, not
multiset) is finite (see~\cite[Prp.~3.5]{azoff-davis:84}), or in terms
of operator algebras, if and only if the $C^*$-algebra generated by
$R$ is finite dimensional, see~\cite[Prp.~2.4]{voiculescu:76}
and~\cite[Sec.~8]{sherman:07}.  In particular, each
$\lambda \in \spec R$ is an isolated eigenvalue (with arbitrary
multiplicity).

% --------------------------------------------------------------------
\subsection{Unitary distance and crude multiplicity functions for
  operators in different spaces}
\label{ssec:unitar.dist.diff}
%--------------------------------------------------------------------

Probably the most canonical way to compare operators acting on
different Hilbert spaces is to allow unitary mappings between the
spaces.

Here and in the sequel we consider two bounded operators
$R_n \in \Lin{\HS_n}$ ($n=1,2$).  Their \emph{unitary distance}
$\dUni(R_1,R_2)$ was defined in~\eqref{eq:uni-dist}.  Recall and note
that
\begin{align}
  \nonumber
  \dUni(R_1,R_2)
  :=& \inf \bigset{\norm[\Lin {\HS_2}]{U R_1U ^*-R_2}}
      {U \in \Unitary{\HS_1,\HS_2}}\\
  \nonumber
  =& \inf  \bigset{\norm[\Lin \HS]{U_1R_1U_1^*-U_2R_2U_2^*}}
     {U_n \in \Unitary{\HS_n,\HS}, \text{$\HS$ Hilbert space}}\\
  \label{eq:duni.unit.dist}
  =& \dist(\UnitOrbit{U_{12}R_1U_{12}^*},\UnitOrbit{R_2})
\end{align}
for any unitary $U_{12} \in \Unitary{\HS_1,\HS_2}$.  The case
$\dUni(R_1,R_2)=0$ is discussed in \Subsec{duni.zero}.

The following fact is easy to see (cf.\
e.g.~\cite[Prp.~2.2]{post-simmer:20b}:
%--------------------------------------------------------------------
\begin{proposition}
  \label{prp:duni.pseudo-metric}
  $\dUni$ is non-negative, symmetric and fulfils the triangle
  inequality.  Moreover, we have $\dUni(R_1^*,R_2^*)=\dUni(R_1,R_2)$.
\end{proposition}
%--------------------------------------------------------------------

Our first main result is now a simple consequence of Azoff and Davis'
result (see \Prp{azoff-davis}) and \eqref{eq:duni.unit.dist}).  Note
also, that the crude multiplicity functions of $R_1$ and $UR_1U^*$
agree for any unitary operator $U \in \Unitary{\HS_1,\HS_2}$.
%--------------------------------------------------------------------
\begin{theorem}[{based on \cite[Thm.~1.3]{azoff-davis:84}}]
  \label{thm:azoff-davis-duni}
  Let $R_n \in \Lin{\HS_n}$ ($n=1,2$) be two bounded and self-adjoint
  operators with related multiplicity function $\alpha_n$, then we
  have
  \begin{equation*}
    \dUni(R_1,R_2)
    = \delta(\alpha_1, \alpha_2)
    \;(=: \dCmf(R_1,R_2)).
  \end{equation*}
\end{theorem}
%--------------------------------------------------------------------
This result also shows that $\dCmf$ (or $\delta$) fulfils the triangle
inequality using \Prp{duni.pseudo-metric}.

We now give an alternative proof of \Thm{main-1} in terms of crude
multiplicity functions:
%-------------------------------------------------------------------------------
\begin{proof}[Proof of \Thm{main-1}]
  The statements follow from \Thm{azoff-davis-duni} and
  \Prp{dist.cmf.haus}.  Note that
  $\dDisc(R_1,R_2)=\dCmffin(\alpha_{R_1},\alpha_{R_2})$, where the
  first is defined in~\eqref{eq:dmult} while the latter
  in~\eqref{eq:delta.cr.mult.d} using again~\eqref{eq:inf.cap}.
\end{proof}
%-------------------------------------------------------------------------------

%-----------------------------------------------------------------------
%
% 4444
\section{Isometric and unitary distance}
\label{sec:diso.duni}
%
%-----------------------------------------------------------------------

In this section, we prove \Thm{main-2}.

%--------------------------------------------------------------------
\subsection{Isometric distance}
\label{ssec:weidmann.dist}
%--------------------------------------------------------------------
The isometric distance $\dIso$ defined in~\eqref{eq:iso-dist}
generalises a concept of Weidmann (\cite[Sec.~9.3]{weidmann:00}) for
operators acting in different Hilbert spaces.  For a detailed
discussion, we refer to~\cite{post-zimmer:22} and references therein.

We show the triangle inequality for $\dIso$ separately.  In the
special case that $0$ is in the essential spectrum of $R_1$ and $R_2$,
it follows already from the triangle inequality for $\dUni$
(\Prp{duni.pseudo-metric}) and \Thm{main-2}:
%----------------------------------------------------------------------------
\begin{proposition}
  \label{prp:diso.pseudo-metric}
  $\dIso$ is non-negative, symmetric and fulfils the triangle
  inequality.  Moreover, we have $\dIso(R_1^*,R_2^*)=\dIso(R_1,R_2)$.
\end{proposition}
%---------------------------------------------------------------------------
\begin{proof}
  Positivity, symmetry and $\dIso(R_1,R_1)=0$ are clear by definition.
  Only the triangle inequality remains to be shown: Given
  $R_n \in \Lin{\HS_n}$ ($n=1,2,3$) and $\eps>0$ we have Hilbert
  spaces $\HS_{12}$ and $\HS_{23}$ and isometries
  \begin{align*}
    \map{\iota_{12}}{\HS_1}{\HS_{12}}, \quad
    \map{\iota_{21}}{\HS_2}{\HS_{12}}, \quad
    \map{\iota_{23}}{\HS_2}{\HS_{23}},
    \quad \map{\iota_{32}}{\HS_3}{\HS_{23}},
  \end{align*}
  by the definition of $\dIso(R_1,R_2)$ and $\dIso(R_2,R_3)$ such
  that
  \begin{equation*}
    \norm[\Lin{\HS_{12}}]{\iota_{12} R_1 \iota_{12}^*-\iota_{21} R_2 \iota_{21}^*}
    + \norm[\Lin{\HS_{23}}]{\iota_{23} R_2 \iota_{23}^*-\iota_{32} R_3 \iota_{32}^*}
    <  \dIso(R_1,R_2)+\dIso(R_2,R_3)  + \eps.
  \end{equation*}
  The idea is to construct $(\HS, \iota_{13}, \iota_{31})$ such
  that
  \begin{equation*}
    \norm[\Lin \HS)]{\iota_{13}R_1 \iota_{13}^*
      -\iota_{31} R_3 \iota_{31}^*}
    \leq  \norm[\Lin{\HS_{12}}]{\iota_{12} R_1 \iota_{12}^*-\iota_{21} R_2 \iota_{21}^*}
    + \norm[\Lin{\HS_{23}}]{\iota_{23} R_2 \iota_{23}^*-\iota_{32} R_3 \iota_{32}^*}.
  \end{equation*}
  This can be done as follows: We split
  $\HS_{12}=\iota_{21}(\HS_2) \oplus \iota_{21}(\HS_2)^\perp$ and write it
  by $f_{12}= \iota_{21} b_{12}+r_{12}$ for an element $f_{12} \in \HS_{12}$.
  Similarly, $\HS_{23}=\iota_{23}(\HS_2) \oplus \iota_{23}(\HS_2)^\perp$ and
  $f_{23}= \iota_{23} b_{23}+r_{23}$ for $f_{23} \in \HS_{23}$.  Now define
  $\HS:=\iota_{21}(\HS_2)^\perp \oplus \HS_2 \oplus \iota_{23}(\HS_2)^\perp$
  and new isometries via
  \begin{align*}
    \map{\wt\iota_{12}&}{\HS_{12}}{\HS},\quad
                        \wt\iota_{12} f_{12}
                        =\wt\iota_{12}(\iota_{21}b_{12}+r_{12}):= (r_{12},b_{12},0)\\
    \map{\wt\iota_{23}&}{\HS_{23}}{\HS},\quad
                  \wt\iota_{23} f_{23}=\wt\iota_{12}(\iota_{23}b_{23}+r_{23}):= (0,b_{23},r_{23}).
  \end{align*}
  Using Pythagoras we obtain
  \begin{align*}
    \normsqr[\HS]{\wt\iota_{12}f_{12}}
    = \normsqr[\HS_{12}]{r_{12}}+\normsqr[\HS_2]{b_{12}}
    = \normsqr[\HS_{12}]{r_{12}}+\normsqr[\HS_{12}]{\iota_{12} b_{12}}
    =\normsqr[\HS_{12}]{f_{12}}.
  \end{align*}
  Similarly, we can show that $\wt\iota_{23}$ is an isometry.  Finally, we
  set
  \begin{align*}
    \map{\iota_{13}&}{\HS_1}{\HS}, \quad
                     \iota_{13}:= \wt\iota_{12}\iota_{12}
    &\quadtext{and}
      \map{\iota_{31}&}{\HS_3}{\HS}, \quad
    \iota_{31}:= \wt\iota_{23}\iota_{32}.
  \end{align*}
  The diagram below gives an overview of this construction, and we
  show that it is commutative:
  
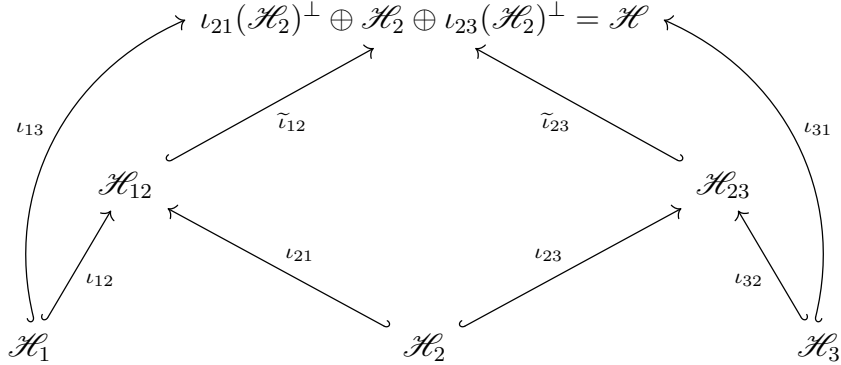
\begin{align*}
    \begin{tikzcd}[column sep=tiny, row sep=huge,ampersand replacement=\&]
      \& \& \& \iota_{21}(\HS_2)^\perp \oplus \HS_2 \oplus \iota_{23}(\HS_2)^\perp=\HS \& \& \\
     \& \&\HS_{12} \ar[overlay, ur,swap, "\wt\iota_{12}", hook]
       \& \&\HS_{23} \ar[overlay, ul, "\wt\iota_{23}", hook'] \&\\
       \&\HS_1 \ar[overlay, uurr, "\iota_{13}", hook, bend left=40] %
       \ar[overlay, ur,swap, "\iota_{12}", hook]
       \& \&\HS_2 \ar[overlay, ul, swap, "\iota_{21}", hook'] \ar[overlay, ur, "\iota_{23}", hook]
       \& \&\HS_3 \ar[overlay, ul, "\iota_{32}", hook'] %
       \ar[overlay, uull,swap, "\iota_{31}", hook', bend right=40]
     \end{tikzcd}
 \end{align*}
  We now have $\wt\iota_{12}\iota_{21}b=(0,b,0)=\wt\iota_{23}\iota_{23}b$ for all
  $b \in \HS_2$.  In particular, we have commutation on the inner
  diagram.  Moreover, we now estimate
  \begin{align*}
    \norm%[\Lin \HS]%
    {\iota_{13}R_1 \iota_{13}^*-\iota_{31} R_3 \iota_{31}^*}
    &=\norm%[\Lin \HS]%
      {\wt\iota_{12}\iota_{12} R_1 \iota_{12}^*\wt\iota_{12}^* %
      -\wt\iota_{23}\iota_{32} R_3 \iota_{32}^*\wt\iota_{23}^*}\\
    &=\norm%[\Lin \HS]%
      {\wt\iota_{12}\iota_{12} R_1 \iota_{12}^*\wt\iota_{12}^* %
      -\wt\iota_{12}\iota_{21} R_2 \iota_{21}^*\wt\iota_{12}^*
      +\wt\iota_{12}\iota_{21} R_2 \iota_{21}^*\wt\iota_{12}^*
      -\wt\iota_{23}\iota_{32} R_3 \iota_{32}^*\wt\iota_{23}^*}\\
    % &=\norm%[\Lin \HS]%
    %   {\wt\iota_{12}\iota_{12} R_1 \iota_{12}^*\wt\iota_{12}^*
    %   -\wt\iota_{12}\iota_{21} R_2 \iota_{21}^*\wt\iota_{12}^*
    %   +	\wt\iota_{23}\iota_{23} R_2 \iota_{23}^*\wt\iota_{23}^*
    %   -\wt\iota_{23}\iota_{32} R_3 \iota_{32}^*\wt\iota_{23}^*}\\
    &\leq \norm[\Lin{\HS_1}]{\iota_{12} R_1 \iota_{12}^*-\iota_{21} R_2 \iota_{21}^*}
      + \norm[\Lin{\HS_2}]{\iota_{23} R_2 \iota_{23}^*-\iota_{32} R_3 \iota_{32}^*}\\
    & < \dIso(R_1,R_2) + \dIso(R_2,R_3) + \eps.
  \end{align*}
  Thus for any $\eps>0$ we obtain
  $\dIso(R_1,R_3) < \dIso(R_1,R_2) + \dIso(R_2,R_3) + \eps$, and hence
  the triangle inequality.  The last assertion on the adjoints follows
  from
  \begin{equation*}
    D_{\HS,\iota_1,\iota_2}(R_1,R_2)^*
    = \iota_1R_1^* \iota_1^* -\iota_2R_2^*\iota_2^*
    =     D_{\HS,\iota_1,\iota_2}(R_1^*,R_2^*)
%    \qedhere
  \end{equation*}
  and $\norm D=\norm{D^*}$.
\end{proof}
%----------------------------------------------------------------------

%----------------------------------------------------------------------------
\begin{remark}[Relation of the isometric distance with a
  Gromov-Hausdorff distance]
  \label{rem:gh-dist}
  If we consider the graphs $\graph R_n$ of $R_n$ ($n=1,2$) as closed
  subspaces of $\HS_n \times \HS_n$, then we can consider their
  Gromov-Hausdorff distance
  \begin{align*}
    \dGH(\graph R_1,\graph R_2)
    = \inf \set{\dHaus(I_1(\graph R_1), I_2(\graph R_2)}
    {\text{$\map {I_n}{\HS_n \times \HS_n}\wt \HS$}},
  \end{align*}
  where the infimum is taken over all isometries $I_n$ into a third
  Hilbert space $\wt \HS$.  This distance is too rough, as e.g.\
  $R_1=\id$ and $R_2=-\id$ on $\HS=\HS_1=\HS_2$ would have distance
  $0$: choose the isometry
  $\map {I_1} {\HS \times \HS}{\wt \HS=\HS \times \HS}$ with
  $I(f,g)=(f,-g)$ and $I_2=\id$.  We somehow have to restrict to
  isometries of \emph{diagonal type}
  $I_n(f_n,g_n)=(\iota_n f_n,\iota_n g_n)$ with isometries
  $\map {\iota_n}{\HS_n}\HS$ and define hence a modified
  Gromov-Hausdorff distance larger than the one defined above.
  Moreover, we were only able to give a lower bound on
  $\dIso(R_1,R_2)$ in terms of the modified Gromov-Hausdorff distance.
  There is also a relation with the gap distance defined by Kato
  (see~\cite[Sec.~IV.2]{kato:66}).  We will treat this and related
  questions in a forthcoming publication.
\end{remark}
%----------------------------------------------------------------------------

%----------------------------------------------------------------------
\subsection{Relation between the unitary and isometric distance}
\label{ssec:duni-diso}
%----------------------------------------------------------------------

The inequality
\begin{equation}
  \label{eq:trivial-duni-diso}
  \dUni(R_1,R_2)
  \ge \dIso(R_1,R_2)
\end{equation}
for $R_n\in \Lin{\HS_n}$ ($n=1,2$) follows immediately from the
definition, as the competition set for $\dUni$ is smaller than the one
for $\dIso$: any operator $U_{12} \in \Unitary{\HS_1,\HS_2}$ induces
the isometries $\iota_1=U_{12}$ and $\iota_2=\id$ into the common
Hilbert space $\HS=\HS_2$.

%-----------------------------------------------------------------------------
\begin{lemma}
  \label{lem:spectra.of.isometries}
  Given a bounded operator $R_1$ on a Hilbert space $\HS_1$ and an
  isometry $\map{\iota_1}{\HS_1}{\HS}$, we have
  \begin{align}
    \label{eq:spec_iso}
    \spec{\iota_1 R_1 \iota_1^*}
    =
    \begin{cases}
      \spec{R_1}, & \text{$\iota_1$ surjective,} \\
      \spec{R_1}\cup \{0\}, & \text{$\iota_1$ not surjective.}
    \end{cases}
  \end{align}
\end{lemma}
%-----------------------------------------------------------------------------
\begin{proof}
  Let $\map{U_1}{\HS_1}{\iota_1(\HS_1)}$ be defined by
  $U_1 f_1=\iota_1f_1 \in \HS$, then clearly $U_1$ is unitary, and
  $\wt R_1:=U_1 R_1 U_1^*$ is unitarily equivalent with $R_1$ having
  the same spectrum by~\eqref{eq:spec.uni.eq}.  If $\iota_1$ is
  surjective, then $U_1=\iota_1$ and hence $R_1$ and
  $\iota_1 R_1 \iota_1^*$ have the same spectra.

  If $\iota_1$ is not surjective, then
  $\HS_1^\perp := \iota_1(\HS_1)^\perp$ is a non-trivial subspace of
  $\HS$.  Moreover, $\iota_1 R_1 \iota_1^*=\wt R_1 \oplus 0$ is a
  diagonal $2\times2$-block matrix operator with respect to the
  splitting $\HS=\iota_1(\HS_1) \oplus \HS_1^\perp$.  The spectra of
  such operators is the union of the spectra of the diagonal
  operators, i.e.,
  $\spec {\iota_1 R_1 \iota_1^*}=\spec {\wt R_1} \cup \{0\}$ and we are done.
\end{proof}
%-----------------------------------------------------------------------------

When examining the isometric distance one can ask whether it is enough
to consider only unitary maps instead of isometries.  It turns out
that $0$ being in both essential spectra of $R_1$ and $R_2$
(i.e.~\eqref{eq:cond.ess.spec}) is a sufficient condition.  Note that
$0 \in \essspec {R_n}$ is always fulfilled if $R_n$ is the resolvent
of an unbounded operator.

The idea is to look how crude multiplicity functions behave under
composing the operator with isometries.  Recall the definition of
$\dCmf$ in~\eqref{eq:def.dcmf}.
%---------------------------------------------------------------------------------
\begin{lemma}
  \label{lem:diso_vs_delta}
  Let $R_n\in \Lin{\HS_n}$ be a self-adjoint operator and let
  $\iota_n \in \Iso{\HS_n,\HS}$ be an isometry for $n=1,2$.  If
  $0 \in \essspec{R_n}$, then
  $\alpha_{R_n}=\alpha_{\iota_n R_n \iota_n^*}$.
\end{lemma}
%---------------------------------------------------------------------------------
\begin{proof}
  Denote by $\map{U_n}{\HS_n}{\iota_n(\HS_n)}$ with
  $U_n f_n=\iota_n f_n$ the unitary operator obtained from $\iota_n$
  by restricting the target space.  Moreover, let
  $\wt R_n = U_n R_n U_n^*$ be the operator
  acting on $\iota_n(\HS_n)$ and unitarily equivalent with $R_n$.

  Let $I$ be an open subset of $\R$; we then have
  \begin{align*}
    \1_I(\iota_n R_n \iota_n^*)
    = \1_I(\wt R_n \oplus 0)
    = \1_I(\wt R_n) \oplus \1_I(0).
  \end{align*}
  If $0 \notin I$, then $\1_I(0)=0$ and hence
  $\1_I(\iota_nR_n\iota_n^*)=\1_I(\wt R_n) \oplus 0$.  In particular,
  $\rank \1_I(\iota_nR_n\iota_n^*)=\rank \1_I(\wt R_n)=\rank
  \1_I(R_n)$.  If $0 \in I$, then
  $\rank \1_I(\iota_nR_n\iota_n^*)=\infty$ and
  $\rank \1_I(R_n)=\infty$ as $0$ is in the essential spectrum (using
  also \Lem{spectra.of.isometries}).

  In particular, $\rank \1_I(R_n)=\rank \1_I(\iota_nR_n\iota_n^*)$ for
  any open set $I$, i.e., the crude multiplicity function of $R_n$ and
  $\iota_nR_n\iota_n^*$ agree, see \Prp{char.cmf}.
\end{proof}
%---------------------------------------------------------------------------------
%---------------------------------------------------------------------------------
\begin{corollary}
  \label{cor:diso_vs_delta}
  Let $R_n\in \Lin{\HS_n}$ be self-adjoint operator and let
  $\iota_n \in \Iso{\HS_n,\HS}$ ($n=1,2$).
  If~\eqref{eq:cond.ess.spec} is fulfilled then
  \begin{equation*}
    % \delta(\alpha_1,\alpha_2) = \delta (\wt \alpha_1,\wt \alpha_2).
    \dCmf(R_1,R_2)
    = \dCmf(\iota_1 R_1 \iota_1^*,\iota_2 R_2\iota_2^*).
  \end{equation*}
\end{corollary}
%---------------------------------------------------------------------------------

We are now able to prove \Thm{main-2}:
%---------------------------------------------------------------------------------
\begin{proof}[Proof of \Thm{main-2}]
  The trivial inequality ``$\ge$'' was stated
  in~\eqref{eq:trivial-duni-diso}.  For the inequality ``$\le$'' we
  argue as follows: we have $\dUni(R_1,R_2)= \dCmf(R_1,R_2)$ by
  \Thm{azoff-davis-duni}; and the latter equals
  $\dCmf(\iota_1 R_1 \iota_1^*,\iota_2 R_2\iota_2^*)$ by
  \Cor{diso_vs_delta}.  Applying again \Thm{azoff-davis-duni} gives
  $\dCmf(\iota_1 R_1 \iota_1^*,\iota_2
  R_2\iota_2^*)=\dUni(\iota_1R_1\iota_1^*,\iota_2R_2\iota_2^*)$.  As
  $\dUni(\iota_1R_1\iota_1^*,\iota_2R_2\iota_2^*)\le
  \norm{D_{\HS,\iota_1,\iota_2}(R_1,R_2)}$ we obtain the desired inequality by
  taking the infimum over all $\iota_n$ and $\HS$.
\end{proof}
%---------------------------------------------------------------------------------
We comment on the condition~\eqref{eq:cond.ess.spec} on the essential
spectra in \Subsec{counterex}, especially in
\ExS{duni.eq.diso.counterex}{diso.0.duni.not}. %{duni.eq.diso.counterex'}

%-------------------------------------------------------------------
%
% 5555
\section{Quasi-unitary distance and its relation with the isometric
  distance}
\label{sec:dque}
%
%-------------------------------------------------------------------

The aim of this subsection is to analyse the last distance between two
operators and its relation with the other distances already defined.

%----------------------------------------------------------------------
\subsection{Quasi-unitary distance}
\label{ssec:qu.dist.def}
%----------------------------------------------------------------------
The next idea to compare these operators is to use so-called
\emph{quasi-unitary operators}, i.e., operators that are unitary up to
an ``error'' $\delta$ in a certain sense.  The quasi-unitary distance
$\dQUE$ is defined in~\eqref{eq:que-dist}.  We first show that $\dQUE$
fulfils some properties of a metric:
%------------------------------------------------------------------
\begin{lemma}
  \label{lem:dque.metric}
  The distance $\dQUE$ fulfils $\dQUE(R_1,R_1)=0$,
  $\dQUE(R_1,R_2) \ge 0$ and $\dQUE(R_1,R_2) = \dQUE(R_2,R_1)$.
  Moreover, a weaker version of the triangle inequality holds, namely
  \begin{equation*}
    \dQUE(R_1,R_3)
    \le \sqrt 3\bigl(\dQUE(R_1,R_2) + \dQUE(R_2,R_3)\bigr).
  \end{equation*}
\end{lemma}
%------------------------------------------------------------------
\begin{proof}
  The first statement is immediate by choosing $J=\id_{\HS_1}$ so that
  $\delta_{\id_{\HS_1}}(R_1,R_1)=0$; the second statement is trivial.
  The symmetry follows from the fact that
  $\delta_{J^*}(R_2,R_1)=\delta_J(R_1,R_2)$.

  For the (generalised) triangle inequality we argue indirectly: we
  use the triangle inequality for $\dIso$ (\Prp{diso.pseudo-metric})
  and $\dQUE \le \sqrt 3 \dIso$ (\Thm{main-3}).
\end{proof}
%------------------------------------------------------------------

Let us first comment on
this version of quasi-unitary distance here, where we use
\begin{align*}
  \norm{R_1^*(\id-J^*J)R_1}& \quadtext{and}
  \norm{R_2^*(\id-JJ^*)R_2}
  \qquad\text{instead of}\\
  \norm{(\id-J^*J)R_1}& \quadtext{and}
  \norm{(\id-JJ^*)R_2}
\end{align*}
in previous publications such
as~\cite{post-zimmer:22,post-simmer:20b}:
%------------------------------------------------------------------
\begin{remarks}[Why the new version of quasi-unitary equivalence?]
  \label{rem:new-que}
  \indent
  \begin{itemize}
  \item We changed the definition of quasi-unitary equivalence here as
    the new one is better adapted to the Weidmann setting; actually,
    the associated distances $\dQUE$ and $\dIso$ are equivalent; hence
    we do not have the phenomenon of \emph{loss of convergence speed}
    as we observed in~\cite{post-zimmer:22}; see also
    \Rem{isometries.wrong.side}

  % \item Also, it seems that the new version leads to better estimates
  %   in concrete examples, especially using quadratic forms.

  \item We can show the generalised triangle inequality directly with
    $\sqrt 2$ instead of $\sqrt 3$ as above, see the forthcoming
    publication~\cite{post-zimmer:pre24c}.  It is not clear to us
    whether $\dQUE$ actually fulfils the (proper) triangle inequality
    (without extra factors such as $\sqrt 2$).  At least on multiples
    of the identity, it is a metric, see \Prp{dque.mult.id}.

  \item Other variants such as more general $J$ instead of only
    \emph{contractions} $\map J {\HS_1}{\HS_2}$ will be treated
    in~\cite{post-zimmer:pre24c}.  Finally, in some versions of
    quasi-unitary equivalences, we used $\map{J'}{\HS_2}{\HS_1}$
    instead of the adjoint $J^*$.  It is shown
    in~\cite[Lem.~2.11~(f)]{post-simmer:20b} that one may choose
    $J'=J^*$ without loss of generality, see
    also~\cite{post-zimmer:pre24c} for further discussions.
  \end{itemize}
\end{remarks}
%------------------------------------------------------------------

%------------------------------------------------------------------
\begin{lemma}
  \label{lem:uni.eq.dense}
  If $\ran R_n$ is dense in $\HS_n$ for $n=1,2$, and if
  $\delta_J(R_1,R_2)=0$ (see~\eqref{eq:delta_J}), then
  $\map J {\HS_1}{\HS_2}$ is unitary and $R_2=JR_1J^*$, i.e., $R_1$ and
  $R_2$ are unitarily equivalent.
\end{lemma}
%------------------------------------------------------------------
\begin{proof}
  If $\ran R_n$ is dense in $\HS_n$, then we can define a densely
  defined operator $S_n$ with $\dom S_n = \ran R_n$ by taking the
  inverse of $\map {R_n}{(\ker R_n)^\perp}{\HS_n}$.  It follows that
  $R_nS_n=\id_{\ran R_n}$.  Moreover,
  $\norm{R_1^*(\id_{\HS_1}-J^*J)R_1}=0$ implies
  \begin{align*}
    0=R_1^*(\id_{\HS_1}-J^*J)R_1, \quadtext{i.e.,}
    0=R_1^*(\id_{\HS_1}-J^*J)R_1S_1=R_1^*(\id_{\HS_1}-J^*J).
  \end{align*}
  As $\ker R_1^*=(\ran R_1)^\perp=\{0\}$, we have $\id_{\HS_1}-J^*J=0$
  on the dense subset $\ran R_1$.  As $\id_{\HS_1}-J^*J$ is bounded we
  have $J^*J=\id_{\HS_1}$.  Similarly, we see that $\id_{\HS_2}=JJ^*$,
  i.e., $J$ is unitary.  From $\norm{R_2J-JR_1}=0$ we conclude
  $R_2=JR_1J^*$, hence the claim.
\end{proof}
%------------------------------------------------------------------

%-----------------------------------------------------------------------------
\subsection{Isometric and QUE-distance}
\label{ssec: weid.vs.que}
%-----------------------------------------------------------------------------
We prove in this section \Thm{main-3}, i.e., that the isometric and
QUE-distance are equivalent.  Before we show this we need some
notation and preparatory results.  We closely
follow~\cite{post-zimmer:22}; for convenience of the reader, we repeat
some arguments here.  Assume that $\HS_n$ is a Hilbert space, $R_n$ an
operator on $\HS_n$ and $\map{\iota_n}{\HS_n}\HS$ an isometry (i.e.,
$\iota_n^*\iota_n=\id_{\HS_n}$) into a third Hilbert space $\HS$ for
$n=1,2$.  We define the projection onto the ranges of the isometry
$\map{\iota_n}{\HS_n}\HS$ as
\begin{equation}
  \label{eq:def.proj}
  \map{P_n}{\HS}{\HS}\quad P_n:= \iota_n \iota_n^*
\end{equation}
for $n=1,2$.  Moreover we write $P_n^\perp:= \id - P_n$. Then one can
easily see that
\begin{subequations}
  \label{eq:formelsammlung}
  \begin{align}
    \label{eq:Ps_und_Iotas}
    P_n\iota_n =
    &\iota_n, \qquad
      \iota_n^*P_n=\iota_n^*,
    \qquad
    P^\perp_n \iota_n
    = 0, \qquad
      \iota_n ^*P^\perp_n  = 0,\\
    \label{eq:dn_in_3teile}
    D:=
    &D_{\HS,\iota_1,\iota_2}(R_1,R_2)
    = \iota_1 R_1 \iota_1^*-\iota_2 R_2 \iota_2^*
    = P_2 D P_1 + P_2^\perp D P_1 + P_2 D P_1^\perp.
  \end{align}
\end{subequations}
Note that $P_2^\perp DP_1^\perp=0$, so the last term only has three
summands.

In order to compare the isometric distance with the quasi-unitary
distance one has to define an identification operator
$\map J{\HS_1}{\HS_2}$.  A canonical way is setting
$J:=\iota_2^*\iota_1$. The next proposition gives an overview how to
express the three summands in~\eqref{eq:dn_in_3teile} in terms of $J$
and the isometries.
%-----------------------------------------------------------------------------
\begin{proposition}
  \label{prp:hilfsmittel.iso.vs.weid}
  Let $\HS$, $\HS_1$, $\HS_2$ be Hilbert spaces.  Moreover, let
  $\iota_n \in \Iso{\HS_n,\HS}$ be isometries and
  $R_n \in \Lin{\HS_n}$ bounded operators for $n=1,2$.
  $J=\iota_2^*\iota_1$ the following holds:
  \begin{subequations}
    \label{eq:formelsammlung2}
    \begin{align}
      \label{eq:D1}
      P_2 D P_1 & = \iota_2(JR_1-R_2J)\iota_1^*\\
      \label{eq:D2}
      P_2^\perp DP_1 &= (\iota_1-\iota_2 J)R_1\iota_1^*\\
      \label{eq:D3}
      P_2 D^*P_1^\perp &= -\iota_2R_2^*(\iota_2^*-J\iota_1^*).
    \end{align}
    Moreover, we have
    \begin{align}
      \label{eq:somekindofroots}
      (P_2^\perp D P_1)^*(P_2^\perp D P_1)
      % &= \iota_1R_1^*(\iota_1^*-J^*\iota_2^*)(\iota_1-\iota_2 J)R_1\iota_1^*
      &= \iota_1R_1^*(\id-J^*J)R_1\iota_1^* \quad\text{and}\\
      \label{eq:somekindofroots'}
      (P_2 D^* P_1^\perp)(P_2 D^* P_1^\perp)^*
      &= \iota_2R_2^*(\id-JJ^*)R_2\iota_2^*.
    \end{align}
    % \begin{equation}
    %   (\iota_1-\iota_2 J)^*(\iota_1-\iota_2 J)=\id_{\HS_1}-J^*J
    %   \quadtext{and}
    %   (\iota_2^*-J\iota_1^*)(\iota_2^*-J\iota_1^*)^*=\id_{\HS_2}-JJ^*.
    % \end{equation}
    Finally, for the operator norms we have
    \begin{align}
      \label{eq:que.iso.norm1}
      \norm[\Lin \HS]{P_2DP_1}
      &=\norm[\Lin{\HS_1}]{JR_1-R_2J},\\
      \label{eq:que-mod}
      \normsqr[\Lin \HS]{P_2^\perp DP_1}
      &=\norm[\Lin{\HS_1}]{R_1^*(\id-J^*J)R_1}\quad\text{and}\\
      \label{eq:que-mod'}
      \normsqr[\Lin \HS]{P_1^\perp D P_2}%{P_2 D^*P_1^\perp}
      &=\norm[\Lin{\HS_2}]{R_2(\id-JJ^*)R_2^*}.
        \intertext{Similar results hold with the indices $1$ and $2$ exchanged
        , e.g.}
      \label{eq:que.iso.norm1'}
        \norm[\Lin \HS]{P_1DP_2}
      &=\norm[\Lin{\HS_2}]{R_1J^*-J^*R_2}
       =\norm[\Lin{\HS_2}]{JR_1^*-R_2^*J} \quad\text{and}\\
      \label{eq:que.iso.norm2'}
        \norm[\Lin \HS]{P_2DP_1^\perp}
      &=\norm[\Lin \HS]{P_1^\perp D^* P_2}
       =\norm[\Lin{\HS_2}]{R_2(\id-JJ^*)R_2^*}.
    \end{align}
  \end{subequations}
\end{proposition}
%-----------------------------------------------------------------------------
\begin{remark}
  \label{rem:isometries.wrong.side}
  If we simply multiply~\eqref{eq:D2} by $\iota_1^*$ from the left and
  $\iota_1$ from the right, then we obtain
  \begin{align*}
    \iota_1^*P_2^\perp D P_1\iota_1
    =(\id-J^*J)R_1.
  \end{align*}
  Taking the operator norm, we only get
  $\norm{(\id-J^*J)R_1} =\norm{\iota_1^*P_2^\perp D P_1\iota_1} \le
  \norm{P_2^\perp D P_1}$, but we (also) need the inequality ``$\ge$''
  which does not hold in general here. Note that if the isometry and
  co-isometry are on the other sides, we have equality of the norms,
  see~\eqref{eq:D1}, \eqref{eq:que.iso.norm1}
  and~\eqref{eq:iso.coiso}.  This is why we pass to the more
  complicated expressions in~\eqref{eq:somekindofroots}
  and~\eqref{eq:somekindofroots'}, where the isometries and
  co-isometries are on the ``good'' side.
\end{remark}
%-----------------------------------------------------------------------------
\begin{proof}[Proof of \Prp{hilfsmittel.iso.vs.weid}]
  For the first equation we simply multiply the given expressions by
  $\iota_k^*\iota_k= \id_{\HS_k}$ ($k=1,2$).
  \begin{equation*}
    P_2 D P_1  =\iota_2\iota_2^*(\iota_1 R_1 \iota_1^*-\iota_2 R_2 \iota_2^*)\iota_1\iota_1^*
    = \iota_2(JR_1-R_2J)\iota_1^*
  \end{equation*}
  To show~\eqref{eq:D2} we use~\eqref{eq:Ps_und_Iotas} and obtain
  \begin{equation*}
    P_2^\perp DP_1
    =(\id_\HS-\iota_2\iota_2^*)\iota_1 R_1 \iota_1^*
    = (\iota_1-\iota_2 J)R_1\iota_1^*.
  \end{equation*}
  The last equation~\eqref{eq:D3} can be shown similarly.  The
  equalities in~\eqref{eq:somekindofroots} follow from
  \begin{align*}
    (\iota_1^*-J^*\iota_2^*)(\iota_1-\iota_2 J)
    =\iota_1^*\iota_1-J^*\iota_2^*\iota_1-\iota_1^*\iota_2 J+J\iota_2^*\iota_2J
    =\id-J^*J
  \end{align*}
  and similarly for the second one.  For the first norm equality note
  that
  \begin{align}
    \label{eq:iso.coiso}
    \norm{\iota_2A\iota_1^*}=\norm A
  \end{align}
  for any $A \in \Lin{\HS_1,\HS_2}$.  The last two norm equalities, we
  apply the $C^*$-norm equality $\norm{A^*A}=\norm{AA^*}=\normsqr A$
  to~\eqref{eq:somekindofroots} and~\eqref{eq:somekindofroots'}.

  For~\eqref{eq:que.iso.norm1'} and similar assertions note that
  $J^*=\iota_1^*\iota_2$ (i.e., interchanging the indices $1$ and $2$
  interchanges $J$ and $J^*$) and that
  \begin{align*}
    \norm[\Lin{\HS_2}]{JR_1^*-R_2^*J}
    =\norm[\Lin{\HS_1}]{J^*R_1-R_2J^*}
    =\norm[\Lin\HS]{P_1DP_2}
  \end{align*}
  by~\eqref{eq:que.iso.norm1} with interchanged indices.
\end{proof}
%-----------------------------------------------------------------------------

We now prove our second main result:
%-----------------------------------------------------------------------------
\begin{proof}[Proof of \Thm{main-3}]
  \myparagraph{First inequality $\dQUE \le \dIso$:} Let $\eps>0$. Then
  there exists $(\HS,\iota_1,\iota_2)$ such that
  \begin{equation}
    \label{ineq:fromdefinf}
    \norm[\Lin \HS]{D}
    :=\norm[\Lin \HS]{D_{\HS,\iota_1,\iota_2}(R_1,R_2)}< \dIso(R_1,R_2)  + \eps.
  \end{equation}
  Now we set $J=\iota_2^*\iota_1$ and calculate the four norms
  appearing in the definition of $\delta_J(R_1,R_2)$
  in~\eqref{eq:delta_J}.  For the first term we conclude
  from~\eqref{eq:que-mod} that
  \begin{align*}
    \norm[\Lin{\HS_1}]{R_1^*(\id-J^*J)R_1}
    &= \normsqr[\Lin \HS]{P_2^\perp DP_1}\leq \norm[\Lin \HS]{D}^2.
  \end{align*}
  Similarly, we see
  $\norm[\Lin{\HS_2}]{R_2^*(\id-JJ^*)R_2}= \normsqr[\Lin
  \HS]{P_1^\perp DP_2}\leq \normsqr[\mathfrak{R_2}(\HS)]{D}$ and the
  fourth norm in the definition of $\delta_J(R_1,R_2)$ can be
  estimated as
  \begin{equation*}
    \norm[\Lin{\HS_1,\HS_2}]{JR_1-R_2J}
    = \norm[\Lin \HS]{\iota_2(JR_1-R_2J)\iota_1^*}
    =\norm[\Lin \HS]{P_2DP_1}
    \leq \norm[\Lin \HS]{D}.
  \end{equation*}
  The last norm can be treated similarly
  using~\eqref{eq:que.iso.norm1'}.  In particular, we have shown
  \begin{align*}
    \dQUE(R_1,R_2)
    \le \delta_J(R_1,R_2)
    \le \norm[\Lin \HS] D
    < \dIso(R_1,R_2) + \eps.
  \end{align*}
  As $\eps>0$ is arbitrary, the desired inequality is shown.

  \myparagraph{Second inequality $\dIso \le \sqrt 3\dQUE$:} We show
  that for every contraction $J$ with $\delta_J(R_1,R_2) \le \delta$
  there exists a triple $(\HS, \iota_1, \iota_2)$ such that
  $\norm[\Lin \HS]{D_{\HS,\iota_1,\iota_2}(R_1,R_2)}\leq \sqrt 3 \delta$.
  This inequality also implies the inequality for the infimum over all
  such $J$'s.

  We construct the triple in the following way. The construction is
  inspired by~\cite{szfbk:10}. We set
  \begin{align}
    \label{eq:constr-nagy}
    \HS
    &:= \HS_1\oplus\HS_2
      \quadtext{and denote}
      f=(f_1,f_2)
      \quadtext{for}
      f\in \HS
  \end{align}
  equipped with the natural norm defined by
  $\normsqr[\HS]{f}=\normsqr[\HS_1]{f_1}+\normsqr[\HS_2]{f_2}$.  Moreover, we set
  \begin{align*}
    \map{\iota_1&}{\HS_1}{\HS},
    & \iota_1f_1&:=(W_1 f_1,Jf_1), \quad
                      W_1:=(\id_{\HS_1}-J^*J)^{1/2}
                  \quad\text{and}\\
    \map{\iota_2&}{\HS_2}{\HS},
    & \iota_2f_2&:=(0,f_2).
  \end{align*}
  Note that $W_1$ is well-defined as $\norm J \le 1$.  Then it is
  easily seen that $\iota_1$ and $\iota_2$ are isometries and that the
  identification operator factorises $Jf_1=\iota_2^*\iota_1f_1$. Using
  the equations~\eqref{eq:que.iso.norm1}--\eqref{eq:que-mod'}, we
  estimate
  \begin{align}
    \nonumber
    \normsqr{D_{\HS,\iota_1,\iota_2}(R_1,R_2)f}
    &=\normsqr{(P_2 D P_1 + P_2 D P_1^\perp)f + P_2^\perp D P_1f}\\
    \nonumber
    &=\normsqr{(P_2 D P_1 + P_2 D P_1^\perp)f} + \normsqr{P_2^\perp D P_1f}\\
    \nonumber
    &\leq 2\bigl(\normsqr{(P_2 D P_1)(P_1f)}
      + \normsqr{(P_2 D P_1^\perp)(P_1^\perp f)}\bigr)
      + \normsqr[\Lin \HS]{P_2^\perp D P_1}\normsqr f\\
    \nonumber
    &\le 2\max\bigl\{\normsqr[\Lin{\HS_1,\HS_2}]{JR_1-R_2J},
      \norm[\Lin{\HS_2}]{R_2^*(\id-JJ^*)R_2}\bigr\}\normsqr f\\
    \nonumber
    &\hspace*{0.35\textwidth}
      + \norm[\Lin{\HS_1}]{R_1^*(\id-J^*J)R_1}\bigr)\normsqr f\\
    \label{eq:2nd.ineq.b}
    &\leq (2\delta^2+\delta^2)\normsqr f
      =3 \delta^2 \normsqr f
  \end{align}
  using Pythagoras two times.  In particular, we have shown
  $\dIso(R_1,R_2) \le \sqrt 3 \dQUE(R_1,R_2)$.
\end{proof}
%-----------------------------------------------------------------------------

%-----------------------------------------------------------------------------
\begin{remark}[sharpness of estimate isometric versus quasi-unitary
  distance]
  \label{rem:diso.dque.sharp}
  We doubt that the estimate
  $\dIso(R_1,R_2) \le \sqrt 3 \dQUE(R_1,R_2)$ is sharp, but we do not
  have a better estimate.  We only have a lower bound on the ratio
  $\dIso(R_1,R_2)/\dQUE(R_1,R_2)$ for some concrete $R_1$ and $R_2$,
  namely $\sqrt 2$ (see \Sec{examples} and \Cor{at.least.sqrt2}).
  Calculating $\dQUE(R_1,R_2)$ even for simple operators $R_1$ and
  $R_2$ seems to be rather complicated in general cases due to its
  generality.  Probably some estimates for variants of $\dQUE$
  provided in a forthcoming publication~\cite{post-zimmer:pre24c}
  might help.

  Ideally, we look for an equivalent variant $\dQUEvar$ of $\dQUE$
  for which  we have $\dQUEvar(R_1,R_2)=\dIso(R_1,R_2)$,
  but we are not sure whether such a variant exists.  At least,
  \Prps{diso.mult.id}{dque.mult.id} show that with our choice of
  $\delta_J(R_1,R_2)$ in the definition of $\dQUE$, this is not the
  case.
\end{remark}
%-----------------------------------------------------------------------------

%-----------------------------------------------------------------------------
\begin{remark}[Hausdorff distance of spectra versus quasi-unitary distance]
  \label{rem:dque.spec}
  For self-adjoint operators $R_n$ with $0 \in \essspec {R_n}$
  ($n=1,2$), we conclude from \ThmS{main-1}{main-3} that
  \begin{align}
    \label{eq:dhaus.dque}
    \dHaus(\spec{R_1},\spec{R_2})
    \le \dUni(R_1,R_2)
    \le \dIso(R_1,R_2)
    \le \sqrt 3 \dQUE(R_1,R_2)
  \end{align}
  using the strong result \Thm{azoff-davis-duni} taken
  from~\cite{azoff-davis:84}.  We now comment on a direct proof of
  this estimate in terms of a variant of $\dQUE(R_1,R_2)$
  using~\cite[Thm.~3.4]{khrabustovskyi-post:21}.  We assume here that
  $R_n$ is non-negative and injective operator, e.g., it is the
  resolvent of a non-negative operator $A_n$, namely
  $R_n=(A_n+1)^{-1}$ for $n=1,2$.  Using
  $\norm{R_1^{1/2}(\id-J^*J)R_1^{1/2}}$ instead of
  $\norm{R_1^*(\id-J^*J)R_1}^{1/2}$ in~\eqref{eq:delta_J} then
  $\delta_J(R_1,R_2) \le \delta$ implies
  \begin{align*}
    \norm{R^{1/2}(\id-J^*J)R_1^{1/2}}
    =\sup_{g_1 \in \HS_1}
    \frac{\normsqr[\HS_1] {R_1^{1/2}g_1} - \normsqr[\wt \HS]{J_1R_1^{1/2}g_1}}%
    {\normsqr[\HS_1] {g_1}}
    =\sup_{f_1 \in \HS_1^1}
    \frac{\normsqr[\HS_1] {f_1} - \normsqr[\HS_2]{Jf_1}}%
    {\normsqr[\HS_1^1] {f_1}}\le \delta
  \end{align*}
  with $f_1=R_1^{1/2}g_1$, where $\HS_1^1:=\dom R_1^{-1/2}$ and
  $\norm[\HS_1^1] {f_1}:=\norm[\HS_1]{R_1^{-1/2}{f_1}}$ denotes the
  quadratic form norm of $\qf a_1 (f_1)=\normsqr[\HS_1]{A_1^{1/2}}$.
  In particular, we conclude
  \begin{align*}
    \normsqr[\HS_1] {f_1}
    \le \mu_1 \normsqr[\HS_2]{Jf_1} + \nu_1 \qf a_1(f_1)
    \quadtext{with}
    \mu_1=1+\frac{\delta^2}{1-\delta^2}
    \quadtext{and}
    \nu_1 = \frac{\delta^2}{1-\delta^2}.
  \end{align*}
  A similar argument holds for $R_2$ and $J$ replaced by $J^*$.  In
  particular, the assumptions of Thm.~3.1
  of~\cite{khrabustovskyi-post:21} are now fulfilled and we conclude
   \begin{align*}
     \dHaus(\spec{R_1},\spec{R_2})
     &\le \max
     \Bigl\{
     \delta \sqrt{\frac1\kappa\Bigl(1+\frac{\delta^2}{1-\delta^2}\Bigr)},
     \frac{\delta^2}{(1-\kappa)(1-\delta^2)}
     \Bigr\}\\
     &=\max
     \Bigl\{
     \frac1{\sqrt \kappa}\Bigl(\delta + \Err(\delta^2)\Bigr),
     \frac1{1-\kappa}\bigl(\delta^2 + \Err(\delta^4)\bigr)
     \Bigr\}
   \end{align*}
   for any $\kappa \in (0,1)$.  Note that we used in the proof of
   \cite[Thm.~3.4]{khrabustovskyi-post:21} that $0$ is in the spectrum
   of $\spec{R_n}$.  The error is of optimal order
   $\delta + \Err(\delta^2)$ for $\kappa=1-\delta$ and
   $\delta \in [0,1)$.  In particular, taking the infimum as in
   $\dQUE$ (with the modified version $\dQUEvar$ of
   $\delta_J(R_1,R_2)$ as above), we obtain
   \begin{align*}
     \dHaus(\spec{R_1},\spec{R_2})
     \le \dQUEvar(R_1,R_2) + \Err(\dQUEvar(R_1,R_2)^2).
   \end{align*}
   This result suggests that the factor $\sqrt 3$
   in~\eqref{eq:dhaus.dque} seems not to be optimal, at least for our
   version of $\dQUE$ (see also \Rem{diso.dque.sharp}).  We will
   comment on related questions in a forthcoming
   publication~\cite{post-zimmer:pre24c}.
\end{remark}
%-----------------------------------------------------------------------------

%---------------------------------------------------------------------------------
\section{Examples of distances}
\label{sec:examples}
%---------------------------------------------------------------------------------
Here we give some examples in which one can actually calculate the
various distances between operators.

%---------------------------------------------------------------------------------
\subsection{Multiples of the identity}
%---------------------------------------------------------------------------------
In this subsection, we analyse the various distances for
$R_n=r_n \id_{\HS_n}$ with $r_n \in \C$ ($n=1,2$).  Trivially, we have
\begin{equation}
  \label{eq:duni.mult.id}
  \dUni(r_1 \id_{\HS_1},r_2 \id_{\HS_2})
  = \abs{r_1-r_2}
\end{equation}
as $U_{12} R_1 U_{12}^*= r_1 U_{12} U_{12}^*=r_1 \id_{\HS_2}$ for
$U_{12} \in \Unitary{\HS_1,\HS_2}$.

The calculations for $\dIso$ are actually more involved:
%---------------------------------------------------------------------------------
\begin{proposition}[isometric distance]
  \label{prp:diso.mult.id}
  For $r_1,r_2 \in \C$, we have
  \begin{equation}
    \label{eq:diso.mult.id}
    \dIso(r_1 \id_{\HS_1},r_2 \id_{\HS_2})
    = \min\{\abs{r_1-r_2},\max\{\abs{r_1},\abs{r_2}\}\}.
  \end{equation}
  In particular, if $r_1,r_2 \in \R$ then
  \begin{itemize}
  \item for $r_1 \cdot r_2 \ge 0$ we have
    $ \dIso(r_1 \id_{\HS_1},r_2 \id_{\HS_2}) = \abs{r_1-r_2}=\dUni(r_1
    \id_{\HS_1},r_2 \id_{\HS_2})$;
  \item for $r_1 \cdot r_2 < 0$ we have
    $ \dIso(r_1 \id_{\HS_1},r_2 \id_{\HS_2}) = \max\{\abs{r_1},
    \abs{r_2}\} < \dUni(r_1 \id_{\HS_1},r_2 \id_{\HS_2})$.
\end{itemize}
\end{proposition}
%---------------------------------------------------------------------------------
\newcommand{\perpsymb}[1]{\check#1}
\begin{proof}
  We keep the notation of the previous sections.  We first note that
  \begin{align*}
    D_{\HS,\iota_1,\iota_2}(r_1 \id_{\HS_1},r_2 \id_{\HS_2})
    =r_1P_1-r_2P_2
  \end{align*}
  where $P_n=\iota_n\iota_n^*$ are the orthogonal projections (ONPs)
  onto the range of $\iota_n$ in $\HS$.  In particular, we have
  \begin{align*}
    \dIso(r_1 \id_{\HS_1},r_2 \id_{\HS_2})
    = \inf \bigset{\norm{r_1P_1-r_2P_2}}%
    {\text{$P_n$ ONP, $\dim \ran P_n =\dim \HS_n$, $n=1,2$}}.
  \end{align*}
  We now use Halmos' two projections theorem~\cite{halmos:69} (see
  also~\cite{boettcher-spitkovsky:10}) to calculate the norm of
  $r_1P_1-r_2P_2$: We set $\HS_n := \ran P_n$ and
  $\HS_{\perpsymb n} := (\ran P_n)^\perp$.  Moreover, we
  set\footnote{Here, $\HS \ominus \HS_1$ denotes the orthogonal
    complement of $\HS_1$ in $\HS$.}
  \begin{align*}
    \HS_{jk}
    &:=\HS_j \cap \HS_k,
    &\HS_{10}
    &:= \HS_1 \ominus (\HS_{12} \oplus \HS_{1\perpsymb2}),
    &\HS_{\perpsymb10}
    &:= \HS_{\perpsymb1} \ominus (\HS_{\perpsymb12} \oplus \HS_{\perpsymb1\perpsymb2})
  \end{align*}
  for $j,k \in \{1,\perpsymb1,2,\perpsymb 2\}$.  In particular, we
  have the orthogonal splitting
  \begin{align*}
    \HS &= \HS_{12}
    \oplus \HS_{1\perpsymb2}
    \oplus \HS_{\perpsymb12}
    \oplus \HS_{\perpsymb1\perpsymb2}
    \oplus \HS_{10}
    \oplus \HS_{\perpsymb10}.
  \end{align*}
  If the last two spaces $\HS_{10}$ and $\HS_{\perpsymb10}$ are
  trivial, then we can represent $P_1$ and $P_2$ in a block matrix
  representation as in~\cite{boettcher-spitkovsky:10} as
  \begin{subequations}
    \begin{align}
      \label{eq:two.projections}
      P_1 = 1 \oplus 1 \oplus 0 \oplus 0
      \qquadtext{and}
      P_2 = 1 \oplus 0 \oplus 1 \oplus 0
    \end{align}
    where $1$ resp.\ $0$ represent the identity resp.\ null operator
    on the corresponding component.  If one of the spaces $\HS_{10}$
    or $\HS_{\perpsymb10}$ is nontrivial, then they are unitarily
    equivalent, i.e.\ there is a unitary operator
    $\map{U_0}{\HS_{\perpsymb10}}{\HS_{10}}$.  Moreover, Halmos'
    result~\cite{halmos:69} reads as
    \begin{align}
      \nonumber
      P_1 &\cong (1 \oplus 1 \oplus 0 \oplus 0) \oplus
      \begin{pmatrix}
        1 & 0\\ 0 & U_0^*
      \end{pmatrix}
      \begin{pmatrix}
        1 & 0\\ 0 & 0
      \end{pmatrix}
      \begin{pmatrix}
        1 & 0\\ 0 & U_0
      \end{pmatrix}\\
      \label{eq:two.projections'}
      P_2 &\cong (1 \oplus 0 \oplus 1 \oplus 0) \oplus
      \begin{pmatrix}
        1 & 0\\ 0 & U_0^*
      \end{pmatrix}
      \begin{pmatrix}
        C^2 & CS\\ CS & S^2
      \end{pmatrix}
      \begin{pmatrix}
        1 & 0\\ 0 & U_0
      \end{pmatrix},
    \end{align}
  \end{subequations}
  where the first four components are the diagonal entries of a block
  matrix operator.  Moreover, the fifth and sixth component act as
  given by the $2 \times 2$-matrix, and $C,S$ are two bounded
  operators on $\HS_{10}$ such that $C^2+S^2=\id_{\HS_{10}}$.  It
  follows that $CS=SC$, $\spec C, \spec S \subset [0,1]$ and
  $\cos \theta \in \spec C$ if and only if $\sin\theta \in \spec S$
  for some $\theta \in [0,\pi/2]$.

  Now, $r_1P_1-r_2 P_2$ is represented by
  \begin{align*}
    ((r_1-r_2)\oplus r_1 \oplus -r_2 \oplus 0)
    \oplus
    \begin{pmatrix}
      1 & 0\\ 0 & U_0^*
    \end{pmatrix}
    \begin{pmatrix}
      r_1-r_2C^2 & -r_2CS\\ -r_2CS & -r_2S^2
    \end{pmatrix}
    \begin{pmatrix}
      1 & 0\\ 0 & U_0
    \end{pmatrix},
  \end{align*}
  and its operator norm is
  \begin{align*}
    &\max
    \Bigl\{\abs{r_1-r_2},
    \abs{r_1},
    \abs{r_2},
    \max \Bigset{\abs{\lambda}}
    {\text{$\lambda$ eigenvalue of $M_{r_1,r_2}(\theta)$},
    \theta \in [0,\pi/2]
    }
    \Bigr\}, \quad\text{were}\\
    &M_{r_1,r_2}(\theta)
    =\begin{pmatrix}
      r_1-r_2(\cos \theta)^2 & -r_2\cos \theta \sin \theta\\
      -r_2\cos \theta \sin \theta & -r_2(\sin \theta)^2
    \end{pmatrix}.
  \end{align*}
  The eigenvalues of the latter matrix are
  \begin{align*}
    \lambda
    = \frac12\cdot(r_1-r_2) \pm
    \frac12 \sqrt{(r_1-r_2)^2+4r_1r_2 (\sin \theta)^2}.
  \end{align*}
  The maximal values of $\abs \lambda$ are
  $\abs{r_1-r_2}$ if $\theta=0$ and $\max\{\abs{r_1},\abs{r_2}\}$ if
  $\theta=\pi/2$.  In particular, we have
  \begin{align*}
    \norm{r_1P_1-r_2P_2}
    = \max
    \bigl\{\abs{r_1-r_2},
    \abs{r_1},
    \abs{r_2}\bigr\}
  \end{align*}
  We now take the infimum over all possible projections $P_1$ and
  $P_2$.  If $P_1=P_2$, then this norm is actually $\abs{r_1-r_2}$.
  If their ranges are orthogonal, then
  $\norm{r_1P_1-r_2P_2}=\max\{\abs{r_1},\abs{r_2}\}$.  This proves the
  claim.
\end{proof}
%---------------------------------------------------------------------------------

%---------------------------------------------------------------------------------
\begin{proposition}[quasi-unitary distance]
  \label{prp:dque.mult.id}
  For $r_1,r_2 \in \C$, we have
  \begin{equation}
    \label{eq:dque.mult.id}
    \dQUE(r_1,r_2)
    := \dQUE(r_1 \id_{\HS_1},r_2 \id_{\HS_2})
    = \Bigl(\frac1{\abssqr{r_1-r_2}}
    + \frac 1{\max\{\abssqr {r_1},\abssqr {r_2}\}}\Bigr)^{-1/2}
  \end{equation}
  if $r_1 \ne r_2$ and $\dQUE(r_1 \id_{\HS_1},r_2 \id_{\HS_2}) =0$
  otherwise.  Moreover, $\dQUE$ on $\C$ is homogeneous, i.e.,
  $\dQUE(\alpha r_1,\alpha r_2)=\abs\alpha \dQUE(r_1,r_2)$ for all
  $\alpha,r_1,r_2 \in \C$.
\end{proposition}
%---------------------------------------------------------------------------------
\begin{proof}
  We have
  \begin{align*}
    \delta_J(r_1\id_{\HS_1},r_2\id_{\HS_2})
    =\max \Bigl\{
      \abs{r_1}\norm{(\id-J^*J)}^{1/2},
    \abs{r_2}\norm{(\id-JJ^*)}^{1/2},
    \abs{r_1-r_2}\norm J
    \Bigr\}.
  \end{align*}
  Moreover,
  $\norm J = \norm{J^*J}^{1/2}=\norm{JJ^*}^{1/2}=\cos \theta_1$ for
  some $\theta_1 \in [0,\pi/2]$.  As
  $\spec{J^*J} \setminus\{0\}=\spec{JJ^*}\setminus \{0\} \subseteq
  (0,1]$ (see e.g.~\cite[Prp.~1.2]{post:09c}), we also know that
  \begin{align*}
    \norm{(\id-J^*J)}^{1/2}
    = \sin \theta_2
    \quadtext{and}
    \norm{(\id-J^*J)}^{1/2}
    = \sin \theta_2'
  \end{align*}
  for some $\theta_2,\theta_2' \in [0,\pi/2]$ and
  $\theta_1 \le \theta_2,\theta_2'$.  Moreover, if both kernels,
  $\ker (J^*J)=\ker J$ and $\ker (JJ^*) =\ker J^*$ are trivial or if
  both are non-trivial, then $\theta_2=\theta_2'$, otherwise,
  $\theta_2=\pi/2 \ge \theta_2'$ (if $\ker J$ is non-trivial and
  $\ker J^*=\{0\}$) and vice versa.  In particular, we have to minimise
  \begin{align}
    \label{eq:dque.mult.id'}
    \delta_J(r_1\id_{\HS_1},r_2\id_{\HS_2})
    =\max \Bigl\{
      \abs{r_1} \sin \theta_2, \abs{r_2} \sin \theta_2',
    \abs{r_1-r_2}\cos \theta_1
    \Bigr\}
  \end{align}
  over $\theta_1,\theta_2,\theta_2' \in [0,\pi/2]$ such that
  $\theta_1 \le \theta_2,\theta_2'$.  It is not hard to see that this
  quantity is minimised for $\theta_0:=\theta_1=\theta_2=\theta_2'$
  and when $r_0 \sin \theta_0=\abs{r_1-r_2} \cos \theta_0$, where
  $r_0:=\max\{\abs{r_1},\abs{r_2}\}$.  We may assume that $r_0>0$ as
  $r_0=0$ implies $r_1=r_2=0$.  It follows that
  $\theta_0 = \arctan (\abs{r_1-r_2}/r_0)$ and that the value
  of~\eqref{eq:dque.mult.id'} is
  \begin{align}
    \label{eq:dque.mult.id''}
    \abs{r_1-r_2}\cos \theta_0
    = \frac {\abs{r_1-r_2}}{(1+(\tan \theta_0)^2)^{1/2}}
    = \frac {\abs{r_1-r_2}}{(1+\frac{\abssqr{r_1-r_2}}{r_0^2})^{1/2}}
    = \Bigl(\frac1{\abssqr{r_1-r_2}}+\frac1{r_0^2}\Bigr)^{-1/2}.
  \end{align}
  using the convention $(1/0+1/r_0^2)^{-1/2}=0$.  If
  $\spec{(J^*J)^{1/2}}$ consists of a single value $\cos \theta_0$
  only, then $\dQUE(r_1 \id_{\HS_1},r_2 \id_{\HS_2})$ is achieved by
  $\delta_J(r_1\id_{\HS_1},r_2\id_{\HS_2})$ with value as
  in~\eqref{eq:dque.mult.id''}.  The homogeneity of $\dQUE$ is clear.
  % Positivity and symmetry of $\dQUE$ (as metric on $\C$) are trivial,
  % as well as homogeneity.  For the triangle inequality, note that
  % $\dQUE(r_1,r_2)=\Xi(\abs{r_1-r_2},d_0(r_1,r_2))$ where
  % $d_0(r_1,r_2)=\max\{\abs{r_1},\abs{r_2}\}$ fulfils the triangle
  % inequality.  Now one needs that $\Xi$ defined by
  % \begin{align*}
  %   \Xi(a,b)
  %   =\frac{ab}{(a^2+b^2)^{1/2}}
  %   =\begin{cases}
  %     (a^{-2}+b^{-2})^{-1/2}, & (a,b)\ne (0,0)\\
  %     0, & (a,b)=(0,0)
  %   \end{cases}
  % \end{align*}
  % for $a,b \ge 0$ is subadditive and monotonously non-decreasing in
  % each variable (the latter is trivial).  In this case, we have
  % \begin{align*}
  %   d_0(r_1,r_3)
  %   &= \Xi\bigl(\abs{r_1,r_3},d_0(r_1,r_3)\bigr)\\
  %   &\le \Xi\bigl(\abs{r_1,r_2}+\abs{r_2,r_3},d_0(r_1,r_2)+d_0(r_2,r_3)\bigr)
  %   & \text{($\Xi$ monotone)}\\
  %   &  \le \Xi\bigl(\abs{r_1,r_2},d_0(r_1,r_2)\bigr)
  %     + \Xi\bigl(\abs{r_2,r_3},d_0(r_2,r_3)\bigr)
  %   & \text{($\Xi$ subadditive)}\\
  %   & = d_0(r_1,r_2)+d_0(r_2,r_3)
  % \end{align*}
  % using the triangle inequality for $\abs\cdot$ and $d_0$.  The
  % subadditivity of $\Xi$ follows from the fact that $\Xi$ is concave
  % (its Hessian is non-negative) and $\Xi(0,0)=0$.
%  \lookO{Check! Could  not find a reference \dots}
\end{proof}
%---------------------------------------------------------------------------------

%---------------------------------------------------------------------------------
\subsection{One operator is zero}
%---------------------------------------------------------------------------------
Here, we analyse the various distances for $R_1=R$ and $R_2=0$.  In
this case we have
\begin{equation}
  \label{eq:duni.zero}
  \dUni(R,0)
  =\norm R
\end{equation}
as $\norm{U_{12} R U_{12}^*}=\norm R$ for
$U_{12} \in \Unitary{\HS_1,\HS_2}$.
Similarly, we have
\begin{equation}
  \label{eq:diso.zero}
  \dIso(R,0)
  =\norm R
\end{equation}
as $\norm{\iota_1 R \iota_1^*}=\norm R$ for isometries
$\map{\iota_1}{\HS_1}\HS$.
%---------------------------------------------------------------------------------
\begin{proposition}
  \label{prp:dque.zero}
  We have
  \begin{equation}
    \label{eq:dque.zero}
    \dQUE(R,0)
    = \frac {\norm R}{\sqrt 2}.
  \end{equation}
\end{proposition}
%---------------------------------------------------------------------------------
\begin{proof}
  In this case, we have
  \begin{align*}
    \delta_J(R,0)
    = \max\{\norm{R^*(\id-J^*J)R}^{1/2},\norm{JR}\}
    = \max\{\norm{SR},\norm{CR}\}
  \end{align*}
  using again the operators $C=(J^*J)^{1/2}$ and $S$ with
  $0\le C,S \le \id$ and $C^2+S^2=\id$ from Halmos' two projection
  theorem~\cite{halmos:69} as in the proof of \Prp{diso.mult.id}: note
  that we have
  \begin{align*}
    \norm{SR}
    &=\norm{R^*S^2R}^{1/2}
    =\norm{R^*(\id-C^2)R}^{1/2}
    =\norm{R^*(\id-J^*J)R}^{1/2}
      \quad\text{and}\\
    \norm{CR}
    &=\norm{R^*C^2R}^{1/2}
    =\norm{R^*J^*JR}^{1/2}
    =\norm{(JR)^*(JR)}^{1/2}
    =\norm{JR}.
  \end{align*}
  Now we have
  \begin{align*}
    \normsqr{SR}
    =\norm{R^*S^2R}
    =\norm{R^*(\id-C^2)R}
    \ge \norm{R^*R}-\norm{R^*C^2R}
    = \normsqr R - \normsqr{CR}
  \end{align*}
  using the reverse triangle inequality.  In particular,
  \begin{align*}
    \delta_J(R,0)
    \ge \max\{\sqrt{\normsqr R - \normsqr{CR}},\norm{CR}\}.
  \end{align*}
  The latter maximum actually becomes minimal for those $C$ for which
  both terms are equal, i.e.\ for which $\norm{CR}=\norm R/\sqrt 2$,
  hence $\delta_J(R,0) \ge \norm R/\sqrt 2$.  Choosing
  $C=S=(1/\sqrt 2)\id$ gives $\dQUE(R,0) \le \norm R/\sqrt 2$ and the
  result is proven.
\end{proof}
%---------------------------------------------------------------------------------

%-----------------------------------------------------------------------------
\subsection{Counterexamples to various statements}
\label{ssec:counterex}
%-----------------------------------------------------------------------------

In this subsection we collect several conterexamples and estimates
showing that our results are sharp.

%-------------------------------------------------------------------------------
\subsubsection*{The distances when operators act in the same Hilbert space}
%-------------------------------------------------------------------------------
We make now some statements about the distances $\dUni$, $\dIso$ and
$\dQUE$ (denoted by $d_\bullet$ in the sequel) in the case when
$\HS_1=\HS_2 (:=\HS)$, i.e., when $R_1$ and $R_2$ act in the same
Hilbert space.  In this case, we can directly calculate
$\norm[\Lin \HS]{R_1-R_2}$, but this number may be different from
$d_\bullet(R_1,R_2)$ in all cases:
%-------------------------------------------------------------------------------
\begin{example}
  \label{ex:dist.norm.diff}
  Assume that $\HS_1=\HS_2=\HS$, then
  $d_\bullet(R_1,R_2)<\norm{R_1-R_2}$ may happen in all cases:

  Assume that $\HS=\HSaux \oplus \HSaux$ for some Hilbert space
  $\HSaux$.  Let $R_1$ and $R_2$ be given by
  \begin{align*}
    R_1 =
    \begin{pmatrix}
      1 & 0\\
      0 & 2
    \end{pmatrix}
   \quadtext{and}
    R_2 = \frac12
    \begin{pmatrix}
      3 & 1\\
      1 & 3
    \end{pmatrix}
  \end{align*}
  in matrix representation with respect to the orthogonal
  decomposition.  These operators are unitarily equivalent, hence
  \begin{align*}
    0 \le
    \dQUE(R_1,R_2)
    \le \dIso(R_1,R_2)
    \le \dUni(R_1,R_2)=0.
  \end{align*}
  (second inequality by \Thm{main-3}, third is the trivial one of
  \Thm{main-2}).  Moreover, both operators have the same spectrum
  $\{1,2\}$ (with multiplicity $\dim \HSaux$).  For the operator norm
  difference, we have
  \begin{align*}
    \norm{R_1-R_2}=
    \Bignorm{
    \begin{pmatrix}
      -1/2 & -1/2\\
      -1/2 & 1/2
    \end{pmatrix}
    }=\frac1{\sqrt 2}.
  \end{align*}
\end{example}
%-------------------------------------------------------------------------------
The latter example is rather trivial, as $R_1$ and $R_2$ are not
simultaneously diagonalised. But even if both operators are
diagonalised we may have strict inequality at least for $\dIso$ and
$\dQUE$:
%-------------------------------------------------------------------------------
\begin{example}[all distances may be strict]
  \label{ex:dist.norm.diff'}
  Let $R_n=r_n \id_\HS$ with $r_1,r_2 \in \C$ for $n=1,2$ and
  $r_1 \ne r_2$.  From~\eqref{eq:duni.mult.id},
  \Prps{diso.mult.id}{dque.mult.id} we conclude
  \begin{align*}
    \dQUE(R_1,R_2)
    &= \Bigl(\frac1{\abssqr{r_1-r_2}}
    + \frac 1{\max\{\abssqr {r_1},\abssqr {r_2}\}}\Bigr)^{-1/2}\\
    < \dIso(R_1,R_2)
    &= \min\{\abs{r_1-r_2},\max\{\abs{r_1},\abs{r_2}\}\\
    \le \dUni(R_1,R_2)
    &= \abs{r_1-r_2}=\norm[\Lin \HS]{R_1-R_2}.
  \end{align*}
  For the first strict inequality, note that
  $(a^{-2}+b^{-2})^{-1/2}< \min\{a,b\}$ if $a,b>0$ and $a \ne b$, and
  the latter inequality is strict provided
  $\max\{\abs{r_1},\abs{r_2}\} < \abs{r_1-r_2}$, this may happen if
  $\arg r_1=\pi+\arg r_2$ modulo $2\pi$, e.g.\ when $r_1, r_2 \in \R$
  and $r_1 \cdot r_2 < 0$.  In particular, we may have
  \begin{align*}
    \dQUE(R_1,R_2)
    < \dIso(R_1,R_2)
    < \dUni(R_1,R_2).
  \end{align*}
\end{example}
%-------------------------------------------------------------------------------

%-------------------------------------------------------------------------------
\subsubsection*{The condition on the essential spectrum is needed}
%-------------------------------------------------------------------------------
We now comment on the condition~\eqref{eq:cond.ess.spec} in
\Thm{main-2} and \Cor{main-4}, namely why
$0 \in \essspec{R_1} \cap \essspec{R_1}$ is needed.  If $0$ is not in
one (or both) of the essential spectra, then various conclusions about
$\dUni$ and $\dIso$ may fail:
%-------------------------------------------------------------------------------
\begin{example}[zero in one essential spectrum needed]
  \label{ex:duni.eq.diso.counterex}
  We now show that
  \begin{align*}
    \dIso(R_1,R_2) < \dUni(R_1,R_2)
    \quadtext{may happen if}
    0 \notin \essspec{R_1}
    \quadtext{but}
    0 \in \essspec{R_2}:
  \end{align*}
  Let $\Sigma_1=[1,2]$ and $\Sigma_2=[-2,1]$ with its Lebesgue
  measure.  Moreover let $\HS_n=\Lsqr{\Sigma_n}$ and let $R_n$ be the
  multiplication operator defined by
  $(R_nf_n)(\lambda)=\lambda f_n(\lambda)$.  We have
  $\spec {R_1}=\Sigma_1$ and $\spec {R_2}=\Sigma_2$.  The Hausdorff
  distance of the spectra is $\dHaus(\Sigma_1,\Sigma_2)=3$. Moreover
  from~\eqref{eq:dhaus.duni'} we conclude
  $\dHaus(\Sigma_1,\Sigma_2)=\dUni(R_1,R_2)$.

  In order to estimate $\dIso(R_1,R_2)$ from above, we choose
  $\HS=\Lsqr{[-2,2]}$ and we let $\iota_n f_n$ be the extension of
  $f_n$ by $0$ onto $[-2,2]$.  Here,
  $\iota_1(\HS_1)\cap \iota_2(\HS_2)=\{0\}$, so that
  \begin{align*}
    \normsqr[\HS]{D_{\HS,\iota_1,\iota_2}(R_1,R_2) f}
    = \int_{\Sigma_1 \cup \Sigma_2}
    \abssqr{\lambda} \abssqr{f(\lambda)} \dd \lambda
    \le 4 \normsqr[{\Lsqr{[-2,2]}}] f.
  \end{align*}
  for all $f \in \Lsqr{[-2,2]}$.  In particular, we have
  \begin{align*}
    \dIso(R_1,R_2)
    \le \norm[\Lin \HS]{D_{\HS,\iota_1,\iota_2}(R_1,R_2)}
    = 2
    < 3
    = \dUni(R_1,R_2).
  \end{align*}
\end{example}
%---------------------------------------------------------------------------------

%---------------------------------------------------------------------------------
\begin{example}[zero in both essential spectra needed]
  \label{ex:duni.eq.diso.counterex'}
  We may also use \Prp{diso.mult.id} to show that
  $\dIso(R_1,R_2)<\dUni(R_1,R_2)$ and $0$ is neither in
  $\essspec{R_1}$ nor in $\essspec{R_2}$.  Furthermore,
  $0<\dUni(R_1,R_2)-\dIso(R_1,R_2)$ may be arbitrarily large:

  Let $R_n=r_n \id_{\HS_n}$ for $r_n \in \R$.  If $r_1>0$ and
  $-r_1 < r_2 < 0$, then we conclude from \Prp{diso.mult.id} that
  \begin{align*}
    \dUni(R_1,R_2)
    =\abs{r_1-r_2}
    =r_1-r_2
    > \dIso(R_1,R_2)
    =\max\{\abs{r_1},\abs{r_2}\}
    = \abs{r_1}=r_1
  \end{align*}
  In particular, we have
  \begin{align*}
    \dUni(R_1,R_2)-\dIso(R_1,R_2)
    = (r_1-r_2) - r_1=-r_2=\abs{r_2}>0,
  \end{align*}
  so choosing $r_1=-r_2+1$ we have $-r_1<r_2$ and we may choose $-r_2$
  arbitrarily large.
\end{example}
%---------------------------------------------------------------------------------

%---------------------------------------------------------------------------------
\begin{example}[isospectral distance zero, unitary not]
  \label{ex:diso.0.duni.not}
  We give a counterexample
  that~\itemref{main-4.a}$\implies$\itemref{main-4.c}
  or~\itemref{main-4.b}$\implies$\itemref{main-4.c} from \Cor{main-4}
  may fail if $0$ is not in both essential spectra, namely that it may
  happen that
  \begin{align*}
    \dIso(R_1,R_2)=0
    \quadtext{while}
    \dUni(R_1,R_2)>0:
  \end{align*}
  Let $R_n$ be operators with purely essential spectrum
  $\Sigma_1=[1,2]$ resp.\ $\Sigma_2=\{0\}\cup [1,2]$ for $n=1,2$.
  From~\eqref{eq:dhaus.duni'} we conclude
  $\dUni(R_1,R_2)=\dHaus(\Sigma_1,\Sigma_2)$, and the latter distance
  is $1$.  Assume further that $[1,2]$ is of the same spectral type
  for $R_1$ and $R_2$, say, absolutely continuous, so that $R_1$ and
  $R_2 \restr{(\ker R_2)^\perp}$ are unitarily equivalent, realised by
  a unitary map $\map U {\HS_1}{(\ker R_2)^\perp}$.  Choose $\HS_2$ as
  parent space with $\map {\iota_2=\id_{\HS_2}} {\HS_2}{\HS_2}$ and
  define an isometry
  $\map{\iota_1}{\HS_1}{\HS_2=\ker R_2 \oplus (\ker R_2)^\perp}$ by
  $\iota_1 f_1 = 0 \oplus U f_1$.  We now have
  \begin{align*}
    \dIso(R_1,R_2) \le \dUni(\iota_1 R_1 \iota_1^*,\iota_2 R_2\iota_2^*).
  \end{align*}
  As $\iota_1$ is not surjective and $\ker R_2$ is infinite
  dimensional, we have
  $\spec{\iota_1 R_1 \iota_1^*}=\{0\} \cup [1,2]=\Sigma_2$ by
  \Lem{spectra.of.isometries}, i.e., $\iota_1 R_1 \iota_1^*$, and
  $R_2$ have the same (purely) essential spectrum.  In particular, we
  conclude $\dUni(\iota_1 R_1 \iota_1^*,R_2)=0$ again
  by~\eqref{eq:dhaus.duni'}.
\end{example}
%---------------------------------------------------------------------------------

%-------------------------------------------------------------------------------
\subsubsection*{Isometric and quasi-unitary distance are not equal}
%-------------------------------------------------------------------------------
We continue with the slightly unsatisfactory constant $\sqrt 3$ in
\Thm{main-3}:
%---------------------------------------------------------------------------------
\begin{corollary}
  \label{cor:at.least.sqrt2}
  The constant $\sqrt 3$ in
  $\dIso(R_1,R_2) \le \sqrt 3 \dQUE(R_1,R_2)$ from \Thm{main-3} has to
  be at least $\sqrt 2$.
\end{corollary}
%---------------------------------------------------------------------------------
\begin{proof}
  The statement follows from~\eqref{eq:diso.zero}
  and~\eqref{eq:dque.zero} as $\dIso(R,0)=\sqrt 2 \dQUE(R,0)$.
\end{proof}
%-------------------------------------------------------------------------------

We finally comment on~\eqref{eq:dhaus.duni} for $\dIso$ and $\dQUE$
instead of $\dUni$:
%-----------------------------------------------------------------------------
\begin{example}[isometric, quasi-unitary and spectral distance]
  \label{ex:diso.dhaus}
  Let $\HS_1:= \C^2$ and $\HS_2:=\C^3$ and consider the operators
  given by their matrix representations
  \begin{equation*}
    R_1=
    \begin{pmatrix}
      \lambda_1 & 0 \\
      0 & \lambda_2\\
    \end{pmatrix}
    \qquadtext{and}
    R_2=
    \begin{pmatrix}
      \lambda_1 & 0 & 0 \\
      0 & \lambda_2 & 0 \\
      0 & 0 & \lambda_3 \\
    \end{pmatrix}.
  \end{equation*}
  for $\lambda_1 > \lambda_2 > 0 > \lambda_3$.  Note that
  $\dUni(R_1,R_2)=\infty$ as there is no unitary map between $\C^2$
  and $\C^3$.  Similarly, for the \LevyProkhorov distance, we have
  $\dCmf(R_1,R_2)=\delta(\alpha_{R_1},\alpha_{R_2})=\infty$ as for
  $I=(\lambda_1-1,\lambda_3+1)$, there is no $\eps>0$ such that
  $(3=)\alpha_{R_2}^*(I)\le \alpha_{R_1}^*(\closedBall_\eps(I))(=2)$.
  Moreover, as
  $\dSpec(R_1,R_2)=\delta(\alpha_{R_1},\alpha_{R_2})=\infty$,
  \eqref{eq:dspec.duni} still holds.

  Nevertheless, the inequality~\eqref{eq:dhaus.duni} may become strict
  for $\dIso$ and $\dQUE$: For the Hausdorff distance, we compute
  $\dHaus(\spec{R_1},\spec{R_2}) =
  \lambda_1-\lambda_3=\lambda_1+\abs{\lambda_3}>\abs{\lambda_3}$. By
  choosing $\HS=\C^3$, $\iota_1 (x_1,x_2)=(x_1,x_2,0)$ and
  $\iota_2= \id_{\HS_2}$ we have
  \begin{equation*}
    \dIso(R_1,R_2)
    \le \norm{\iota_1R_1\iota_1^*-R_2}
    =\Bignorm{
      \begin{pmatrix}
        0 & 0 & 0 \\
        0 & 0 & 0 \\
        0 & 0 & -\lambda_3
      \end{pmatrix}}
    =\abs{\lambda_3}
    <\dHaus(\spec{R_1},\spec{R_2}).
  \end{equation*}
  In particular,~\eqref{eq:dhaus.duni} can be strict if one replaces
  $\dUni$ by $\dIso$ (and not only because of different
  multiplicities).  Again, $0$ is not in the essential spectra of
  $R_n$ as both Hilbert spaces are finite dimensional.  More
  important, $R_2$ has also negative spectrum.  As
  $\dQUE(R_1,R_2)\le \dIso(R_1,R_2)$ by \Thm{main-3}, we also have
  \begin{equation*}
    \dQUE(R_1,R_2)
    \le \abs{\lambda_3}
    <\lambda_1+\abs{\lambda_3}
    =\dHaus(\spec{R_1},\spec{R_2}).
  \end{equation*}
\end{example}
%-----------------------------------------------------------------------------

%-------------------------------------------------------------------------------
\section{Related convergence notions}
\label{sec:conv}

%-------------------------------------------------------------------------------
As mentioned in the introduction, the idea of investigating distance
notions on varying spaces arised from very similar convergence
notions.

%-------------------------------------------------------------------------------
\subsection{Unbounded operators}
%-------------------------------------------------------------------------------
Originally, the concept of quasi-unitary equivalence and related
convergence notions use \emph{resolvents} of unbounded operators.  We
briefly discuss the concepts defined already in~\cite{post-zimmer:22}
as well as some new concepts of convergence notions applied to
unbounded operators.

Here and in the sequel we assume that $\D_n$ is a closed (possibly
unbounded) operator acting in a Hilbert space $\HS_n$ for
$n \in \Nbar=\N \cup \{\infty\}$.  Moreover, we assume that there
exists a common point $z_0$ in the resolvent set, i.e.,
$z_0 \in \bigcap_{n \in \Nbar} \rho(\D_n)$, where
$\rho(\D_n)=\C \setminus \spec{\D_n}$ denotes the resolvent set of
$\D_n$, so we consider the bounded operators $R_n=(\D_n-z_0)^{-1}$.

We start with a concept used by Weidmann (and implicitly also in many
other publications), called here briefly the \emph{Weidmann
  convergence}:
%--------------------------------------------------------------------------------
\begin{definition}[Weidmann convergence]
  \label{def:gnrc-weid}
  Let $\D_n$ be a self-adjoint bounded or unbounded operator in a
  Hilbert space $\HS_n$ for
  $n \in \Nbar$.  We say
  that the sequence $(\D_n)_{n \in \N}$ \emph{converges to}
  $\D_\infty$ \emph{in generalised norm resolvent sense of Weidmann}
  (or shortly \emph{Weidmann-converges}), if the following conditions
  are true:
  \begin{enumerate}
  \item There exist a Hilbert space $\HSgen$, called \emph{parent
      (Hilbert) space} and for each $n \in \Nbar$ an isometry
    $\map {\iota_n}{\HS_n}{\HSgen}$.
  \item We have $\delta_n := \norm[\Lin{\HSgen}]{D_n} \to 0$ as
    $n \to \infty$, where
    \begin{equation}
      \label{eq:weidmanns_diff}
      D_n := \iota_n R_n \iota_n^* - \iota_\infty R_\infty \iota_\infty^*.
    \end{equation}
  \end{enumerate}
  For short, we write $ \D_n \gnrcW \D_\infty $ (with respect to
  $(\iota_n)_{n \in \Nbar}$) and call $(\delta_n)_n$ the
  \emph{convergence speed}.
\end{definition}
%--------------------------------------------------------------------------------
There is a subtle point in $\D_n \gnrcW \D_\infty$ and
$\bardIso(R_n,R_\infty)\to 0$: in the former, we have the \emph{same}
parent space $\HS$ for all $n \in \N$, while in the latter, the parent
space $\HS^{(n)}$ might depend on $n$.  In particular, the
implication~\itemref{main-5.b}$\implies$\itemref{main-5.a} in
\Thm{main-5} is not completely trivial.  We bypass this fact by
repeating the original proof
of~\itemref{main-5.c}$\implies$\itemref{main-5.a}
from~\cite[Thm.~3.1]{post-zimmer:22}, also for convenience of the
reader and because we use a slightly different version of
QUE-convergence here (see below).

Another concept was developed by the first author, named
``QUE-convergence'' here:
%--------------------------------------------------------------------------------
\begin{definition}[QUE-convergence]
  \label{def:gnrc-que}
  Let $\D_n$ be a self-adjoint bounded or unbounded operator in a
  Hilbert space $\HS_n$ for
  $n \in \Nbar:=\N \cup \{\infty\}=\{1,2,3,\dots,\infty\}$.  We say
  that the sequence $(\D_n)_{n \in \N}$ \emph{converges to}
  $\D_\infty$ \emph{in generalised norm resolvent sense} (or shortly
  \emph{QUE-converges}), if there exist
  $z_0 \in \bigcap_{n \in \Nbar} \rho(\D_n)$, a sequence of bounded
  operators $(\map{J_n}{\HS_n}{\HS_\infty})_n$ and a sequence
  $(\delta_n)_{n \in \N}$ with $\delta_n \to 0$ as $n \to \infty$ such
  that
  \begin{subequations}
    \label{eq:gnrc-que}
    \begin{align}
      \label{eq:que1}
      \norm[\Lin{\HS_n, \HS_\infty}] {J_n} \le
      &1 \\ %+\delta_n &\to 1,\\
      \label{eq:que2}
      \bignorm[\Lin {\HS_n}]{R_n(\id_{\HS_n}-J_n^*J_n)R_n}^{1/2}  \le \delta_n
      & \to 0,
        \quad
        \bignorm[\Lin {\HS_\infty}]
        {R_\infty(\id_{\HS_\infty}-J_nJ_n^*)R_\infty}^{1/2}
        \le \delta_n \to 0,\\
      \label{eq:que3}
      \norm[\Lin{\HS_n, \HS_\infty}]{R_\infty J_n-J_nR_n} \le \delta_n &\to 0
    \end{align}
  \end{subequations}
  We write for short $\D_n \gnrc \D_\infty$ (with respect to $J_n$).
  We call $(\delta_n)_n$ the \emph{convergence speed}.
\end{definition}
%--------------------------------------------------------------------------------
%--------------------------------------------------------------------------------
\begin{remarks}[on previous versions of quasi-unitary equivalence and convergence]
  \indent
  \begin{enumerate}
  \item \myparagraph{Resolvents on both sides.} %
    In previous publications, we used a slightly different version:
    namely, the resolvent appeared only once in~\eqref{eq:que2}, and
    there is no square root of the norms, the reason is already
    explained in \Rem{new-que}.

  \item \myparagraph{Contraction property in old version not needed.}%
    When comparing to the definition of QUE-distance the attentive
    reader might notice that $\norm{J_n}\le 1$ is not required in our
    previous version of QUE-convergence
    (see~\cite[Eq.~(1.6a)]{post-zimmer:22}).  We have shown
    in~\cite[Lemma 3.17]{post-zimmer:22} that one can turn a general
    $J_n$ (with $\norm{J_n}\le 1+\delta_n$) not being a contraction
    into a contraction $\hat J_n := (1/\norm {J_n}) J_n$, and (the
    old) QUE-convergence still holds with respect to $\hat J_n$ and
    the same convergence speed.  For the new version of quasi-unitary
    equivalence as in~\eqref{eq:que2}, one has a loss of convergence
    speed here if we replace~\eqref{eq:que1} with
    $\norm{J_n}\le 1+\delta_n$, namely QUE-convergence for $\hat J_n$
    holds only with convergence speed of order $\delta_n^{1/2}$,
    see~\cite[Lem.~4.3.8]{zimmer:24}.
  \end{enumerate}
\end{remarks}
%--------------------------------------------------------------------------------

%--------------------------------------------------------------------------------
\begin{proof}[Proof of \Thm{main-5}]
  \itemref{main-5.a}$\implies$\itemref{main-5.b} follows from the fact
  that if $ \D_n \gnrcW \D_\infty$ there exists a parent space $\HS$
  such that
  $\delta_n = \norm[\Lin{\HSgen}]{\iota_n R_n \iota_n^* - \iota_\infty
    R_\infty \iota_\infty^*} \to 0$. The definition via the infimum
  directly implies
  \begin{equation*}
    \bardIso(\D_n,\D_\infty)
    =\dIso(R_n,R_\infty)
    \le \norm[\Lin{\HSgen}]%
    {\iota_n R_n \iota_n^* - \iota_\infty R_\infty \iota_\infty^*} \to 0.
  \end{equation*}
  The equivalence~\itemref{main-5.c}$\iff$\itemref{main-5.d} follows
  as above by the definition of the infimum.
  For~\itemref{main-5.b}$\iff$\itemref{main-5.d} we use \Thm{main-3}.

  We now show~\itemref{main-5.c}$\implies$\itemref{main-5.a} as
  in~\cite{post-zimmer:22}, for convenience of the reader and as we
  use slightly different definitions here, we repeat briefly the
  argument: Let $\D_n \gnrc \D_\infty$ with respect to to a
  contraction sequence $(J_n)_n$ (cf.\
  \cite[Lemma~3.17]{post-zimmer:22} if $\norm{J_n}>1$). Then we
  proceed as in the proof of the second inequality of \Thm{main-3}:
  Similarly as in~\eqref{eq:constr-nagy}, we construct a parent space
  via the direct sum
 \begin{subequations}
    \label{eq:parent.hs}
    \begin{align}
      \label{eq:parent.hs.sum}
      \HSgen = \HS_\infty \oplus \bigoplus_{n=1}^\infty \HS_n,
      \qquad f=(f_\infty, f_1, f_2, \dots) \in \HSgen\\
      \label{eq:parent.hs.norm}
      \text{with}\quad
      \normsqr[\HSgen] f
      = \normsqr[\HS_\infty]{f_\infty}
      + \sum_{n=1}^\infty \normsqr[\HS_n]{f_n}
      < \infty
    \end{align}
  \end{subequations}
  and the isometries via
  \begin{subequations}
    \begin{equation}
      \label{eq:concrete.iso.n}
      \map {\iota_n}{\HS_n} \HSgen,
      \quad \iota_nf_n = (J_nf_n,0,\dots,0, W_n f_n,0,\dots)
    \end{equation}
    for $n \in \N$, where the entry
    $W_n f_n=(\id_{\HS_n}-J_n^*J_n)^{1/2}f_n$ is at the $n$-th component
    and
    \begin{equation}
      \label{eq:concrete.iso.infty}
      \map{\iota_\infty}{\HS_\infty}\HSgen,
      \quad \iota_\infty f_\infty=(f_\infty,0,\dots).
    \end{equation}
  \end{subequations}
  As in~\cite{post-zimmer:22} one easily sees that
  $J_n=\iota_\infty^*\iota_n$ for any $n \in \N$.
  Using~\eqref{eq:dn_in_3teile}
  and~\eqref{eq:que.iso.norm1}--\eqref{eq:que-mod'} we conclude
  \begin{multline*}
    \norm[\Lin{\HS}]{D_n}
    \leq\\
    \norm[\Lin{\HS_n}]{R_n(\id_n-J_n^*J_n)R_n}^{1/2}
    +\norm[\Lin{\HS_\infty}]{R_\infty(\id_\infty-J_nJ_n^*)R_\infty}^{1/2}
    +\norm[\Lin{\HS_n, \HS_\infty}]{J_nR_n-R_\infty J_n},
  \end{multline*}
  where right hand side tends to $0$ as $\D_n \gnrc \D_\infty$.  We
  therefore have shown the equivalence
  of~\itemref{main-5.a}--\itemref{main-5.d}.

  The equivalence~\itemref{main-5.e}$\iff$\itemref{main-5.f} follows
  again as above by the definition of the infimum.  The
  self-adjointness of $\D_n$ ensures that \Thm{main-1} holds, i.e., we
  have~\itemref{main-5.f}$\iff$\itemref{main-5.g}.

  \itemref{main-5.f}$\implies$\itemref{main-5.b} holds because we
  always have $\dUni(R_1,R_2) \ge \dIso(R_1,R_2)$.

  For~\itemref{main-5.b}$\implies$\itemref{main-5.f} we use
  \Thm{main-2}; note that we always have~$0 \in \essspec{R_n}$ for
  $n \in \Nbar$ as $\D_n$ is unbounded.  The
  equivalence~\itemref{main-5.f}$\iff$\itemref{main-5.g} follows from
  \Thm{main-1}.
\end{proof}
% ----------------------------------------------------------------------------

Finally, we can reformulate our results as follows:
% ----------------------------------------------------------------------------
\begin{corollary}
  \label{cor:conv}
  Let $(\D_n)_n$ be a sequence of self adjoint operators acting in
  $\HS_n$ each, $\D_\infty$ acting on $\HS_{\infty}$ and
  $\operatorname d_\circ \in \{\dIso, \dUni, \dQUE, \dCmf\}$. Then the
  following holds:
  \begin{enumerate}
  \item If $\operatorname d_\circ(R_n, R_\infty) \to 0$ for $n \in \N$, then
    $\D_n \gnrc \D_\infty$ and $ \D_n \gnrcW \D_\infty$.
  \item If $0 \in \bigcap\limits_{n \in \Nbar} \essspec{R_n}$
    (equivalently, if all operators $\D_n$ are unbounded), then the
    converse implication is also true.
  \end{enumerate}
\end{corollary}
%----------------------------------------------------------------------------
Recall that $\dCmf(R_1,R_2)=\delta(\alpha_{R_1},\alpha_{R_2})$ is the
\LevyProkhorov distance of the crude multiplicity functions
associated with $R_1$ and $R_2$, see \Prp{char.cmf}, \Def{dist.cmf}
and~\eqref{eq:def.dcmf}.

Finally, we discuss the topology generated by any of our distances on
\begin{equation*}
%  \label{def.meta.bounded}
  \setOp:= \bigcup_\HS \Lin \HS, \;\;
  \setOp_\R := \bigcup_\HS \Lin[\R] \HS, \;\;
  \setOp_0 := \bigcup_\HS \Lin[0] \HS
  \;\;\text{and}\;\;
  \setOp_{\R,0} := \setOp_\R \cap \setOp_0,
\end{equation*}
where $\Lin[\R]\HS$ and $\Lin[0]\HS$ denote the set of all
self-adjoint resp.\ operators with $0$ in the essential spectrum.
Here, the union is taken over the set of all separable (or finite
dimensional) Hilbert spaces.  We will not go into details on
set-theoretic problems of such a definition, as we mostly have
applications in mind, where the set is a family depending on a real
parameter or a natural number.  One could think of a \emph{bundle of
  Hilbert spaces} as discussed e.g.\
in~\cite[Sec.~2]{kuwae-shioya:03}.

From \Cor{conv} we conclude:
%----------------------------------------------------------------------
\begin{corollary}
  The topologies on $\setOp_{\R,0}$ generated by $\dUni$, $\dIso$,
  $\dQUE$ and $\dCmf$ are all the same.% on $\setOp$.
\end{corollary}
%----------------------------------------------------------------------
Note that the topology is not Hausdorff (by \Cor{main-4}).  Taking the
quotient metric space $\setOp_{\R,0}/{\sim}$ with $R_1 \sim R_2$ if
and only one of the conditions in \Cor{main-4} is fulfilled, then
$[R]$ is determined by $\essspec R$ and $\disspec R$ including
multiplicities.  Let $\Comp$ be the set of all crude multiplicity
functions with compact support, then we can naturally identify compact
subsets $\Sigma \subset \R$ with $\alpha=\infty \1_K$ and finite
multisets with another crude multiplicity function.  Using
\Prp{dist.cmf.haus}, we see that the \LevyProkhorov distance
$\delta(\alpha_1,\alpha_2)$ defined in \Def{dist.cmf} can be expressed
in terms of the Hausdorff distance of the compact (essential) part and
another distance taking care of the finite part.  In particular, the
quotient map
\begin{equation*}
  (\setOp_{\R,0}, \operatorname d_\circ) \to (\Comp,\delta), \quad
  R \mapsto \alpha_R,
\end{equation*}
mapping operators to their crude multiplicity functions is an isometry
for $\operatorname d_\circ \in \{\dIso, \dUni, \dCmf\}$, and
bi-lipschitz-continuous for $\operatorname d=\dQUE$.  Note that the
space $(\Comp,\delta)$ is a Hausdorff space.  We have hence
generalised~\cite[Lem.~2.5]{post-simmer:20} to general bounded
self-adjoint operators with $0$ in their essential spectrum.  We will
discuss related questions in a forthcoming publication.

%----------------------------------------------------------------------

% yyyy
%
% Bibliography
%
%----------------------------------------------------------------------

% Olaf's database
%\bibliography{/home/post/Seafile/Wolke/BibTeX/literatur}
%\bibliographystyle{/home/post/Seafile/Wolke/BibTeX/my-amsalpha}

\def\cprime{$'$}
\providecommand{\bysame}{\leavevmode\hbox to3em{\hrulefill}\thinspace}
\providecommand{\MR}{\relax\ifhmode\unskip\space\fi MR }
% \MRhref is called by the amsart/book/proc definition of \MR.
\providecommand{\MRhref}[2]{%
  \href{http://www.ams.org/mathscinet-getitem?mr=#1}{#2}
}
\providecommand{\href}[2]{#2}

%\tableofcontents

\end{document}